\documentclass[11pt]{amsart}

\usepackage{epigamath}

\usepackage[english]{babel}

\numberwithin{equation}{section}

\usepackage[shortlabels]{enumitem}
\usepackage{macrosE719}
\usepackage{extarrows}

\setlist[enumerate,1]{label={\rm(\roman*)}, ref={\rm\roman*}}

\DeclareMathOperator{\Sets}{Sets}
\DeclareMathOperator{\Fix}{Fix}
\DeclareMathOperator{\proet}{proet}
\DeclareMathOperator{\fppf}{fppf}
\DeclareMathOperator{\Sc}{sc}
\DeclareMathOperator{\spl}{spl}
\DeclareMathOperator{\bas}{bas}
\DeclareMathOperator{\Dyn}{Dyn}
\DeclareMathOperator{\gr}{gr}
\DeclareMathOperator{\Int}{Int}

\newcommand{\longhar}{\lhook\joinrel\longrightarrow}

\newcommand*{\MyDef}{\,\mathrm{def}\,}
\newcommand*{\eqdef}{\mathop{\overset{\MyDef}{\scalebox{2.2}[1]{=}}}}

\def\isolow{\vbox to 0pt{\vss\hbox{$\scriptstyle\sim$}\vskip-2.4pt}}
\newcommand{\isor}{\,\xrightarrow{\isolow}\,}
\newcommand{\isorlong}{\,\xlongrightarrow{\isolow}\,}

\EpigaVolumeYear{7}{2023} \EpigaArticleNr{19} \ReceivedOn{October 8, 2021}
\InFinalFormOn{April 23, 2023}
\AcceptedOn{June 9, 2023}

\title{On a decomposition of $\boldsymbol{p}$-adic Coxeter orbits}
\titlemark{On a decomposition of $p$-adic Coxeter orbits}

\author{Alexander B.~Ivanov}
\address{Mathematisches Institut, Universit\"at Bonn, Endenicher Allee 60, 53115 Bonn, Germany}
\address{(new address) Ruhr-Universit\"at Bochum,  Fakult\"at f\"ur Mathematik, Universit\"atsstrasse 150, 44780 Bochum, Germany}
\email{a.ivanov@rub.de}

\authormark{A.\,B.~Ivanov}

\AbstractInEnglish{We analyze the geometry of some $p$-adic Deligne--Lusztig spaces $X_w(b)$, introduced by the author in a paper in Crelle,  attached to an unramified reductive group $\bfG$ over a non-archimedean local field. We prove that when $\bfG$ is classical, $b$ basic and $w$ Coxeter, $X_w(b)$ decomposes as a disjoint union of translates of a certain integral $p$-adic Deligne--Lusztig space. Along the way we extend some observations of DeBacker and Reeder on rational conjugacy classes of unramified tori to the case of extended pure inner forms and prove a loop version of the Frobenius-twisted Steinberg cross section.}


\acknowledgement{The author acknowledges support by the DFG via the Leibniz prize of Peter Scholze.}

\begin{document}



\maketitle

\begin{prelims}

\DisplayAbstractInEnglish

\bigskip

\DisplayKeyWords

\medskip

\DisplayMSCclass







\end{prelims}


\newpage

\setcounter{tocdepth}{1}

\tableofcontents


\section{Introduction}\label{sec:intro}

Let $\bfG$ be an unramified reductive group over a non-archimedean local field $k$. In \cite{Ivanov_DL_indrep} $p$-adic Deligne--Lusztig spaces $X_w(b)$ attached to $\bfG$ were constructed. This allows one to apply the rich variety of methods from Deligne--Lusztig theory \cite{DeligneL_76} to the study of representations of the $p$-adic group $\bfG(k)$. Indeed, many results in this direction were obtained recently; \textit{cf.} \cite{CI_loopGLn, ChenS_17, ChanO_21, DudasI_20} and references therein.


The classical Deligne--Lusztig varieties are smooth schemes of finite type over a finite field. In contrast, the $X_w(b)$ are sheaves for the arc-topology, \textit{cf.} \cite{BhattM_18}, on perfect algebras over the residue field of $k$. Concerning their representability, the situation is slightly delicate: In many cases it is known that the $X_w(b)$ are ind-representable; \textit{cf.} \cite[Theorem~C]{Ivanov_DL_indrep}.
Also, there are examples where $X_w(b)$ is known to be an ind-scheme which is not a scheme; \textit{cf.} \cite[Corollary~10.2]{Ivanov_DL_indrep}. On the other hand, for $\bfG = {\bf GL}_n$, $c$ Coxeter and $b$ basic, $X_c(b)$ is a scheme; \textit{cf.} \cite[Proposition~2.6]{CI_loopGLn}. Moreover, the proof of this last fact has led to a quite explicit description of (a cover of) $X_c(b)$ in this case, eventually allowing one to compute large parts of its cohomology and relate it to the local Langlands correspondences, which was the main task of \cite{CI_loopGLn, CI_DrinfeldStrat}.

Motivated by the last sentence, we analyze here the geometry of the spaces $X_c(b)$ and their natural covers $\dot X_{\bar c}(b)$ when $\bfG$ is a classical group,\footnote{We call an unramified $k$-group \emph{classical} if each connected component of its Dynkin diagram is of type $A_{n-1}, B_m,C_m,D_m,{}^2A_{n-1}$ or ${}^2D_m$; see Definition~\ref{def:classical_type} for more details.} $c$ is a (twisted) Coxeter element and $b$ is basic. Under these assumptions our main result, Theorem~\ref{thm:main_decomposition}, shows that $X_c(b)$ and $\dot X_{\bar c}(b)$ admit natural decompositions as disjoint unions of translates of certain  integral level Deligne--Lusztig spaces of Coxeter type. Those integral level spaces are easier to handle. In particular, they are affine schemes, and it follows that $X_c(b)$ and $\dot X_{\bar c}(b)$ are schemes, proving \cite[Conjecture~1.1]{Ivanov_DL_indrep} in this case. Our decomposition result, along with the Mackey-type formula shown in \cite{DudasI_20} (as well as techniques from \cite{Lusztig_76_inv,ChanO_21}), can be applied---as
in \cite{CI_loopGLn}---to study the cohomology of $X_c(b)$ and $\dot X_{\bar c}(b)$ and smooth $\bfG(k)$-representations therein.  
Beside our main result, we prove some related results on rational conjugacy classes of tori and a version of Steinberg's cross section; see Section~\ref{sec:intro_further_topics}
below for a summary. Now, we introduce the relevant notation and state our main result, Theorem~\ref{thm:main_decomposition}.

\subsection{Setup}\label{sec:nota} Let $\breve k$ denote the completion of a maximal unramified extension of $k$, and let $\caO_k, \caO_{\breve k}$ be the rings of integers of $k, \breve k$. Let $\sigma$ be the Frobenius of $\breve k/k$. Let $\bfT \subseteq \bfB \subseteq \bfG$ be a centralizer of a maximal split $k$-torus and a $k$-rational Borel subgroup of the unramified group $\bfG$. (Note that $\bfT$ is a $k$-rational maximal torus of $\bfG$.) Let $\bfU$, resp.\ $\bfU^-$, be the unipotent radical of $\bfB$, resp.\ the opposite Borel subgroup. Let $W$ be the Weyl group of $\bfT$ and $S \subseteq W$ the set of simple reflections determined by $\bfB$. Denote by $\caB_{\bfG,\breve k}$ the Bruhat--Tits building of the adjoint group $\bfG^{\ad}$ of $\bfG$ over $\breve k$ and by $\caA_{\bfT,\breve k}$ the apartment of $\bfT$ therein. Denote by $\sigma$ the action induced by Frobenius on $W,S,\caB_{\bfG,\breve k}, \caA_{\bfT,\breve k}$ and $\bfG(\breve k)$ (as well as its subgroups).

An element $b \in \bfG(\breve k)$ is called \emph{basic} if its Newton point $\nu_b$ factors through the center of $\bfG$; \textit{cf.}~\cite{Kottwitz_85}. Twisting $\sigma$ by $\Ad(b)$, one gets an inner $k$-form $\bfG_b$ of $\bfG$, see Section~\ref{sec:twisting_Frobenius}. Similarly, for $w \in W$, twisting the Frobenius on $\bfT(\breve k)$ by $\Ad(w)$, one gets an (outer) $k$-form $\bfT_w$ of $\bfT$.

The space $X_w(b)$, introduced in \cite{Ivanov_DL_indrep}, is the intersection of $L\caO(w) \subseteq L(\bfG/\bfB)^2$ with the graph of $b\sigma \colon L(\bfG/\bfB) \rar L(\bfG/\bfB)$, where $L$ denotes the loop functor (see Section~\ref{sec:loops}). Similarly, we have its coverings $\dot X_{\bar w}(b)$, where $\bar w$ lies in $\overline F_w\,/\ker\bar\kappa_w$, a certain discrete quotient of the set $F_w(\breve k) \subseteq \bfG(\breve k)$ of all lifts of $w$; \textit{cf.}~Sections~\ref{sec:fibers_LFw_Fw} and~\ref{sec:stuff_on_torsors}. For more details on $X_w(b)$ and $\dot X_{\bar w}(b)$, see Section~\ref{sec:review_some_pDL} and Definition~\ref{conv:dotXwb_indep}. The group $\bfG_b(k)$ acts on $X_w(b)$, and $\bfG_b(k)\times \bfT_w(k)$ acts on $\dot X_{\bar w}(b)$.  The natural map $\dot X_{\bar w}(b) \rar X_w(b)$ is $\bfG_b(k)$-equivariant.

An element $c \in W$ is called \emph{twisted Coxeter}, or simply \emph{Coxeter}, if any of its reduced expressions contains precisely one element from each $\sigma$-orbit on $S$; \textit{cf.} \cite[Section~7]{Springer_74}. Usually, we denote by $w$ an arbitrary element of $W$, whereas the letter $c$ is reserved for Coxeter elements.

\smallskip

We refer to Section~\ref{sec:notation_and_setup} for more notation and setup.

\subsection{Main result}\label{sec:intro_main} Let $c \in W$ be Coxeter, $\bar c \in \overline F_c \,/ \ker\bar\kappa_c$ a lift of $c$ and $b \in \bfG(\breve k)$ basic. We prove a decomposition result for $X_c(b)$ and $\dot X_{\bar c}(b)$. The space $X_c(b)$ is always non-empty, and---after an equivariant isomorphism---one can assume that $c$ is a special Coxeter element (see Definition~\ref{def:special_Coxeter}) and $b$ is a lift of $c$. Moreover, in Proposition~\ref{prop:simplifying_assumptions}, we will see that either $\dot X_{\bar c}(b) = \varnothing$, or---after an equivariant isomorphism---we may even assume that $b$ lies over $\bar c$. In the rest of Section~\ref{sec:intro_main} we work under these simplifying assumptions (which do not result in any loss of generality).

As $b$ lies over $c$, the affine transformation $b\sigma$ of $\caA_{\bfT,\breve k}$ has precisely one fixed point $\bfx$ (this follows from \cite[Lemma~7.4]{Springer_74} as $c$ is Coxeter). 
Let $\caG_\bfx$ be the corresponding  parahoric $\caO_{\breve k}$-model of $\bfG$. It descends to an $\caO_k$-model $\caG_{\bfx,b}$ of $\bfG_b$. In particular, the Frobenius $\sigma_b:=   \Ad(b) \circ \sigma$ of $L\bfG$ fixes the subgroup $L^+\caG_{\bfx}$. Here $L$ and $L^+$ denote the functors of (positive) loops; see Section~\ref{sec:loops}. We also have the corresponding parahoric model $\caG_{\bfx}^{\ad}$ of $\bfG^{\ad}$. Consider the scheme $\dot X_{\bar c,b}^{\caG_{\bfx}}$ defined by the cartesian diagram
\begin{equation}\label{eq:integral_Coxeter_Stcsss_space}
\xymatrix{
\dot X_{\bar c,b}^{\caG_{\bfx}} \ar[r] \ar[d] & L^+({}^c\bfU \cap \bfU^-)_{\bfx}  \ar[d] \\ 
L^+\caG_{\bfx} \ar[r] & L^+\caG_{\bfx}\rlap{,}
}
\end{equation}
where the lower map is $g \mapsto g^{-1}\sigma_b(g)$ and where $({}^c\bfU \cap \bfU^-)_{\bfx}$ is the closure of ${}^c\bfU \cap \bfU^-$ in $\caG_{\bfx}$. We also consider the quotient for the pro-\'etale (equivalently, arc-)topology,
\begin{equation}\label{eq:integral_Coxeter_Stcsss_quot_space}
X_{c,b}^{\caG^{\ad}_{\bfx}} = \dot X_{\bar c,b}^{\caG^{\ad}_{\bfx}} / \bfT_c^{\ad}(k),
\end{equation}
where $\dot X_{\bar c,b}^{\caG^{\ad}_{\bfx}}$ is \eqref{eq:integral_Coxeter_Stcsss_space} applied to the parahoric model $\caG_{\bfx}^{\ad}$ of the adjoint group $\bfG^{\ad}$. Then $X_{c,b}^{\caG^{\ad}_{\bfx}}$ and $\dot X_{\bar c,b}^{\caG_{\bfx}}$ are infinite-dimensional affine schemes; \textit{cf.}~Proposition~\ref{representability_of_integral_DLVs}.

\begin{thm}\label{thm:main_decomposition}
Let $\bfG$ be an unramified group of classical type. Let $c,\bar c$ and $b$ be as above. Then there exist a $\bfG_b(k)$-equivariant isomorphism
\begin{equation}\label{eq:main_iso_1}
X_c(b) = \coprod_{\gamma \in \bfG^{\ad}_b(k)/\caG^{\ad}_{\bfx, b}(\caO_k)} \gamma X_{c,b}^{\caG^{\ad}_{\bfx}}
\end{equation}
and a $\bfG_b(k) \times \bfT_c(k)$-equivariant isomorphism
\begin{equation}\label{eq:main_iso_2}
\dot X_{\bar c}(b) \cong \coprod_{\gamma \in \bfG_b(k)/\caG_{\bfx,b}(\caO_k)} \gamma \dot X_{c,b}^{\caG_{\bfx}}.
\end{equation}
\end{thm}

After necessary preparations, Theorem~\ref{thm:main_decomposition} will be proven in Sections~\ref{sec:pDL_on_Coxeter_case} and~\ref{sec:type_by_type}. Let us outline the proof. By a $v$-descent argument, \eqref{eq:main_iso_1} follows from a particular case of \eqref{eq:main_iso_2} for absolutely almost simple adjoint groups. Showing \eqref{eq:main_iso_2} in this case presents the main technical difficulty, which is solved in Section~\ref{sec:type_by_type} by using a special property of the Newton polygons of isocrystals and analyzing each irreducible Dynkin type separately (here the assumption that $c$ is Coxeter becomes crucial). Finally, \eqref{eq:main_iso_2} can be deduced in general from \eqref{eq:main_iso_1}. We expect Theorem~\ref{thm:main_decomposition} to hold for all unramified reductive groups $\bfG$.


At this point, let us stress that $p$-adic Deligne--Lusztig spaces are quite different in nature from affine Deligne--Lusztig varieties (say, at Iwahori level). The latter were  introduced by Rapoport in \cite{Rapoport_05}. They are important because their unions appear as special fibers of Rapoport--Zink spaces. Note, for example, that whereas affine Deligne--Lusztig varieties are (perfect) schemes locally of finite type over $\overline\bF_q$, the $p$-adic Deligne--Lusztig spaces are essentially never of finite type over $\overline\bF_q$. On the other hand, there is, remarkably, some relation between these two types of objects. In \cite{CI_ADLV} it was shown that the inverse limit of a certain system of affine Deligne--Lusztig varieties (of increasing level) for the group ${\bf GL}_n$ is isomorphic to a $p$-adic Deligne--Lusztig space $\dot X_{\bar c}(b)$ for the same group and a Coxeter element $c$. We expect this result to generalize to other groups, especially in light of the similarity of the decomposition of Coxeter-type spaces $X_c(b)$ provided by Theorem~\ref{thm:main_decomposition} and a very similar-looking decomposition result of He--Nie--Yu \cite[Theorem~1.1(3)]{HeNY_22} for affine Deligne--Lusztig varieties attached to elements of the affine Weyl group with finite Coxeter part.


Let us also note that Theorem~\ref{thm:main_decomposition} implies \cite[Conjecture~1.1]{Ivanov_DL_indrep} in the following cases.

\begin{cor}\label{cor:schemes_they_are}
If\, $\bfG$ is an unramified group of classical type, $c \in W$ Coxeter, $\bar c \in \overline F_c\,/\ker\bar\kappa_c$ arbitrary and $b$ basic, then $X_c(b)$ and $\dot X_{\bar c}(b)$ are disjoint unions of affine schemes.
\end{cor}

\subsection{Quasi-split case}\label{sec:quasisplit_statement} If $b$ is $\sigma$-conjugate to $1$, Theorem~\ref{thm:main_decomposition} admits a clearer reformulation. Let $c \in W$ be Coxeter and $\bar c \in \overline F_c\,/\ker\bar\kappa_c$. If $\dot X_{\bar c}(1) \neq \varnothing$, then a (\textit{i.e.}, any) lift $\dot c \in \bfG(\breve k)$ of $\bar c$ is $\sigma$-conjugate to $1$ (by Corollary~\ref{cor:choice_b_cox_over_barw}). Moreover, choosing $\dot c$ appropriately, we may additionally assume that the (unique and automatically hyperspecial) fixed point $\bfx = \bfx_{\dot c}$ of $\dot c \sigma$ on $\caA_{\bfT,\breve k}$ lies in $\caA_{\bfT,k} = \caA_{\bfT,\breve k}^{\langle \sigma \rangle}$ (\textit{cf.}~Section~\ref{sec:quasisplit_proof}). In this situation we have the integral level analogues $X_c^{\caG_{\bfx}}(1)$ and $\dot X_{\dot c}^{\caG_{\bfx}}(1)$ of $X_c(1)$ and $\dot X_{\bar c}(1)$, which were introduced in \cite[Definition~4.1.1]{DudasI_20}; see Section~\ref{sec:review_some_pDL}. In this situation we have the following corollary, proven in Section~\ref{sec:quasisplit_proof}.


\begin{cor}\label{cor:main_decomposition_b1_case}
With notation as above, there are $\bfG(k)$- $($resp.\ $\bfG(k) \times \bfT_c(k)$-$)$equivariant isomorphisms
\[
X_c(1) \cong \coprod_{\gamma \in \bfG^{\ad}(k)/\caG^{\ad}_{\bfx,1}(\caO_k)} \gamma X_c^{\caG_\bfx}(1) \quad \text{ and } \quad \dot X_{\bar c}(1) \cong \coprod_{\gamma \in \bfG(k)/\caG_{\bfx,1}(\caO_k)} \gamma \dot X_{\dot c}^{\caG_\bfx}(1). 
\]
\end{cor}

\subsection{Two related results}\label{sec:intro_further_topics} In the rest of the introduction we briefly discuss two further themes of this article. First, we review in Section~\ref{sec:conj_classes_tori}, following DeBacker and Reeder \cite{deBacker_06,deBackerR_09,Reeder_11}, the parametrization of rational conjugacy classes of unramified maximal $k$-tori in pure inner forms of $\bfG$. Moreover, in Sections~\ref{sec:rationa_conj_classes_ext_pure_forms}--\ref{sec:fibers_LFw_Fw} we extend their discussion to all extended pure inner forms and give a description of rational conjugacy classes in terms of $\overline F_w\,/\ker\bar\kappa_w$, the discrete parameter space for the spaces $\dot X_{\bar w}(b)$ over $X_w(b)$. In Section~\ref{sec:Coxeter_tori_case} we  investigate in detail the Coxeter case , where (under some mild assumptions on $\bfG$) it turns out that the set of different covers $\dot X_{\bar c}(b)$ of $X_c(b)$ for $b$ basic is in canonical bijection with the set of rational conjugacy classes of unramified Coxeter tori in the inner form $\bfG_b$ of $\bfG$; see Corollary~\ref{cor:tori_and_covers}. Thus the $p$-adic Deligne--Lusztig space $X_c(b)$ ``can see'' the difference between the various rational conjugacy classes of tori whose stable class is given by $c$.

Secondly, in Section~\ref{sec:variant_steinberg_crosssection} we show a loop version of the twisted Steinberg cross section. Steinberg's original result \cite[Proposition~8.9]{Steinberg_65} states roughly that if $\dot c \in F_c(\breve k)$ is any lift of a Coxeter element $c \in W$, any ${}^c\bfU \cap \bfU$-orbit on ${}^c\bfU$ for the action given by $g(u) = g^{-1}u\Ad(\dot c)(g)$ contains precisely one point in ${}^c\bfU \cap \bfU^-$ ($k$ can be replaced by any field here). In \cite[Section~3.14(c)]{HeL_12} this is generalized in various respects, in particular to the case where $k$ is replaced by a finite field $\bF_q$ and the action is Frobenius-twisted, \textit{i.e.}, given by $g(u) = g^{-1}u\sigma_{\dot c}(g)$, where $\sigma$ is the (geometric) Frobenius over $\bF_q$ and $\sigma_{\dot c} = \Ad(\dot c) \circ \sigma$. This can be used to prove affineness of classical Deligne--Lusztig varieties attached to Coxeter elements (for example, by combining \cite[Corollary~1.12]{DeligneL_76} and \cite[Section~4.3]{HeL_12}). In our setup, with $k$ non-archimedean local, (to perform a reduction step in the proof of Theorem~\ref{thm:main_decomposition}) we need the same result for the action of $L({}^c\bfU \cap \bfU)$ on $L({}^c\bfU)$. A part of the proof in \cite{HeL_12} can be adapted, but at some point in \cite{HeL_12} the Ax--Grothendieck theorem\footnote{It states, for an endomorphism $X \rar X$ of a scheme of finite type over an algebraically closed field $\Omega$, that if $X(\Omega) \rar X(\Omega)$ is injective, then it is also surjective.} is used, which is impossible in our setting (as $L\bfU$ is not even a scheme). Therefore, we carry through a different argument, at least for classical groups and special Coxeter elements. 


%

\subsection*{Outline} In Section~\ref{sec:notation_and_setup} we introduce more notation, prove a fact about integral $p$-adic Deligne--Lusztig spaces and review the Borovoi fundamental group and $\sigma$-conjugacy classes. In Section~\ref{sec:conj_classes_tori} we review (and extend) some results on rational/stable conjugacy classes of unramified maximal tori of $\bfG$. In Section~\ref{sec:stuff_on_torsors} we review and extend some results from \cite{Ivanov_DL_indrep} about the maps $\dot X_{\bar w}(b) \rar X_w(b)$. In Section~\ref{sec:variant_steinberg_crosssection} we prove loop (and integral loop) versions of the twisted Steinberg cross section. This is used in Section~\ref{sec:pDL_on_Coxeter_case} to replace $\dot X_{\bar c}(b)$ (and its integral analogue) by a more tractable object. The main part of the article consists of Sections~\ref{sec:pDL_on_Coxeter_case} and~\ref{sec:type_by_type}, where Theorem~\ref{thm:main_decomposition} is proven.

\subsection*{Acknowledgements} I want to thank Johannes Ansch\"utz, Charlotte Chan, Lucas Mann and Peter Scholze for several discussions related to the present work. Especially, I want to thank Charlotte Chan, with whom I initially started to work on Lusztig’s conjecture, and who, in particular, pointed out the reference \cite{HeL_12} to me. I thank Michael Borovoi for pointing out the fact used in the beginning of Section~\ref{sec:proof_main_cover_general_case} to me  and an anonymous referee for making various helpful suggestions.

\section{Setup and preliminaries} \label{sec:notation_and_setup}

Recall that $\bfG$ is an unramified group over the local field $k$. We use the notation from Section~\ref{sec:nota}.

\subsection{More notation}\label{sec:further_notation} Given a group $G$, $g \in G$ and any subset $H \subseteq G$, we write ${}^gH = gHg^{-1}$. For an abelian group $A$, we write $A_{\tors}$ for its torsion subgroup.

Let $p$ be the residue characteristic of $k$. For a ring $R$ of characteristic $p$, denote by $\Perf_R$ the category of perfect algebras over $R$. Denote by $\varpi$ a uniformizer of $k$ and by $\bF_q$ (resp.\ $\obF$) the residue field of $k$ (resp.\ $\breve k$). 

For $R$ in $\Perf_{\bF_q}$, let $\bW(R)$ be the unique (up to unique isomorphism) $\varpi$-adically complete and separated $\caO_k$-algebra in which $\varpi$ is not a zero-divisor and which satisfies $\bW(R)\,/ \varpi\bW(R) = R$. In other words, if $W(R)$ denotes the $p$-typical Witt-vectors of $k$, then $\bW(R) = W(R) \otimes_{W(\bF_q)} \caO_k$ if $\charac k = 0$ and $\bW(R) = R[\![\varpi]\!]$ if $\charac k = p$. Denote by $\phi = \phi_R$ the Frobenius automorphisms of $\bW(R)$, $\bW(R)[1/\varpi]$, induced by $x \mapsto x^q$ on~$R$ (this notation is only used in the proof of Proposition~\ref{prop:loop_Steinbergs_twisted_crosssection} and  Section~\ref{sec:type_by_type}). We have $\bW(\bF_q)[1/\varpi] = k$ and $\bW(\obF)[1/\varpi] = \breve k$ (so that $\phi_{\obF} = \sigma$). 

Denote by $k^{\alg}$ a fixed algebraic closure of $k$ containing the maximal algebraic subextension of $\breve k/k$. Denote by $\langle \sigma \rangle \cong \hat\bZ$ the group of continuous automorphisms of $\breve k/k$.

\subsubsection{Loop functors}\label{sec:loops}
For a $\breve k$-scheme $X$, let 
\[
LX \colon \Perf_{\obF} \longrar \Sets, \quad R \longmapsto X\left(\bW(R)\left[1/\varpi^{-1}\right]\right)
\] 
be the loop space of $X$. If $X$ is affine of finite type over $\breve k$, then $LX$ is representable by an ind-scheme. For an $\caO_{\breve k}$-scheme $\caX$, let 
\[
L^+\caX \colon \Perf_{\obF} \longrar \Sets, \quad R \longmapsto \caX(\bW(R))
\]
be the space of positive loops of $\caX$. Moreover, if $n \geq 1$, let
\[
L^+_n\caX \colon \Perf_{\obF} \longrar \Sets,\quad R \longmapsto \caX(\bW(R)/\varpi^n\bW(R))
\]
be the truncated version of $L^+\caX$. By \cite[Proposition-Definition~6.2]{BertapelleG_18}, $L^+_n\caX$ is a scheme (resp.\ affine scheme) for arbitrary (resp.\ affine) $\caX$, and if $\caX$ is of finite type, then $L^+\caX$ is too. Moreover, $L^+\caX = \prolim_n L^+_n\caX$, along affine transition maps (\textit{cf.} \cite[Proposition~8.5]{BertapelleG_18}). Thus, if $\caX$ is a scheme (resp.\ affine scheme), then $L^+\caX$ is too. 

If $X$ is the base change to $\breve k$ of a $k$-scheme $X_0$, then the functor  $LX$ carries a natural automorphism, the geometric Frobenius, denoted by $\sigma_{X_0}$ (or simply by $\sigma$ if $X_0$ is clear from context). Similarly, we have the geometric Frobenius $\sigma = \sigma_{\caX_0}$ on $L^+\caX$ if the $\caO_{\breve k}$-scheme $\caX$ is the base change of an $\caO_k$-scheme $\caX_0$.

\subsubsection{Groups}\label{sec:notation_groups} 
We write $\bfG^{\ad}$, resp.\ $\bfG^{\sssc}$, for the adjoint group, resp.\ the simply connected cover of the derived group, of $\bfG$. If $\bfH \subseteq \bfG$ is a closed subgroup of $\bfG$, denote by $\bfH^{\ad}$, resp.\ $\bfH^{\sssc}$, its image in $\bfG^{\ad}$, resp.\ its preimage in $\bfG^{\sssc}$. We identify the Weyl groups of $\bfG^{\ad}$ and $\bfG^{\sssc}$ with $W$. We denote by $\bfZ$ the center of $\bfG$.

Write $X_\ast(\bfT)$, resp.\ $X^\ast(\bfT)$, for the free $\bZ$-module of cocharacters, resp.\ characters, of $\bfT$. Denote by $\Phi \subseteq X^\ast(\bfT)$, resp.\ $\Phi^\vee \subseteq X_\ast(\bfT)$, the set of roots, resp.\ coroots, of $\bfG$; denote by $\Phi^+$, resp.\ $\Delta$, the set of positive, resp.\ simple, roots, determined by $\bfB$; denote by $s_{\alpha} \in W$ the reflection attached to $\alpha \in \Phi$. For $\xi \in X_\ast(\bfT)$, denote by $\varpi^\xi$ the image in $\bfT(\breve k)$ of $\varpi$ under $\xi$.

Write $\bfN$ for the normalizer of $\bfT$. We have the reductive $\caO_k$-model $\caT$ of $\bfT$, and $\caT(\caO_{\breve k})$ is the maximal compact subgroup of $\bfT(\breve k)$. The quotient $\widetilde W = \bfN(\breve k)/\caT(\caO_{\breve k})$ is the \emph{extended affine Weyl group} of $\bfT$ in $\bfG$. It is naturally an extension $\widetilde W = W \ltimes X_\ast(\bfT)$,
whose splittings (resp.\ $\sigma$-equivariant splittings) $W \rar \widetilde W$ correspond to hyperspecial points of $\caA_{\bfG,\breve k}$ (resp.\ of $\caA_{\bfG,\breve k}^{\langle \sigma \rangle}$). 

\subsubsection{Twisting Frobenius}\label{sec:twisting_Frobenius} For $b \in \bfG(\breve k)$, we have the $k$-group $\bfG_b$, which is defined as the functorial $\sigma$-centralizer of $b$; \textit{cf.}~\cite[1.12]{RapoportZ_96} or \cite[3.3, Appendix~A]{Kottwitz_97}. Then $\bfG_b(k) = \{g \in \bfG(\breve k) \colon g^{-1}b\sigma(g) = b \}$. If $b$ is basic (\textit{cf.}~Section~\ref{sec:newton_and_sigma_classes}), then $\bfG_b$ is an inner form of $\bfG$, we may identify $\bfG_b(\breve k) = \bfG(\breve k)$ and we denote by $\sigma_b := \Ad(b) \circ \sigma$ the twisted Frobenius on $\bfG(\breve k)$. We also write $\sigma_b$ for the corresponding geometric Frobenius on $L\bfG_b = L\bfG$.

Suppose that the basic element $b$ lies in $\bfN(\breve k)$ and that $\bfx \in \caA_{\bfT,\breve k}$ is fixed by the action of $b\sigma$. Let $\caG_{\bfx}$ be the corresponding parahoric model of $\bfG$ over $\caO_{\breve k}$. Then $\sigma_b$ fixes $\caG_{\bfx}(\caO_{\breve k}) \subseteq \bfG(\breve k)$, and $L^+\caG_{\bfx} \subseteq L\bfG$. Thus $\caG_{\bfx}$ descends to an $\caO_k$-model of $\bfG_b$, which we denote by $\caG_{\bfx,b}$. We have $L^+\caG_{\bfx,b} \otimes_{\bF_q} \obF = L^+\caG_{\bfx}$.

Similarly, for $w \in W$, we have the (outer) $\breve k/k$-form $\bfT_w$ of $\bfT$, satisfying $\bfT_w(k) = \{t \in \bfT(\breve k) \colon \Ad\,w(\sigma(t)) = t\}$. We may identify $\bfT_w(\breve k) = \bfT(\breve k)$, and we denote by $\sigma_w := \Ad(w)\circ \sigma$ the twisted Frobenius on this group, as well as on $L\bfT$.

\subsubsection{Topologies} We work with two topologies on $\Perf_{\bF_p}$: the pro-\'etale topology, \textit{cf.} \cite{BhattS_15}, and the arc-topology, \textit{cf.}  \cite{BhattM_18}. Recall that a map $S' \rar S$ of qcqs schemes is an arc-cover if any immediate specialization in $S$ lifts to $S'$. By \cite[Theorem~5.16]{BhattM_18}, the arc-topology on perfect qcqs $\bF_p$-schemes coincides with the canonical topology. Thus any perfect qcqs scheme is an arc-sheaf. In particular, for any qcqs scheme $\caX/\caO_k$, $L^+\caX$ is an arc-sheaf. For any quasi-projective scheme $X/k$, $LX$ is an arc-sheaf by \cite[Theorem~A]{Ivanov_DL_indrep}. 

\subsubsection{Locally profinite groups}\label{sec:loc_prof_gps} As in \cite[Section~4]{Ivanov_DL_indrep}, we regard a locally profinite group $H$ as a (representable) functor
\[ 
S \longmapsto C^0(|S|, H) 
\]
on $\Perf_{\obF}$. In particular, $\bfG_b(k)$ and $\bfT_w(k)$ are such functors. If $H$ is profinite, this functor is even represented by an affine scheme.

\subsubsection{Galois cohomology}\label{sec:Gal_coh} For the definition and standard facts about non-abelian cohomology, we refer to \cite{Serre_Galois_Cohomology}. If $M$ is a discrete $\langle \sigma \rangle$-group, where $\langle \sigma \rangle$ is as in Section~\ref{sec:further_notation}, we write $H^1(\sigma, M)$ for the corresponding first cohomology set. We identify a continuous $1$-cocycle $a \colon \langle \sigma \rangle \rar M$ with $a(\sigma) \in M$ and write $[a(\sigma)] \in H^1(\sigma, M)$ for the corresponding cohomology class.

We will write $H^1(\sigma,\bfG)$, resp.\ $H^1(\sigma_w,\bfT)$, instead of $H^1(\sigma, \bfG(\breve k))$, resp.\ $H^1(\sigma,\bfT_w(\breve k))$. By a theorem of Steinberg \cite[Theorem~III.1']{Serre_Galois_Cohomology}, resp.\ its extension \cite[Section~8.6]{BorelS_68} if $\charac k >0 $, we have $H^1(\sigma,\bfG) = H^1(k,\bfG)$ and $H^1(\sigma_w,\bfT) = H^1(k,\bfT_w)$ via inflation along $\Gal_k \tar \langle \sigma \rangle$. We will always use the notation $H^1(\sigma,\bfG)$ and $H^1(\sigma_w,\bfT)$ without further comment.

\subsection{Review of some $\boldsymbol{p}$-adic Deligne--Lusztig spaces}\label{sec:review_some_pDL}

\subsubsection{Spaces $\boldsymbol{X_w(b)}$ and $\boldsymbol{\dot X_{\dot w}(b)}$}\label{sec:review_DL_general}
Let $w\in W$, $\dot w \in \bfN(\breve k)$ be a lift of $w$ and $b \in \bfG(\breve k)$. Denote by $\caO(w) \subseteq (\bfG/\bfB)^2$ (resp.\ $\caO(\dot w) \subseteq \bfG/\bfU$) the $\bfG$-orbit corresponding to $w$ (resp.\ $\dot w$) under the Bruhat decomposition.  The \emph{$p$-adic Deligne--Lusztig spaces} $X_w(b)$ and $\dot X_{\dot w}(b)$ were defined in \cite[Definition~8.3]{Ivanov_DL_indrep} by the cartesian diagrams of functors on $\Perf_{\obF}$

\centerline{\begin{tabular}{cc}
\begin{minipage}{2in}
\begin{displaymath}
\leftline{
\xymatrix{
X_w(b) \ar[r] \ar[d] & L\caO(w) \ar[d]\\
L(G/B) \ar[r]^-{(\id,b\sigma)} & L(G/B) \times L(G/B) 
}
}
\end{displaymath}
\end{minipage}
& \qquad\qquad  and \qquad\qquad
\begin{minipage}{2in}
\begin{displaymath}
\leftline{
\xymatrix{
\dot X_{\dot w}(b) \ar[r] \ar[d] & L\dot \caO(\dot w) \ar[d]\\
L(G/U) \ar[r]^-{(\id,b\sigma)} & L(G/U) \times L(G/U)\rlap{,} 
}
}
\end{displaymath}
\end{minipage}
\end{tabular}
}
\noindent where $b\sigma \colon L(G/B) \isor L(G/B)$ and $b$ acts by left multiplication. Then $X_w(b)$ and $\dot X_{\dot w}(b)$ are sheaves for the arc-topology (see \cite{BhattM_18}) on $\Perf_{\obF}$; \textit{cf.} \cite[Corollary~8.4]{Ivanov_DL_indrep}. There is a natural action of $\bfG_b(k)$ on $X_w(b)$ and of $\bfG_b(k) \times \bfT_w(k)$ on $\dot X_w(b)$.

\subsubsection{Spaces $\boldsymbol{\dot X_{\bar c,b}^{\caG_{\bfx}}}$ and $\boldsymbol{X_{c,b}^{\caG_{\bfx}^{\ad}}}$} 
These spaces were introduced in \eqref{eq:integral_Coxeter_Stcsss_space} and \eqref{eq:integral_Coxeter_Stcsss_quot_space}. Here we prove their representability.

\begin{lm}\label{lm:isom_for_replete_topos}
Let $\{H_n\}_{n \geq 1}$ be an inverse system of finite groups, acting compatibly and freely on an inverse system $\{X_n\}_{n\geq 1}$ of pro-\'etale $($resp.\ arc-$)$sheaves on $\Perf_{\obF}$. Write $H = \prolim_{n} H_n$ and $X = \prolim_n X_n$. The natural map $X/H \rar \prolim_n X_n/H_n$ is an isomorphism, where the quotients are formed for the pro-\'etale $($resp.\ arc-$)$topology.  
\end{lm}

\begin{proof}
Consider the case of pro-\'etale topology. The existence of the map is easy. For the injectivity, it is enough to show that the map from the presheaf quotient $R \mapsto X(R)/H(R)$ to $\prolim_n X_n/H_n$ is injective. Let $R \in \Perf_{\obF}$. Suppose $s_1,s_2 \in X(R)$ map to the same element in $(\prolim_n X_n/H_n)(R) = \prolim_n (X_n/H_n)(R)$. Then for each $n$, the images $s_{1,n}$, $s_{2,n}$ of $s_1,s_2$ in $X_n(R)/H_n(R)$ coincide (as any quotient presheaf is automatically separated). Thus for each $n \geq 1$, $S_n = \{h_n \in H_n(R)\colon h_ns_{1,n} = s_{2,n}\} \neq \varnothing$. This set forms an inverse system. By the freeness of the action, we have $\#S_n = 1$, and it follows that $\prolim_n S_n \neq \varnothing$. Then $h \in \prolim_n S_n \subseteq H(R)$ satisfies $hs_1 = s_2$. This proves the injectivity. For the surjectivity, let $s_n \in X_n(R)/H_n(R)$ be a compatible collection of elements. For each $n$, find some pro-\'etale cover $R\rar R_n$ such that $s_n|_{R_n}$ lifts to $\tilde s_n \in X_n(R_n)$. As $R \rar R'=\dirlim_n R_n$ is still a pro-\'etale cover, \textit{cf.} \cite[proof of Lemma~4.1.8]{BhattS_15}, we may replace each $R_n$ with $R'$ and $\tilde s_n$ with $\tilde s_n|_{R'}$. Now, we construct an element $\tilde s\in X(R')$ mapping to $(s_n)_{n\geq 1}$. Starting with $\tilde s_1$ if $n=1$, and working by induction on $n$, suppose that $(\tilde s_i)_{i=1}^n$ is compatible. As the system $(s_n)_{n\geq 1}$ is compatible, we can find some $h_{n+1} \in H_{n+1}(R')$ such that $h_{n+1}\tilde s_{n+1}$ maps to $\tilde s_n$ under $X_{n+1}(R') \rar X_n(R')$. Then replace $\tilde s_{n+1}$ by $h_{n+1}\tilde s_{n+1}$, so that $(\tilde s_i)_{i=1}^{n+1}$ becomes compatible. We conclude by induction. The argument for the arc-topology is the same, using \cite[Corollary~2.18]{BhattM_18} instead of \cite[Lemma~4.1.8]{BhattS_15}
\end{proof}

For $n \geq 1$, let $\dot X_{\bar c,b,n}^{\caG_{\bfx}}$ be defined by the diagram \eqref{eq:integral_Coxeter_Stcsss_space} with $L^+$  replaced by $L^+_n$ everywhere. Then all $\dot X_{\bar c,b,n}^{\caG_{\bfx}}$ are affine schemes and $\dot X_{\bar c,b}^{\caG_{\bfx}} = \prolim_n \dot X_{\bar c,b,n}^{\caG_{\bfx}}$. We also have the truncated version of \eqref{eq:integral_Coxeter_Stcsss_quot_space}: let $\caT$ be the closure of $\bfT$ in $\caG_{\bfx}$ and $\caT_c$ the $\caO_k$-torus, arising by twisting the Frobenius on $\caT$ by $c$. Then put $X_{c,b,n}^{\caG_{\bfx}^{\ad}} := \dot X_{\bar c,b,n}^{\caG_{\bfx}^{\ad}}/T_n$, where $T_n := \caT_c^{\ad}(\caO_k/\varpi^n)$.  

\begin{prop}\label{representability_of_integral_DLVs}
The schemes $\dot X_{\bar c,b}^{\caG_{\bfx}}$ and $X_{c,b}^{\caG_{\bfx}^{\ad}}$ are affine and perfect. Also, $X_{c,b}^{\caG_{\bfx}^{\ad}} = \dot X_{\bar c,b}^{\caG_{\bfx}^{\ad}}/\bfT_c(k)$ remains unchanged if one takes the quotient with respect to the arc-topology instead of the pro-\'etale topology. Moreover, $X_{c,b}^{\caG_{\bfx}^{\ad}} = \prolim_n X_{c,b,n}^{\caG_{\bfx}^{\ad}}$.
\end{prop}
\begin{proof}
The claim for $\dot X_{\bar c,b}^{\caG_{\bfx}}$ follows from the definition and the affineness of $L^+\caX$ for an affine scheme $\caX/\caO_{\breve k}$; \textit{cf.}~Section~\ref{sec:loops}. Suppose that $\bfG$ is adjoint. For $n \geq 1$, let $T_n = \caT_c^{\ad}(\caO_k/\varpi^n)$ be as above. Note that the finite group $T_n$ acts freely on the affine scheme $\dot X_{\bar c,b,n}^{\caG}$ and we have $X_{c,b}^{\caG} = (\dot X_{\bar c,b}^{\caG}/T)_{\proet} = \prolim_n (\dot X_{\bar c,b,n}^{\caG}/T_n)_{\proet}$ by Lemma~\ref{lm:isom_for_replete_topos}. By \cite[Tag 07S7]{StacksProject}, the fppf-quotient $(\dot X_{\bar c,b,n}^{\caG}/T_n)_{\fppf}$ is a scheme, which is even affine (since $\dot X_{\bar c,b,n}^{\caG}$ is), and $\dot X_{\bar c,b,n}^{\caG} \rar (\dot X_{\bar c,b,n}^{\caG}/T_n)_{\fppf}$ is an fppf-torsor. As $T_n$ is finite and acts freely on $\dot X_{\bar c,b,n}^{\caG}$, this is even an \'etale torsor. From this, it follows that $(\dot X_{\bar c,b,n}^{\caG}/T_n)_{\fppf} = (\dot X_{\bar c,b,n}^{\caG}/T_n)_{\et}$, the quotient for the \'etale topology. In particular, the latter is an affine perfect scheme, hence already a pro-\'etale and even an arc-sheaf. This implies that the sheafification of the quotient for the arc- and pro-\'etale topologies agree, as well as the fact that $X_{c,b}^{\caG}$, being the inverse limit of affine perfect schemes, is itself one. 
\end{proof}

\begin{rem}
The map $g \mapsto g^{-1}\sigma_b(g) \colon L^+\caG_{\bf x} \rar L^+\caG_{\bf x}$ is pro-\'etale and---by Lang's theorem applied to all truncations---also surjective. Thus $\dot X_{\bar c,b}^{\caG_{\bfx}}$ and $X_{c,b}^{\caG_{\bfx}^{\ad}}$ and their truncations are always non-empty. Also, it follows that $\dot X_{\bar c,b}^{\caG_{\bfx}}$ and $X_{c,b}^{\caG_{\bfx}^{\ad}}$ are infinite-dimensional. The truncations are perfections of finite-dimensional schemes. More precisely, as the Lang map is \'etale, we have
\[ \dim \dot X_{\bar c,b,n}^{\caG_{\bfx}} = \dim X_{c,b,n}^{\caG_{\bfx}^{\ad}} = \dim L^+_n({}^c\bfU \cap \bfU^-)_{\bf x} = \ell(c)n,
\]
where $\ell(c)$ denotes the length of the Coxeter element $c \in W$.
\end{rem}

\begin{conj}
The scheme $X_{c,b}^{\caG_{\bfx}^{\ad}}$ is an $\bA^\infty$-bundle over a classical Deligne--Lusztig variety attached to the reductive quotient of the special fiber of $\caG_{\bfx}$.
\end{conj}

This is easily verified by hand for $\bfG = {\bf GL}_n$, using the explicit description in \cite{CI_ADLV}.

\subsubsection{Spaces $\boldsymbol{X_w^{\caG}(1)}$, $\boldsymbol{\dot X_{\dot w}^{\caG}(1)}$} 
These spaces only appear in Sections~\ref{sec:quasisplit_statement} and~\ref{sec:quasisplit_proof}. Suppose $\bfx \in \caA_{\bfT,\breve k}^{\langle\sigma\rangle}$ is a hyperspecial point, and let $\caG = \caG_{\bfx}$ be the corresponding reductive $\caO_{\breve k}$-model of $\bfG$. The reductive $\caO_k$-group $\caG_{\bfx, 1}$ is quasi-split, and the closure $\bfB_{\bfx}$ of $\bfB$ in $\caG$ is an $\caO_k$-rational Borel subgroup. Moreover, $\caG$ has a Bruhat decomposition, in the sense that $(\caG/\bfB_{\bfx})^2 \otimes_{\caO_k} \caO_{\breve k}$ contains locally closed subschemes $\caO(w)$ ($w \in W$), flat over $\caO_{\breve k}$, which fiberwise induce the usual Bruhat decomposition. If $\bfN_{\bfx}, \bfU_{\bfx}$ are the closures of $\bfN,\bfU$ in $\caG$, one similarly has $\caO(\dot w) \subseteq (\caG/\bfU_{\bfx})^2 \otimes_{\caO_k} \caO_{\breve k}$ ($\dot w \in \bfN_{\bfx}(\caO_{\breve k})$). Taking $b=1$, $\dot w \in \bfN_{\bfx}(\caO_{\breve k})$ and replacing $\bfG,\bfB,\bfU$ with $\caG,\bfB_{\bfx},\bfU_{\bfx}$ and $L$ with $L^+$ everywhere in the diagrams in Section~\ref{sec:review_DL_general}, we obtain integral $p$-adic Deligne--Lusztig spaces $X_w^{\caG}(1)$ and $\dot X_{\dot w}^{\caG}(1)$, as introduced in \cite[Definition~4.1.1]{DudasI_20}. They are acted on by $\caG_{\bfx,1}(\caO_k)$ and $\caG_{\bfx,1}(\caO_k) \times \caT(\caO_k)$, respectively.

\subsection{Fundamental group and the Kottwitz map}\label{sec:sheaf_Kottwitz_map}

The natural map $\bfT^{\sssc} \rar \bfT$ induces an injection $X_\ast(\bfT^{\sssc}) \rar X_\ast(\bfT)$, and we identify $X_\ast(\bfT^{\sssc})$ with its image. The fundamental group of $\bfG$ in the sense of Borovoi \cite{Borovoi_98} is the group 
\[
\pi_1(\bfG) = X_\ast(\bfT)/X_\ast(\bfT^{\sssc}). 
\] 
It is independent of the choice of $\bfT$ and functorial in $\bfG$, the induced $W$-action on it is trivial, and it admits a Frobenius action, which we denote by $\sigma$. Kottwitz constructed a homomorphism $\tilde\kappa_\bfG \colon \bfG(\breve k) \rar \pi_1(\bfG)$; see \cite[Section~7]{Kottwitz_97} and \cite[Section~2.a.2]{PappasR_08}. Reformulated, a classical result of Kottwitz \cite[Section~6]{Kottwitz_84} states that $\tilde\kappa_\bfG$ induces an isomorphism $H^1(\sigma,\bfG) \cong \pi_1(\bfG)_{\langle\sigma\rangle, {\tors}}$. For tori, we have $\pi_1(-) = X_\ast(-)$. For any $w\in W$, any $k$-rational embedding $\bfT_w \rar \bfG$ induces the map
\[
X_\ast(\bfT)_{\langle \sigma_w \rangle, {\tors}} = H^1(\sigma_w, \bfT) \longrar H^1(\sigma,\bfG) = \pi_1(\bfG)_{\langle\sigma\rangle,{\tors}}, 
\]
which is obtained from $X_\ast(\bfT) \rar \pi_1(\bfG)$ by taking $\sigma_w$- (resp.\ $\sigma$-)invariants and passing to the torsion subgroup. From the definition of $\pi_1(\bfG)$ we deduce an exact sequence
\begin{equation}\label{eq:def_of_beta_w}
X_\ast(\bfT^{\sssc})_{\langle\sigma_w\rangle} \stackrel{\beta_w}{\longrar}  X_\ast(\bfT)_{\langle\sigma_w\rangle} \longrar \pi_1(\bfG)_{\langle \sigma\rangle} \longrar 0.
\end{equation}
In particular, 
\begin{equation}\label{eq:ker_of_coho}
\ker\left(H^1(\sigma_w, \bfT) \longrar H^1(\sigma,\bfG)\right) = \im\left(\beta_w\right)_{\tors}
\end{equation}

\begin{rem}\label{rem:beta_w_0_adjoint_elliptic}
If $\bfG$ is adjoint and $w$ Coxeter, then $\beta_w = 0$ and (equivalently) $X_\ast(\bfT)_{\langle \sigma_w \rangle} \cong \pi_1(\bfG)_{\langle \sigma  \rangle}$. For split groups, this follows from \cite[Corollary~8.3]{Springer_74}. In the non-split case, note that there is some $n\geq 1$ such that $\sigma_w^n = \sigma_{w_{\spl}}$ as automorphisms of $X_{\ast}(\bfT)$ and $X_{\ast}(\bfT^{\Sc})$, where $w_{\spl}$ is some (non-twisted) Coxeter element in $W$ and $\sigma_{w_{\spl}} := \Ad(w_{\spl}) \circ \sigma^n$ is formed with respect to the base change of $\bfG$ to the unramified extension of $k$ of degree $n$. Then, as automorphisms of $X_{\ast}(\bfT)$ and of $X_{\ast}(\bfT^{\Sc})$, $\sigma_{w_{\spl}} - 1$ factors through $\sigma_w - 1$, so that we have a surjection $X_\ast(\bfT)_{\langle\sigma_w\rangle} \tar X_\ast(\bfT)_{\langle\sigma_{w_{\spl}} \rangle}$ and a similar surjection for $X_\ast(\bfT^{\Sc})$. These surjections form a commutative square with the maps $\beta_w$ and $\beta_{w_{\spl}}$ (as both are induced by the inclusion $X_\ast(\bfT^{\Sc}) \rar X_\ast(\bfT)$). As $\beta_{w_{\spl}} = 0$ by the split case, it follows that also $\beta_w = 0$.

\end{rem}

By \cite[Section~5]{PappasR_08} we may reinterpret $\tilde\kappa_\bfG$ as a map of ind-schemes $\tilde\kappa_\bfG \colon L\bfG \rar \pi_1(\bfG)$, where we regard $\pi_1(\bfG)$ as a discrete perfect scheme via Section~\ref{sec:loc_prof_gps}. Passing to $\sigma$-coinvariants gives a surjective map
\begin{equation}\label{eq:bar_Kottwitz_map}
\kappa_\bfG \colon L\bfG \longrar \pi_1(\bfG)_{\langle \sigma \rangle}.
\end{equation}

Finally, let us note that for $w \in W$, we have the exact sequence
\begin{equation}\label{eq:torus_sequence_kottwitz_map}
1 \longrar \bfT_w(k) \longrar L\bfT \stackrel{\sigma_w - 1}{\longrar} L\bfT \stackrel{\kappa_w}{\longrar} X_\ast(\bfT)_{\langle \sigma_w \rangle} \longrar 1,
\end{equation}
where we write $\kappa_w$ instead of $\kappa_{\bfT_w}$. Moreover, $\kappa_w$ factors through the natural projection $\bar\kappa_w \colon X_\ast(\bfT) \rar X_\ast(\bfT)_{\langle \sigma_w\rangle}$.

\subsection{Newton map and $\boldsymbol{\sigma}$-conjugacy classes}\label{sec:newton_and_sigma_classes}

Let $\ff \in \Perf_{\bF_q}$ be an algebraically closed field, so that $L = \bW(\ff)[1/\varpi]$ is an extension of $\breve k$. Then $L/k$ is equipped with the Frobenius $\sigma$ lifting that of $\ff/\bF_q$ and $L^\sigma = k$. In \cite{Kottwitz_85} the set 
\[ 
B(\bfG) = H^1(\sigma^\bZ, \bfG(L)) \cong \bfG(L)/\Ad\nolimits_\sigma \bfG(L) 
\] 
of $\sigma$-conjugacy classes of $\bfG(L)$ in considered, where $x, y$ are $\sigma$-conjugate if there is some $g \in \bfG(L)$ with $g^{-1}x\sigma(g) = y$. Denote the $\sigma$-conjugacy class of $b \in \bfG(L)$ by $[b]$ (or $[b]_\bfG$ if 
ambiguity is possible). The set $B(\bfG)$ is independent of the choice of $\ff$. To a $\sigma$-conjugacy class $[b]$, one can attach the \emph{Newton point} $\nu_b \in (W\backslash X_\ast(\bfT)_{\bQ})^{\langle \sigma \rangle}$. Moreover, the Kottwitz map $\kappa_\bfG(\ff)$ from \eqref{eq:bar_Kottwitz_map} factors through a map $\kappa_\bfG \colon B(\bfG) \tar \pi(\bfG)_{\langle \sigma \rangle}$. Then the map 
\begin{equation}\label{eq:newton_and_kottwitz}
(\nu, \kappa_\bfG) \colon B(\bfG) \longrar (W\backslash X_\ast(\bfT)_{\bQ})^{\langle \sigma \rangle} \times \pi_1(\bfG)_{\langle \sigma \rangle}
\end{equation}
is injective. An element $b\in \bfG(L)$ is called \emph{basic} if $\nu_b$ is central in $\bfG$. We denote by $B(\bfG)_{\bas} \subseteq B(\bfG)$ the subset of basic classes. There is a canonical injection $H^1(\sigma,\bfG) \har B(\bfG)_{\bas}$. In \cite{Kottwitz_85}, the isomorphism $H^1(\sigma,\bfG) \cong \pi_1(\bfG)_{\langle\sigma\rangle, {\tors}}$ was extended to the natural commutative diagram
\begin{equation}\label{eq:H1_and_basic_and_fundamental_group}
\xymatrix{
& H^1(\sigma,\bfG) \ar[r]^(.45){\sim} \ar@{^(->}[d] & \pi_1(\bfG)_{\langle \sigma \rangle,\tors} \ar@{^(->}[d] \\
& B(\bfG)_{\bas} \ar[r]^{\sim} & \pi_1(\bfG)_{\langle \sigma \rangle},}
\end{equation}
where the lower map is the restriction of $\kappa_{\bfG}$ to $B(\bfG)_{\bas}$.

We recall how to compute the Newton point $\nu_b$ of $b \in \bfN(L)$; see \textit{e.g.} \cite[Remark~2.3(iii)]{RapoportV_14}. Let $\tilde w \in \widetilde W$ be the image of $b$ under $\bfN(L) \tar \widetilde W$. There is an integer $d > 0$ such that $\sigma^d$ acts trivially on $\widetilde W$ and such that $\mu = \tilde w \sigma(\tilde w) \dots \sigma^{d-1}(\tilde w) \in X_\ast(\bfT)$. Then $\nu_b$  is the image of $d^{-1}\mu$ in $W\backslash X_\ast(\bfT)_\bQ$. In particular, $\nu_b$ only depends on $\tilde w$, and it makes sense to define $\nu_{\tilde w} = \nu_b$. Moreover, we call $\tilde w$ basic if $\nu_{\tilde w}$ is central.

Finally, we recall that any $\sigma$-conjugacy class has a representative in $\bfN(\breve k)$ and that the induced surjection $\bfN(\breve k) \tar B(\bfG)$ factors through $\widetilde W \tar B(\bfG)$; \textit{cf.}~\cite[Corollary 7.2.2]{GortzHKR_10} for split groups and \cite[Theorem~3.7]{He_14} in general. In particular, the Kottwitz homomorphism induces a map $\bar\kappa_\bfG \colon \widetilde W \rar \pi_1(\bfG)_{\langle \sigma \rangle}$.

\section{Rational and stable conjugacy classes of unramified tori}\label{sec:conj_classes_tori}

Recall that two maximal $k$-tori $\bfT_1,\bfT_2$ in $\bfG$ are \emph{stably conjugate} if there exists a $g \in \bfG(k^{\alg})$ such that ${}^g(\bfT_1(k)) = \bfT_2(k)$. If two maximal $k$-tori are rationally conjugate, they are also stably conjugate. Thus any stable conjugacy class breaks up into a (finite) union of rational conjugacy classes. We review from \cite{deBacker_06,deBackerR_09,Reeder_11} how this happens for unramified tori of pure inner forms of $\bfG$. Then we investigate the case of \emph{extended} pure inner forms of $\bfG$. We wish to point out that the description from \cite{deBacker_06} in terms of the Bruhat--Tits building covers this case, but we are interested in a slightly different description, which is closely related to the spaces $\dot X_{\bar w}(b)$.

\subsection{Stable conjugacy classes} \label{sec:stable_classes_tori}

For $b \in \bfN(\breve k)$, the twisted Frobenius $\sigma_b = \Ad\,b \circ \sigma$ of $\bfG(\breve k)$ induces one on  $W$. Now suppose  $b \in \bfN(\breve k)$ is basic. By \cite[Lemma 4.3.1]{deBacker_06}, there is a natural injection from the set of stable conjugacy classes of unramified maximal $k$-tori in $\bfG_b$ into $H^1(\sigma_b,W)$, which is even a bijection when $\bfG_b$ is quasi-split (\textit{e.g.}, $[b] = [1]$). Identifying $H^1(\sigma_b,W)\isor H^1(\sigma,W)$ via the natural twisting map $x \mapsto xb$ on cocycles (using \cite[Proposition~I.35\,bis]{Serre_Galois_Cohomology}), we get
\begin{equation}\label{eq:stable_classes}
\{\text{stable conjugacy classes of $k$-tori in $\bfG_b$}\} \longhar H^1(\sigma,W),
\end{equation}
mapping the class of $\bfT (\subseteq \bfG_b)$ to the image of $b$ in $H^1(\sigma,W)$.

\begin{Def}\label{def:stable_conj_classes_tori}
For $b\in \bfN(\breve k)$ basic and $w \in W$, denote by $\fT(\bfG_b, w)$ the \emph{stable conjugacy class} of all unramified maximal $k$-tori in $\bfG_b$, which corresponds to $[w] \in H^1(\sigma,W)$ under \eqref{eq:stable_classes}.
\end{Def}

The elements of $\fT(\bfG_b,w)$ are precisely the images of $k$-embeddings $\bfT_w \rar \bfG_b$. 

\begin{rem}\label{rem:shape_of_tori_in_a_stable class}
We recall the construction from the proof of \cite[Lemma 4.3.1]{deBacker_06}. Any unramified maximal $k$-torus of $\bfG_b$ is of the form ${}^g\bfT$ for some $g \in \bfG_b(\breve k)$. Moreover, ${}^g\bfT$ is $k$-rational if and only if  $\sigma_b({}^g\bfT) = {}^g\bfT$, which is equivalent to $g^{-1}b\sigma(g) \in \bfN(\breve k)$. In that case, if $w \in W$ is the image of $g^{-1}b\sigma(g)$, then ${}^g\bfT \in \fT(\bfG_b,w)$. Replacing $g$ by $g\dot v$ for any lift $\dot v\in \bfN(\breve k)$ of some $v \in W$ has the effect of replacing $w$ by $v^{-1}w\sigma(v)$.
\end{rem}

\subsection{Rational conjugacy classes for pure inner forms} \label{sec:rational_classes_pure_forms}

It is easy to see that the surjection $\bfN(\breve k) \tar \widetilde W$ induces an isomorphism\footnote{Indeed, one can use \cite[Corollary~2~to~Proposition~I.39 and Corollary~to~Proposition~I.41]{Serre_Galois_Cohomology} along with the pro-version of Lang's theorem for the group $\caT(\caO_{\breve k}) = \ker(\bfN(\breve k) \tar \widetilde W)$ with varying rational structures given by $\sigma_b$ ($b\in \bfN(\breve k)$).} $H^1(\sigma, \bfN(\breve k)) \cong H^1(\sigma, \widetilde W)$.

The following is a consequence of \cite[Proposition~6.1]{Reeder_11};\footnote{In fact, \cite{Reeder_11} works with $\Gal_k$-cohomology and \emph{all} maximal $k$-tori. To deduce our statements, one can run literally the same arguments with $\langle \sigma \rangle$ replacing $\Gal_k$ and unramified $k$-tori replacing all $k$-tori.} see also \cite[Section~2]{deBackerR_09} and \cite[Section~4.3]{deBacker_06}. Suppose that $[b] \in B(\bfG)_{\bas}$ comes from a class $\xi \in H^1(\sigma,\bfG)$ (\textit{cf.}~\eqref{eq:H1_and_basic_and_fundamental_group}), \textit{i.e.}, $\bfG_b$ is a pure inner form of $\bfG$. Then rational conjugacy classes of unramified maximal $k$-tori in $\bfG_b$ are in bijection with preimages of $\xi$ under the natural map $H^1(\sigma,\widetilde W) \isor H^1(\sigma, \bfN(\breve k)) \rar H^1(\sigma,\bfG)$; \textit{cf.}~\cite[Lemma 6.2]{Reeder_11} (and use twisting). There are maps
\begin{equation}\label{eq:stabe_rational_classes}
H^1(\sigma,W) \stackrel{\;q_0}{\longleftarrow} H^1(\sigma,\widetilde W) \stackrel{\pi_0}{\longrightarrow} H^1(\sigma,\bfG).
\end{equation}
By Section~\ref{sec:stable_classes_tori}, given $[w] \in H^1(\sigma, W)$, represented by $w \in W$, and $\xi \in H^1(\sigma,\bfG)$, represented by $b \in \bfN(\breve k)$, the set $\fT(\bfG_b,w)/\Ad \bfG_b(k)$ of rational conjugacy classes in $\fT(\bfG_b,w)$ is in bijection with $\pi_0^{-1}(\xi) \cap q_0^{-1}([w])$. In particular, $[w]$ is in the image of \eqref{eq:stable_classes} if and only if $\pi_0^{-1}(\xi) \cap q_0^{-1}([w]) \neq \varnothing$.

\subsection{Rational conjugacy classes for extended pure inner forms}
\label{sec:rationa_conj_classes_ext_pure_forms}

We wish to extend \eqref{eq:stabe_rational_classes} to all extended pure inner forms. We have $X_\ast(\bfZ) \subseteq X_\ast(\bfT) \subseteq \widetilde W$. Define 
\[
\widetilde Z^1(\sigma, \widetilde W) = \left\{ \tilde w \in \widetilde W\colon \nu_{\tilde w} \in X_\ast(\bfZ)_{\bQ} \right\},
\]
where $\nu_{\tilde w}$ is as in Section~\ref{sec:newton_and_sigma_classes}. In other words, $\widetilde Z^1(\sigma,\widetilde W)$ is precisely the set of all basic elements in $\widetilde W$ in the sense of Section~\ref{sec:newton_and_sigma_classes}. Define an equivalence relation on $\widetilde Z^1(\sigma, \widetilde W)$ by $x \sim y$ if and only if there exists a $g \in \widetilde W$ such that $y = g^{-1} x \sigma(g)$. Put 
\[
\widetilde H^1(\sigma,\widetilde W) = \widetilde Z^1(\sigma, \widetilde W)/\sim.
\]
Note that we have $H^1(\sigma, \widetilde W)\subseteq \widetilde H^1(\sigma, \widetilde W)$ and $Z^1(\sigma,\widetilde W)\subseteq \widetilde Z^1(\sigma, \widetilde W)$ is precisely the set of those $\tilde w$ with $\nu_{\tilde w} = 0$. Recalling from \eqref{eq:H1_and_basic_and_fundamental_group} that $H^1(\sigma,\bfG) \subseteq B(\bfG)_{\bas}$, we have the following lemma.

\begin{lm}
There is a natural surjection $\pi \colon \widetilde H^1(\sigma, \widetilde W) \rar B(\bfG)_{\bas}$ extending the surjection $\pi_0 \colon H^1(\sigma, \widetilde W) \rar H^1(\sigma,\bfG)$. 
\end{lm}
\begin{proof}
Suppose $\tilde w_1 \sim \tilde w_2 \in \widetilde Z^1(\sigma,\widetilde W)$, \textit{i.e.}, $\tilde w_2 = g^{-1}\tilde w_1 \sigma(g)$. Then obviously $\bar\kappa_{\bfG}(\tilde w_1) = \bar\kappa_{\bfG}(\tilde w_2) \in \pi_1(\bfG)_{\langle \sigma \rangle}$, and an easy computation shows that $\nu_{\tilde w_1} = \nu_{\tilde w_2}$. Thus some (\textit{i.e.}, any) lifts to $\bfG(\breve k)$ of $\tilde w_1$, $\tilde w_2$ are $\sigma$-conjugate. This gives a map $\widetilde H^1(\sigma,\widetilde W) \rar B(\bfG)$ whose image is contained in $B(\bfG)_{\bas}$ as each $\tilde w \in \widetilde Z^1(\sigma,\widetilde W)$ is basic. The surjectivity is clear.
\end{proof}

The projection $\widetilde W \tar W$ induces a map $q \colon \widetilde H^1(\sigma, \widetilde W) \rar H^1(\sigma, W)$. Altogether, we have the diagram
\begin{equation}\label{eq:diagram_pi_q_extended}
H^1(\sigma,W) \stackrel{q}{\longleftarrow} \widetilde H^1(\sigma,\widetilde W) \stackrel{\pi}{\longrightarrow} B(\bfG)_{\bas}
\end{equation}
extending \eqref{eq:stabe_rational_classes}. DeBacker and Reeder's parametrization from Section~\ref{sec:rational_classes_pure_forms} admits the following completely analogous generalization to all extended pure inner forms.

\begin{prop}\label{prop:rational_stable_classes_extended_pure_inner_forms}
Given $[w] \in H^1(\sigma, W)$ and $[b] \in B(\bfG)_{\bas}$, we have a natural bijection $\fT(\bfG_b,w)/\Ad \bfG_b(k) \cong \pi^{-1}([b]) \cap q^{-1}([w])$. In particular, $[w]$ is in the image of \eqref{eq:stable_classes} if and only if $\pi^{-1}([b]) \cap q^{-1}([w]) \neq \varnothing$.
\end{prop}

\begin{proof}
Any unramified maximal $k$-torus in $\bfG_b$ is of the form ${}^g\bfT$ with $g^{-1}b\sigma(g) \in \bfN(\breve k)$; see Remark~\ref{rem:shape_of_tori_in_a_stable class}. Let $\tilde w_g \in \widetilde W$ be the image of $g^{-1}b\sigma(g)$. Then $\nu_{\tilde w_g} = \nu_b$, and hence $\tilde w_g \in \widetilde Z^1(\sigma, \widetilde W)$ and $\pi([\tilde w_g]) = [b]$. Let ${}^{g_i}\bfT$ ($i=1,2$) be two such tori with corresponding elements $\tilde w_i \in \widetilde Z^1(\sigma, \widetilde W)$. Suppose that the ${}^{g_i}\bfT$ are rationally conjugate in $\bfG_b$, \textit{i.e.}, there is some $h \in \bfG_b(k)$ with ${}^{g_1}\bfT = {}^{hg_2}\bfT$. Then $g_1^{-1}h g_2 \in \bfN(\breve k)$, and letting $\tilde v \in \widetilde W$ be its image, we compute
\begin{equation}\label{eq:rational_conj_aux_eq}
\tilde v \tilde w_2 \sigma(\tilde v)^{-1} = g_1^{-1}h g_2 g_2^{-1} b \sigma(g_2) \sigma(g_2^{-1} h^{-1}g_1) = g_1^{-1} h b \sigma(h^{-1}) \sigma(g_1) = g_1^{-1} b \sigma(g_1) = \tilde w_1;
\end{equation}
\textit{i.e.}, $\tilde w_1$ and $\tilde w_2$ define the same class in $\widetilde H^1(\sigma,\widetilde W)$. Conversely, given $[\tilde w]\in \tilde\pi^{-1}([b])$, we can lift $\tilde w$ to some $\dot w \in \bfN(\breve k)$, which necessarily belongs to $[b]$. If $\dot w = g^{-1}b\sigma(g)$, one immediately checks that $\tilde w \mapsto {}^g\bfT$ induces an inverse to the above map. This shows that rational conjugacy classes of unramified tori in $\bfG_b$ are in bijection with $\pi^{-1}([b])$, establishing the analogue of \cite[Lemma 6.2]{Reeder_11}. The rest of the proof (the analogue of \cite[Lemma 6.4]{Reeder_11}) is done analogously.\qedhere
\end{proof}

\subsection{A description of $\boldsymbol{\pi^{-1}([b]) \cap q^{-1}([w])}$} \label{sec:fibers_LFw_Fw}

Fix $w \in W$ and $[b] \in B(\bfG)_{\bas}$. Put
\[
F_w \eqdef \text{fiber of $\bfN \rar W$ over $w$} \quad \text{and} \quad \overline F_w \eqdef \text{fiber of $\widetilde W \rar W$ over $w$}.
\]
We regard $F_w$ as a $\breve k$-scheme and $\overline F_w$ as a discrete $\obF$-scheme (or, sometimes, as a discrete set). Then $F_w$ is a (trivial) $\bfT_{\breve k}$-torsor and $\overline F_w$ is a (trivial) $X_\ast(\bfT)$-torsor. The action of $\bfT_{\breve k}$ on $F_w$ is given by $t,\tilde w \mapsto t\tilde w$, and the action of $X_\ast(\bfT)$ on $\overline F_w$ is induced by this. Recall the maps $\kappa_w \colon L\bfT \rar X_\ast(\bfT)_{\langle \sigma_w\rangle}$ and $\bar\kappa_w \colon X_\ast(\bfT) \rar X_\ast(\bfT)_{\langle \sigma_w\rangle}$ from Section~\ref{sec:sheaf_Kottwitz_map}.

\begin{lm}\label{lm:pass_to_ext_affine_Weylgroup}
The natural surjection $F_w(\breve k) \tar \overline F_w$ induces an isomorphism $LF_w\,/ \ker\kappa_w \cong \overline F_w \,/ \ker\overline\kappa_w $. This discrete set is a $($trivial\,$)$ $X_\ast(\bfT)_{\langle \sigma_w \rangle}$-torsor.
\end{lm}

\begin{proof} As $LF_w$ (resp.\ $\overline F_w$) is a trivial $L\bfT$-torsor (resp.\ $X_\ast(\bfT)$-torsor), everything follows from $L\bfT/\ker\kappa_w = X_\ast(\bfT)\,/\ker\bar\kappa_w = X_\ast(\bfT)_{\langle \sigma_w \rangle}$. \qedhere
\end{proof}

Observe that for $\tilde w \in \overline F_w$, the values of $\nu_{\tilde w}$ and $\bar\kappa_\bfG(\tilde w)$ remain constant if $\tilde w$ is multiplied by an element in $\ker\bar\kappa_w$. Indeed, for $\bar\kappa_w$ this is clear, and for $\nu$ this follows from the description in Section~\ref{sec:newton_and_sigma_classes}. In particular, $\bar\kappa_{\bfG}|_{\overline F_w}$ induces a map
\begin{equation}\label{eq:barbar_kappa_w}
\bar\kappa^w \colon \overline F_w\,/\ker\bar\kappa_w \longrar \pi_1(\bfG)_{\langle \sigma \rangle}. 
\end{equation}
Moreover, from Section~\ref{sec:newton_and_sigma_classes} we immediately deduce the following lemma.

\begin{lm}\label{lm:Newton_and_kottwitz_map_on_Fwmodker}
The map \eqref{eq:newton_and_kottwitz} factors through an injection $(\nu,\bar\kappa^w) \colon \overline F_w\,/\ker\bar\kappa_w \rar (W\backslash X_\ast(\bfT)_{\bQ})^{\langle \sigma \rangle} \times \pi_1(\bfG)_{\langle \sigma \rangle}$. 
\end{lm}

Fix $w \in W$ and $[b] \in B(\bfG)_{\bas}$. Put 
\[
[b] \cap \overline F_w \eqdef {\im}\left([b] \cap F_w(\breve k) \har F_w(\breve k)\tar \overline F_w\right).
\]
Then $[b] \cap \overline F_w \subseteq [b] \cap \widetilde Z^1(\sigma, \widetilde W)$, inducing a map $[b] \cap \overline F_w \rar \pi^{-1}([b]) \subseteq \widetilde H^1(\sigma, \widetilde W)$. This obviously factors through a map $[b] \cap \overline F_w \rar \pi^{-1}([b]) \cap q^{-1}([w])$. By Lemma~\ref{lm:Newton_and_kottwitz_map_on_Fwmodker}, $[b] \cap \overline F_w$ is stable under the action of $\ker\bar\kappa_w$, and we have the (finite) quotient set $[b] \cap \overline F_w\, / \ker\bar\kappa_w$, consisting of all $\bar w \in \overline F_w\,/\ker\bar\kappa_w$ satisfying $\nu_{\bar w} = \nu_b$ and $\bar\kappa^w(\bar w) = \kappa_\bfG(b)$.

{\samepage 
\begin{prop}\label{prop:conj_classes_of_tori_basic_case} \leavevmode
\begin{enumerate}
\item\label{cctbc-1} The quotient $[b] \cap \overline F_w\,/ \ker\bar\kappa_w$ is either empty or a $\ker(H^1(\sigma_w,\bfT) \rar H^1(\sigma, \bfG)) = \im(\beta_w)_{\tors}$-torsor $($\textit{cf.}~\eqref{eq:ker_of_coho}$)$. It is empty if and only if\, $\fT(\bfG_b,w) = \varnothing$. 
\item\label{cctbc-2} The centralizer $C_W(w\sigma) \subseteq W$ of\, $w\sigma \in W \rtimes \langle \sigma \rangle$ acts on $[b] \cap \overline F_w\,/\ker\bar\kappa_w$ in a natural way, and there are canonical isomorphisms 
\[
\left([b] \cap \overline F_w\,/\ker\bar\kappa_w \right) \slash C_W(w\sigma) \isorlong  \pi^{-1}([b]) \cap q^{-1}([w]) \cong \fT(\bfG_b,w)\,/\Ad \bfG_b(k).
\]
\end{enumerate}
\end{prop}
}

\begin{proof}
We have $\fT(\bfG_b,w)\neq \varnothing \,\LRar\, \exists g \in \bfG(\breve k) \text{ such that } g^{-1}b\sigma(g) \in F_w(\breve k) \,\LRar\, [b] \cap \overline F_w \neq \varnothing$,
where the first equivalence follows from Remark~\ref{rem:shape_of_tori_in_a_stable class}. To show the first claim of~\eqref{cctbc-1}, we may assume that there is an element $\tilde w_1 \in [b] \cap \overline F_w \neq \varnothing$. Recall that $\overline F_w$ is a $X_\ast(\bfT)$-torsor. 

\begin{lm}\label{lm:torsor_in_part_i_of_rat_conj_classes_prop} For $\tau \in X_\ast(\bfT)$, we have $\tau\tilde w_1 \in [b] \cap \overline F_w \LRar \bar\kappa_w(\tau) \in \ker(H^1(\sigma_w, \bfT) \rar H^1(\sigma, \bfG))$. 
\end{lm}

\begin{proof}
We have $\tau\tilde w_1 \in [b] \cap \overline F_w$ if and only if $\bar\kappa_\bfG(\tau) = 0$ (as $\kappa_\bfG$ is a group homomorphism) and $\nu_{\tau \tilde w_1} = \nu_{\tilde w_1}$. Now, $\bar\kappa_\bfG(\tau) = 0$ means that the image $\bar\kappa_w(\tau) \in X_\ast(\bfT)_{\langle \sigma_w \rangle}$ maps to $0 \in \pi_1(\bfG)_{\langle \sigma \rangle}$, which is equivalent to $\bar\kappa_w(\tau) \in \im(\beta_w)$, with $\beta_w$ as in \eqref{eq:def_of_beta_w}. Now we use the algorithm to compute $\nu$ from Section~\ref{sec:newton_and_sigma_classes}. Note that the integer $d$ can be chosen the same for $\tilde w_1$ and $\tau\tilde w_1$. Then let $\mu_{\tilde w_1}$ and $\mu_{\tau \tilde w_1}$ be as described in Section~\ref{sec:newton_and_sigma_classes}. Then the equality $\nu_{\tau \tilde w_1} = \nu_{\tilde w_1}$ is equivalent to the existence of some $v \in W$ such that the equality
\[
\sum_{i=0}^{d-1} (w\sigma)^i(\tau) = v(\mu_{\tilde w_1}) - \mu_{\tilde w_1}
\]
holds in $X_\ast(\bfT)$. But $\tilde w_1$ is basic, so $\mu_{\tilde w_1}$ is central; \textit{i.e.}, the right-hand side is zero. Thus $\tau\tilde{w}_1 \in [b] \cap \overline F_w$ if and only if $\bar\kappa_w(\tau) \in \im(\beta_w)$ and $\sum_{i=0}^{d-1} (w\sigma)^i(\tau) = 0$. By \cite[Lemma 2.4.1]{deBackerR_09}, the latter condition is equivalent to $\tau \in \bar\kappa_w^{-1}(X_\ast(\bfT)_{\langle \sigma_w \rangle, \tors}) = \bar\kappa_w^{-1}(H^1(\sigma_w,\bfT))$. 
\end{proof}

By Lemma~\ref{lm:torsor_in_part_i_of_rat_conj_classes_prop}, $[b] \cap \overline F_w$ (being non-empty) is a torsor under the preimage in $X_\ast(\bfT)$ of $\ker(H^1(\sigma, \bfT_w) \rar H^1(\sigma, \bfG)) \subseteq X_\ast(\bfT)_{\langle\sigma_w\rangle}$. This preimage contains $\ker\bar\kappa_w$, and factoring it out, we arrive at the claim. This finishes the proof of~\eqref{cctbc-1}.

In~\eqref{cctbc-2}, it is clear that we have a map $[b] \cap \overline F_w \rar \pi^{-1}([b]) \cap q^{-1}([w])$, and one checks that it factors through $[b] \cap \overline F_w\,/\ker\bar\kappa_w$. Moreover, surjectivity of this map amounts to the tautological fact that any element in $[w]\subseteq W$ is $\sigma$-conjugate to $w$. Next, let us describe the natural action of $C_W(w\sigma)$ on $[b] \cap \overline F_w\,/\ker\bar\kappa_w$. Let $\bar w\in [b] \cap \overline F_w\,/\ker\bar\kappa_w$ with preimage $\tilde w \in [b]\cap \overline F_w$. Let $v \in C_W(w\sigma)$, and let $\tilde v \in \widetilde W$ be any preimage. Then the action of $v$ sends $\bar w$ to the image of $\tilde v^{-1} \tilde w \sigma(\tilde v)$ in $[b]\cap \overline F_w\,/\ker\bar\kappa_w$. This is independent of the choice of~$\tilde v$, as if $\tilde v'$ is another choice, putting $\tau = \tilde v^{-1}\tilde v' \in X_\ast(\bfT)$ gives the element $\tilde v^{\prime -1} \tilde w \sigma(\tilde v') = \tau^{-1}\tilde v^{-1}\tilde w \sigma(\tilde v)\sigma(\tau)$, which has the same image in $[b]\cap \overline F_w\,/\ker\bar\kappa_w$. Similarly, one shows that it is independent of the choice of~$\tilde w$. Now, $\bar w' \in [b] \cap \overline F_w\,/\ker\bar\kappa_w$ has the same image as $\bar w$ in $\pi^{-1}([b]) \cap q^{-1}([w])$ if and only if there is some $\tilde v \in \widetilde W$ with $\tilde v^{-1}\tilde w \sigma(\tilde v) = \tilde w'$ for some lift $\tilde w' \in \overline F_w$ of $\bar w'$. By the definition of the $C_W(w\sigma)$-action, this is equivalent to the fact that $\bar w$ and $\bar w'$ lie in the same $C_W(w\sigma)$-orbit. Thus $C_W(w\sigma)$-orbits in $[b] \cap \overline F_w\,/\ker\bar\kappa_w$ coincide with fibers of the surjection onto $\pi^{-1}([b]) \cap q^{-1}([w])$, and we are done.\qedhere
\end{proof}

\begin{rem} Let us explicate the correspondence in Proposition~\ref{prop:conj_classes_of_tori_basic_case}\eqref{cctbc-2}. If $\bar w \in [b] \cap \overline F_w\,/\ker\bar\kappa_w$ represents a class on the left side, $\dot w \in [b] \cap F_w(\breve k)$ is a (\textit{i.e.}, any) lift of $\bar w$, and $g \in \bfG(\breve k)$ satisfies $g^{-1}b \sigma(g) = \dot w$, then the map ${\Ad(g)} \circ \iota \colon \bfT_w \rar \bfG_b$ is $k$-rational, its image ${}^g \bfT \subseteq \bfG_b$ lies in $\fT(\bfG_b,w)$ and the $\bfG_b(k)$-conjugacy class of ${}^g\bfT$ is the image of $\bar{w}$ under the above map.
\end{rem}

\subsection{The case of Coxeter tori}\label{sec:Coxeter_tori_case}

We keep the setup of Sections~\ref{sec:rationa_conj_classes_ext_pure_forms} and~\ref{sec:fibers_LFw_Fw}. Furthermore, throughout Section~\ref{sec:Coxeter_tori_case} we assume that $w=c \in W$ \emph{is a Coxeter element}. Recall $\beta_c \colon X_\ast(\bfT^{\sssc})_{\langle\sigma_c\rangle} \rar X_\ast(\bfT)_{\langle\sigma_c\rangle}$ from \eqref{eq:def_of_beta_w}.

{\samepage 
\begin{prop}\label{cor:rational_classes_Coxeter_case} Let $[b] \in B(\bfG)_{\bas}$ and $c \in W$ be Coxeter. Then the following hold: 
\begin{enumerate}
\item\label{rcCc-1} The image $\im(\beta_c)$ is finite.
\item\label{rcCc-2} We have 
\[ [b] \cap \overline F_c\,/\ker\bar\kappa_c = (\bar\kappa^c)^{-1}(\kappa_\bfG(b)).
\]
This is an $\im(\beta_c)$-torsor. In particular, $[b]\cap \overline F_c \neq \varnothing$.
\end{enumerate}
\end{prop}
}

\begin{proof}
\eqref{rcCc-1} As $c$ Coxeter, $X_\ast(\bfT^{\sssc})^{\langle \sigma_c \rangle} = 0$; hence $X_\ast(\bfT^{\sssc})_{\langle \sigma_c \rangle}$ is finite. Hence $\im(\beta_c)$ is finite. For \eqref{rcCc-2}  we first record the following easy and well-known lemma. 

\begin{lm}\label{lm:Cox_gives_basics}
If $c$ is Coxeter, all elements in $\overline F_c$ are basic.
\end{lm}

\begin{proof}
Let $\tilde c \in \overline F_c$. If $\mu$ is as in the algorithm to compute $\nu_b$ from Section~\ref{sec:newton_and_sigma_classes}, we see that $\mu \in X_\ast(\bfT)^{\langle\sigma_c\rangle}_{\bQ} \subseteq X_\ast(\bfZ)_{\bQ}$,  the last inclusion being true as $c$ is Coxeter.
\end{proof}

Now note that $\bar \kappa_\bfG|_{\overline F_c} \colon \overline F_c \rar \pi_1(\bfG)_{\langle \sigma \rangle}$ is surjective. This holds as $\overline F_c$ is a $X_\ast(\bfT)$-torsor, $\kappa_{\bfG}$ is a group homomorphism and $X_\ast(\bfT) \rar \pi_1(\bfG)_{\langle \sigma \rangle}$ is surjective.  Taking any preimage $b'$ of $\kappa_\bfG(b)$ in $\overline F_c$, Lemma~\ref{lm:Cox_gives_basics} shows that $b'$ is basic. Moreover, we know that $\kappa_\bfG$ induces a bijection $B(\bfG)_{\bas} \rar \pi_1(\bfG)_{\langle \sigma \rangle}$. It follows that $b' \in [b]$; \textit{i.e.}, $[b] \cap \overline F_c \neq \varnothing$. Moreover, as this holds for any $b'$ in the preimage, we even get the equality claimed in~\eqref{rcCc-2}. Now, $(\bar\kappa^c)^{-1}(\bar\kappa_\bfG(b))$ is a $\ker(X_\ast(\bfT_{\langle \sigma_c \rangle} \rar \pi_1(\bfG)_{\langle \sigma \rangle}) = \im(\beta_c)$-torsor. 
\end{proof}

\begin{rem}\label{rem:basic_represented_over_Coxeter}
It follows from Proposition~\ref{cor:rational_classes_Coxeter_case}\eqref{rcCc-2} that for any $[b] \in B(\bfG)_{\bas}$ and any Coxeter element $c$, there exists a representative of $[b]$ in $F_c(\breve k)$. This is a converse to Lemma~\ref{lm:Cox_gives_basics}.
\end{rem}

\begin{prop}\label{prop:Cox_tori_parametrization}
Let $c \in W$ be Coxeter. Suppose one of the following holds:
\begin{enumerate}
\item\label{Ctp-1} We have $[b]=1$.
\item\label{Ctp-2} We have $\bfG = \prod_{i=1}^s \bfG_i$, where for each $i$, $\bfG_i = {\Res}_{k_i/k}\bfG_i'$ for a finite unramified extension $k_i/k$ and an absolutely almost simple $k_i$-group $\bfG'_i$. 
\end{enumerate}
Then the $C_W(c\sigma)$-action on $[b] \cap \overline F_c\,/\ker\bar\kappa_c$ is trivial, and there is a natural bijection 
\[ [b] \cap \overline F_c\,/\ker\bar\kappa_c \isorlong \fT(\bfG_b,c)\,/\Ad\bfG_b(k).\] 
\end{prop}

\begin{proof}
Choose some $\sigma$-equivariant splitting $\iota \colon W \rar \widetilde W$ (which exists, as $\bfG$ is unramified), and denote its image by $\dot W$. Let $\dot c \in \overline F_c \cap \dot W$. Let $\tau \in X_\ast(\bfT)$ be such that $\tau \dot c \in [b] \cap \overline F_c$. Let $v \in C_W(c\sigma)$, and let $\dot v = \iota(v)$. The action of $v$ sends the image of $\tau\dot c$ in $[b]\cap \overline F_c\,/\ker\bar\kappa_c$ to the image in $[b]\cap \overline F_c\,/\ker\bar\kappa_c$ of the element $\dot v^{-1}\tau \dot c \sigma(\dot v)$.  As $\iota$ is $\sigma$-equivariant, $\dot v \in C_{\dot W}(\dot c\sigma)$, so $\dot v^{-1}\tau \dot c \sigma(\dot v) = \dot v^{-1}\tau \dot v\dot c = \Ad(v)^{-1}(\tau)\dot c$; \textit{i.e.}, we must show (replacing $v$ with $v^{-1}$) that $(\Ad\,v - \id)(\tau) = \Ad\,v(\tau) - \tau \in \ker\bar\kappa_c$. 

In case~\eqref{Ctp-1}, we have $[1] \cap \overline F_c = X_\ast(\bfT^{\sssc})\cdot \dot c$. Thus, as $\ker\bar\kappa_c^{\sssc} \subseteq \ker\bar\kappa_c$, it suffices to show that we have $(\Ad\,v - \id)(X_\ast(\bfT^{\sssc})) \subseteq \ker\bar\kappa_c^{\sssc}$. In other words, we are reduced to the case that $\bfG$ is simply connected, which is covered by~\eqref{Ctp-2}. 

In case~\eqref{Ctp-2}, one is immediately reduced to the case that $\bfG = \Res_{k'/k}\bfG'$ with $k'/k$ a finite subextension of $\breve k/k$ and $\bfG'$ an absolutely almost simple $k'$-group. Let $d = [k'\colon k]$. Write $X_\ast(\bfT) = \bigoplus_{i=1}^d X_\ast(\bfT')$ and $W = \prod_{i=1}^d W'$, with $\sigma$-action given by $(v_1,\dots,v_d) \mapsto (\sigma'(v_d),v_1,\dots, v_{d-1})$ on both objects, where $\sigma'$ is the $\breve k/k'$-Frobenius for $\bfG'$. If $c,c_1 \in W$ are $\sigma$-conjugate, the statement claimed in the proposition holds for $c$ if and only if it holds for $c_1$. Thus we may assume that $c = (c',1,\dots, 1)$, where $c' \in W'$ is a Coxeter element in~$W'$. Suppose we know the claim for $\bfG',c'$; \textit{i.e.}, for any $v' \in C_{W'}(c'\sigma')$, $(\Ad v' - \id)(X_\ast(\bfT')) \subseteq (\sigma'_{c'} - \id)(X_\ast(\bfT'))$. We have 
\begin{align*} \ker\bar\kappa_c &= (\sigma_c - \id)(X_\ast(\bfT)) = \left\{ \sigma'_{c'}(\alpha_d) - \alpha_1, \alpha_1 - \alpha_2,\dots,\alpha_{r-1}-\alpha_d  \colon \alpha_i \in X_\ast(\bfT') \text{ for } 1\leq i \leq d\right\} \\ 
&=\left\{\left(\sigma'_{c'}(\alpha_r) - \alpha_d - \sum_{i=1}^{d-1}\beta_i,\beta_1, \dots, \beta_{d-1}\right) \colon \beta_i,\alpha_d \in X_\ast(\bfT') \text{ for } 1\leq i \leq d-1 \right\}
\end{align*}
The inclusion $W' \rar W$ into the first factor induces an isomorphism $C_{W'}(c'\sigma') \isor C_W(c\sigma)$, $v' \mapsto (v',\dots,v') = v$. For $(e_i)_{i=1}^d \in X_\ast(\bfT)$, we have $(\Ad v - \id)((e_i)_i) = ((\Ad v' - \id)(e_i))_{i=1}^d$. This lies in $\ker\bar\kappa_c$ as, setting $\beta_i := (\Ad v' - \id)(e_{i+1})$ for $1\leq i \leq d-1$, it suffices to find an $\alpha_d \in X_\ast(\bfT')$ such that $\sigma'_{c'}(\alpha_d) - \alpha_d = \Ad v'(\sum_{i=1}^d e_i) - \sum_{i=1}^d e_i$, which works by assumption for $\bfG'$.

It remains to check the claim for an absolutely almost simple group $\bfG$ over $k$. Let $d > 0$ be the smallest positive integer such that $\sigma^d$ acts trivially on the Dynkin diagram of $\bfG$. By \cite[Theorem~7.6(v)]{Springer_74}, $C_W(c\sigma)$ is the cyclic subgroup of $W$ generated by $c^{\spl} = c\sigma(c)\dots\sigma^{d-1}(c)$. Thus, it suffices to check that $(\Ad c^{\spl} - \id)(X_\ast(\bfT)) \subseteq \ker\bar\kappa_c = \im(\sigma_c - \id \colon X_\ast(\bfT) \rar X_\ast(\bfT))$. Observe that $\sigma_c(\ker\bar\kappa_c) = \ker\bar\kappa_c$ (as $\sigma_c$ is an automorphism and for any $v \in X_\ast(\bfT)$, $\sigma_c(\sigma_c(v) - v) = \sigma_c(\sigma_c(v)) - \sigma_c(v) \in \ker\bar\kappa_c$). 
Now $\sigma^d$ acts trivially on $X_\ast(\bfT)$, so it suffices to check that for all $\tau \in X_\ast(\bfT)$, $c^{\spl}\sigma^d(\tau) - \tau \in \ker\bar\kappa_c$, or equivalently, $\sigma_c^d(v) - v \in \ker\bar\kappa_c$. We now prove that $\sigma_c^i(v) - v \in \ker\bar\kappa_c$ for all $i>0$. For $i=1$, this holds as $\ker\bar\kappa_c = \im(\sigma_c - 1)$. For $i>1$, we have $\sigma_c^i(v) - v = \sigma_c^i(v) - \sigma_c^{i-1}(v) + \sigma_c^{i-1}(v) - v = \sigma_c^{i-1}(\sigma_c(v) - v) + \sigma_c^{i-1}(v) - v$. Now we have $\sigma_c^{i-1}(v) - v\in \ker\bar\kappa_c$ by induction on $i$, and $\sigma_c^{i-1}(\sigma_c(v) - v) \in \sigma_c^{i-1}(\ker\bar\kappa_c) = \ker\bar\kappa_c$. \qedhere
\end{proof}

Proposition~\ref{prop:Cox_tori_parametrization} becomes false without any assumptions on $\bfG$, $b$, as the next example shows. 

\begin{ex} Suppose $\charac k \neq 2$. Let $\bfG' = {\bf SL}_2 \times {\bf SL}_2$ and $\bfG = \bfG'/\mu_2$, embedded diagonally. Let $\bfT' \subseteq \bfG'$ be the product of diagonal tori in ${\bf SL}_2$ and $\bfT$ its image in $\bfG$. Let $\varepsilon_1 = ((1,-1), (0,0))$, $\varepsilon_2 = ((0,0),(1,-1))$ be a $\bZ$-basis of $X_\ast(\bfT')$. We may identify $X_\ast(\bfT) = \{\frac{a_1 \varepsilon_1 + a_2\varepsilon_2}{2} \in X_\ast(\bfT')_\bQ \colon a_1,a_2 \in \bZ, a_1 \equiv a_2 \mod 2\}$. Let $s$ be the image of $\left(\begin{smallmatrix}0&1\\-1&0\end{smallmatrix}\right)$ in the Weyl group of the diagonal torus in ${\bf SL}_2$, so that $c= (s,s)$ represents a Coxeter element in $W$. Let $\dot c_0 = \left(\left(\begin{smallmatrix}0&1\\-1&0\end{smallmatrix}\right),\left(\begin{smallmatrix}0&1\\-1&0\end{smallmatrix}\right)\right) \in F_c(\breve k)$ be a hyperspecial lift, and let $\bar c_0$ be its image in $\overline F_c$. We have $C_W(c) = W \cong \bZ/2\bZ \times \bZ/2\bZ$. Note that $\tau_1 = (\varpi^{\frac{1}{2}},-\varpi^{-\frac{1}{2}}),(\varpi^{\frac{1}{2}},-\varpi^{-\frac{1}{2}}) \in \bfG(k)$ by Galois descent. The action of the Frobenius $\sigma$ on all objects below is trivial, so we ignore it. We have $\pi_1(\bfG) = X_\ast(\bfT)\,/X_\ast(\bfT') \cong \bZ/2\bZ$, the non-trivial element being the class of $\frac{\varepsilon_1+\varepsilon_2}{2}$. It is represented by the basic element $b_1 = \dot c_0\tau_1$. We have $\ker\bar\kappa_c = \bZ \cdot 2\varepsilon_1 + \bZ \cdot (\varepsilon_2 - \varepsilon_1) \subseteq X_\ast(\bfT)$.

\subsubsection*{\it Case~$b = 1$}
We have $[1] \cap \overline F_c = X_\ast(\bfT') \cdot \bar c_0$. As $X_\ast(\bfT')\,/\ker\bar\kappa_c = \{\bar 0, \bar \varepsilon_1\} \cong \bZ/2\bZ$, we have $[1]\cap \overline F_c\,/\ker\bar\kappa_c = (X_\ast(\bfT')\,/\ker\bar\kappa_c)\cdot \bar c_0 = \{\bar0, \bar\varepsilon_1 \} \cdot \bar c_0$. Now it is immediate to check that the action of $C_W(c)$ from Proposition~\ref{prop:conj_classes_of_tori_basic_case} on this set is trivial. For example, $\Ad(s,1)$ send $\varepsilon_i \in X_\ast(\bfT')$ to $(-1)^i\varepsilon_i$ ($i \in 1,2$), so $\Ad(s,1)(\varepsilon_i) - \varepsilon_i \in \ker\bar\kappa_c$. This is in agreement with Proposition~\ref{prop:Cox_tori_parametrization} and also with the Bruhat--Tits picture: $\bfG^{\ad} \cong {\bf PGL}_2 \times {\bf PGL}_2$, so an alcove of $\caB_{\bfG,k}$ is a square whose vertices are hyperspecial of type $(\alpha,\beta)$ with $\alpha,\beta \in \bZ/2\bZ$ the parity of the valuation of the determinant of the lattice representing a vertex of $\caB_{{\bf PGL}_2}$. The action of $\tau_1 \in \bfG(k)$ identifies the vertices $(0,0) \leftrightarrow (1,1)$ and $(0,1) \leftrightarrow (1,0)$, so that there are two $\bfG(k)$-orbits in hyperspecial vertices of a base alcove and hence (by \cite{deBacker_06}) two rational conjugacy classes of Coxeter tori. 

\subsubsection*{\it Case~$b = b_1$}
We have $[b_1] \cap \overline F_c = (X_\ast(\bfT) \sm X_\ast(\bfT')) \cdot \bar c_0$. As $(X_\ast(\bfT) \sm X_\ast(\bfT'))\,/\ker\bar\kappa_c = \{ \frac{\bar\varepsilon_1 + \bar\varepsilon_2}{2}, \frac{\bar\varepsilon_1 + 3\bar\varepsilon_2}{2}\}$, we have $[b_1] \cap \overline F_c\,/\ker\bar\kappa_c = \{ \frac{\bar\varepsilon_1 + \bar\varepsilon_2}{2}, \frac{\bar\varepsilon_1 + 3\bar\varepsilon_2}{2}\}\cdot \dot c_0$. The action of $(s,1) \in C_W(c)$ identifies the two elements in $[b_1] \cap \overline F_c\,/\ker\bar\kappa_c$ as $\frac{\varepsilon_1 + 3\varepsilon_2}{2} - \Ad(s,1)(\frac{\varepsilon_1 + \varepsilon_2}{2}) = \frac{\varepsilon_1 + 3\varepsilon_2}{2} - \frac{-\varepsilon_1 + \varepsilon_2}{2} = \varepsilon_1 + \varepsilon_2 \in \ker\bar\kappa_c$; \textit{i.e.}, the $C_W(c)$-action is non-trivial. This agrees with Bruhat--Tits theory as the adjoint building of $\bfG_{b_1}$ over $k$ is a point (it is the barycenter of an alcove of $\caB_{\bfG,k}$), so that there is just one $\bfG_{b_1}(k)$-conjugacy class of Coxeter tori.
\end{ex}

\section{The maps \texorpdfstring{$\boldsymbol{\dot X_{\bar w}(b) \rar X_w(b)}$}{from Xdot\textunderscore{(bar w)}(b) to X\textunderscore{w}(b)}}\label{sec:stuff_on_torsors}

Here we review and extend some results from \cite[Section~11]{Ivanov_DL_indrep}.
By \cite[Lemma 11.5]{Ivanov_DL_indrep} the space $\dot X_{\dot w}(b)$ depends---up to a $\bfG_b(k) \times \bfT_w(k)$-equivariant isomorphism---only on the image of $\dot w$ in the discrete set $LF_w\,/\ker\kappa_w$ and not on $\dot w$. As the concrete choice of $\dot w$ never plays an important role, we make the following simplifying definition. Recall from Lemma~\ref{lm:pass_to_ext_affine_Weylgroup} that $LF_w\,/\ker\kappa_w = \overline F_w\,/\ker\bar\kappa_w$.

\begin{Def}\label{conv:dotXwb_indep}
For $b \in \bfG(\breve k)$ and $\bar w \in \overline F_w\,/\ker\bar\kappa_w$, we denote by $\dot X_{\bar w}(b)$ the space $\dot X_{\dot w}(b)$ for an arbitrary lift $\dot w \in F_w(\breve k)$ of $\bar w$.
\end{Def}

\begin{prop}[\protect{\textit{cf.} \cite[Propositions~11.4 and 11.9]{Ivanov_DL_indrep}}] \label{prop:from_DL_unram_class_torsors} 
Let $w \in W$ and $b \in \bfG(\breve k)$. There is a canonical map of arc-sheaves
\[
\alpha_{w,b} \colon X_w(b) \longrar \overline F_w\,/\ker\overline \kappa_w
\]
defining a clopen decomposition $X_w(b) = \coprod_{\bar w \in \overline F_w \,/ \ker\overline\kappa_w} X_w(b)_{\bar w }$, where $X_w(b)_{\bar w} := \alpha_{w,b}^{-1}(\{\bar w\})$. For $\bar w \in \overline F_w\,/\ker\bar\kappa_w$, $\dot X_{\bar w}(b) \rar X_w(b)_{\bar w}$ is a pro-\'etale $\bfT_w(k)$-torsor. In particular, $\dot X_{\bar w}(b) \neq \varnothing$ if and only if\, $\bar w \in \im (\alpha_{w,b})$.
\end{prop}
\begin{proof}
Combine \cite[Propositions~11.4 and 11.9]{Ivanov_DL_indrep} with Lemma~\ref{lm:pass_to_ext_affine_Weylgroup}.
\end{proof}

We have the following emptyness criterion for $\dot X_{\bar w}(b)$, \textit{i.e.}, a constraint on the image of $\alpha_{w,b}$. Recall the map $\bar\kappa^w \colon \overline F_w\,/\ker\bar\kappa_w \rar \pi_1(\bfG)_{\langle \sigma \rangle}$ from \eqref{eq:barbar_kappa_w}.

\begin{prop}\label{lm:nonemptiness_condition}
Let $b \in \bfG(\breve k)$ and $\bar w \in \overline F_w\,/\ker\bar\kappa_w$. If\, $\dot X_{\bar w}(b) \neq \varnothing$, then $\kappa_{\bfG}(b) = \bar\kappa^w(\bar w)$. In particular, $\alpha_{w,b}$ factors through a map $\alpha_{w,b} \colon X_w(b) \rar (\bar\kappa^w)^{-1}(\kappa_\bfG(b))$.
\end{prop}

\begin{proof}
Let $\dot w \in F_w(\breve k)$ be a lift of $\bar w$. As $\dot X_{\dot w}(b) \neq \varnothing$, there is some algebraically closed field $\ff \in \Perf_{\obF}$ such that $\dot X_{\dot w}(b)(\ff) \neq \varnothing$. Let $\bar g \in \dot X_{\dot w}(b)(\ff)$, and let $g \in \bfG(L)$ be a lift of $\bar g$. Then $g^{-1}b\sigma(g) \in \bfU(L) \dot w \bfU(L)$. We clearly have $\kappa_\bfG(g^{-1}b \sigma(g)) = \kappa_\bfG(b)$. Furthermore, any element of $\bfU(L)$ can be conjugated to some Iwahori subgroup, and the Kottwitz map is trivial on Iwahori subgroups. Thus $\kappa_\bfG(\bfU) = 1$.
Thus $\kappa_\bfG(\bfU(L) \dot w \bfU(L)) = \kappa_\bfG(\dot w)$, and the first claim follows. The last claim follows from the first and Proposition~\ref{prop:from_DL_unram_class_torsors}.
\end{proof}

Recall the map $\beta_w \colon X_\ast(\bfT^{\sssc})_{\langle \sigma_w \rangle} \rar X_{\ast}(\bfT)_{\langle \sigma_w \rangle}$ from \eqref{eq:def_of_beta_w}. We have the following diagram with compatible simply transitive actions:
\begin{equation}\label{eq:diagram_actions_easy}
\xymatrix{
(\bar\kappa^w)^{-1}(\kappa_{\bfG}(b)) \ar@{^(->}[r]               &  \overline F_w\,/\ker\bar\kappa_w \ar@{->>}[r]^{\bar\kappa^w} & \pi_1(\bfG)_{\langle \sigma\rangle } \\
\im(\beta_w) \ar@{^(->}[r]  \ar@{}[u]^{\circlearrowleft}   & X_\ast(\bfT)_{\langle \sigma_w \rangle} \ar@{->>}[r] \ar@{}[u]^{\circlearrowleft}  & \pi_1(\bfG)_{\langle \sigma\rangle }\rlap{,} \ar@{}[u]^{\circlearrowleft} 
}
\end{equation}
where the actions are induced by the left multiplication action of $X_\ast(\bfT)$ on $\overline F_w$.

\begin{cor}\label{cor:condition_when_only_one_torsor}
If\, $\beta_w = 0$, then $\alpha_{w,b}$ factors through a point, \textit{i.e.}, there is at most one $\bar w \in \overline F_w\,/\ker\bar\kappa_w$ such that $\dot X_{\bar w}(b) \neq \varnothing$. In this case, $X_w(b) \neq \varnothing$ if and only if  $\dot X_{\bar w}(b) \neq \varnothing$ and $\dot X_{\bar w}(b) \rar X_w(b)$ is a pro-\'etale $\bfT_w(k)$-torsor. 
\end{cor}

\begin{cor}\label{cor:adjoint_Case_torsors}
If\, $\bfG$ is of adjoint type and $w$ is Coxeter, the conclusions of Corollary~\ref{cor:condition_when_only_one_torsor} hold. Moreover, in this case, the map $\dot X_{\bar w}(b) \rar X_w(b)$ is quasi-compact.
\end{cor}
\begin{proof}
The first claim follows from Corollary~\ref{cor:condition_when_only_one_torsor} and Remark~\ref{rem:beta_w_0_adjoint_elliptic}.
The last claim follows as $\dot X_{\bar w}(b) \rar X_w(b)$ is a torsor under the profinite group $\bfT_w(k)$.
\end{proof}

In the case that $b$ is basic and $w = c$ is Coxeter, the situation simplifies as follows.

\begin{cor}\label{cor:choice_b_cox_over_barw}
Suppose $b$ is basic and $c$ is Coxeter. For any $\bar c \in \overline F_c\,/\ker\bar\kappa_c$ such that $X_{\bar c}(b) \neq \varnothing$, there exists a $b_1 \in [b] \cap F_c(\breve k)$ lying over $\bar c$, and there is a $\bfG_b(k) \times \bfT_c(k)$-equivariant isomorphism $\dot X_{\bar c}(b) \cong \dot X_{\bar c}(b_1)$.
\end{cor}

\begin{proof}
By Proposition~\ref{lm:nonemptiness_condition}, $\dot X_{\bar c}(b) \neq \varnothing$ forces $\bar c \in (\bar\kappa^c)^{-1}(\bar\kappa_{\bfG}(b)) = [b] \cap \overline F_c \,/ \ker\bar\kappa_c$, where the equality holds by Proposition~\ref{cor:rational_classes_Coxeter_case}. This implies the existence of $b_1$. For any $g \in \bfG(\breve k)$ such that $g^{-1}b\sigma(g) = b_1$, $x \mapsto gx$ induces an equivariant isomorphism $\dot X_{\bar c}(b) \cong \dot X_{\bar c}(b_1)$ by \cite[Remark~8.7]{Ivanov_DL_indrep}.
\end{proof}

Let us also record the following consequence.

\begin{cor}\label{cor:tori_and_covers}
Suppose $b$ is basic and $c$ is Coxeter. Then $\alpha_{c,b}$ factors through a surjection 
\[
\alpha_{c,b} \colon X_c(b) \longrar [b] \cap \overline F_c\,/\ker\bar\kappa_c.
\]
There is a canonical surjection from the set of different non-empty pieces $X_c(b)_{\bar c} \subseteq X_c(b)$ to the set $\fT(\bfG_b,c)\,/\Ad\bfG_b(k)$ of rational conjugacy classes of unramified Coxeter tori in $\bfG_b$. 

In particular, if\, $\bfG, b$ satisfy the assumptions of Proposition~\ref{prop:Cox_tori_parametrization}, then this surjection is a $1$-to-$1$ correspondence.
\end{cor}

\begin{proof}
The surjectivity of the map $\alpha_{c,b}$ in the corollary follows (for example) from Theorem~\ref{thm:main_decomposition}. The rest immediately follows from Propositions~\ref{prop:conj_classes_of_tori_basic_case},~\ref{cor:rational_classes_Coxeter_case} and~\ref{prop:Cox_tori_parametrization}
\end{proof}

We finish this section with the following remark on the non-basic case.

\begin{rem}\label{rem:Cox_nonbasic_case}
If $[b]$ is non-basic and $c$ Coxeter, then $[b] \cap \overline F_c \,/\ker\bar\kappa_c = \varnothing$ by Lemma~\ref{lm:Cox_gives_basics}. Hence---although we still have $\im(\beta_c)_{\tors} = \im(\beta_c)$ just as in Proposition~\ref{cor:rational_classes_Coxeter_case}---we cannot conclude that $[b] \cap \overline F_c\,/\ker\kappa_c = (\bar\kappa^c)^{-1}(\kappa_\bfG(b))$. In particular, our arguments do not imply that $X_c(b)$ (equivalently, all of its covering spaces $\dot X_{\bar c}(b)$ with $\bar c \in \overline F_c\,/\ker\bar\kappa_c$) is empty. In agreement with this, $X_c(b)$ indeed can be non-empty; \textit{cf.} \cite[Theorem~13.2]{Ivanov_DL_indrep}.
\end{rem}

\section{A variant of Steinberg's cross section}\label{sec:variant_steinberg_crosssection}

By the \emph{Dynkin diagram} $\Dyn(\bfG)$ of the unramified group $\bfG$, we mean the Dynkin diagram of the split group $\bfG_{\breve k}$ endowed with the action induced by the Frobenius $\sigma$. There is a unique decomposition into $\sigma$-stable subdiagrams $\Dyn(\bfG) = \coprod_i \Gamma_i$ such that $\sigma$ cyclically permutes $\pi_0(\Gamma_i)$ for each $i$. For each $i$, let $\Gamma_{i,0}$ be a connected component of $\Gamma_i$, so that $\Gamma_i = \coprod_{j=1}^{d_i}\sigma^j(\Gamma_{i,0})$. Then $(\Gamma_{i,0}, \sigma^{d_i})$ is a \emph{connected} Dynkin diagram of a group over $k_{d_i}$, an unramified extension of $k$ of degree $d_i$.

\begin{Def}\label{def:classical_type}
We say that the unramified group $\bfG$ is \emph{of classical type} if each connected component $(\Gamma_{i,0}, \sigma^{d_i})$ as above of $\Dyn(\bfG)$  is of one of the following types: $A_{n-1}$ ($n\geq 2$), $B_m$ ($m\geq 2$), $C_m$ ($m \geq 2$), $D_m$ ($m \geq 4$), ${}^2A_{n-1}$ ($n \geq 3$) or ${}^2D_m$ ($m \geq 4$).
\end{Def}

Although Theorem~\ref{thm:main_decomposition} and its corollaries essentially hold for all Coxeter elements (\textit{cf.} the discussion preceding Theorem~\ref{thm:main_decomposition}), it will be convenient to work with particular ones. Therefore, we introduce the following auxiliary notion.

\begin{Def}\label{def:special_Coxeter} Let $\bfG$ be of classical type. Let $c \in W$ be a Coxeter element. We call $c$ \emph{special} in the following cases:
\begin{itemize}
\item If $\bfG$ is absolutely almost simple, $c$ is special if it is of the form\footnote{The explicit shape of these elements only becomes important when everything boils down to explicit calculations in Section~\ref{sec:type_by_type}. For the sake of better readability, we postpone the explicit description until the relevant notation is introduced.} given in \eqref{eq:special_Coxeter_element_type_A} (resp.\  \eqref{eq:special_Coxeter_type_C}, \eqref{eq:special_Coxeter_type_B}, \eqref{eq:special_Coxeter_type_D}, \eqref{eq:special_Coxeter_element_type_2A}, \eqref{eq:special_Coxeter_element_type_2D}) if $\bfG$ is of type $A_{n-1}$ (resp.\ $C_m$, $B_m$, $D_m$, ${}^2A_{n-1}$, ${}^2D_m$). 
\item If $\bfG = \Res_{k'/k}\bfG'$ for a finite unramified extension $k'/k$ and an absolutely almost simple $k'$-group $\bfG'$, we may identify the Weyl group $W$ of $\bfG$ with $\prod_{i=1}^{[k':k]} W'$, where $W'$ is the Weyl group of $\bfG'$ and $\sigma$ acts cyclically. Then $c$ is special if it is of the form $(c',1,\dots,1)$, where $c'$ is a special Coxeter element of $W'$. 
\item If $\bfG$ is isogeneous to a product $\prod_{i=1}^m \bfG_i$ of almost simple $k$-groups $\bfG_i$, we may identify $W = \prod_{i=1}^m W_i$, where $W_i$ is the Weyl group of $\bfG_i$. Then $c = (c_i)_{i=1}^m$ is special if all $c_i$ are special.
\end{itemize}
\end{Def}

\subsection{A loop version of Steinberg's cross section}\label{sec:steinberg_crosssection_technical}
For classical groups and special Coxeter elements, we now prove the existence of certain Steinberg cross sections in our setup. This is a Frobenius-twisted loop version of \cite[Proposition~8.9]{Steinberg_65} and \cite[Theorem~3.6 and  Section~3.14]{HeL_12}, which will be an important ingredient in the proof of Theorem~\ref{thm:main_decomposition}. 

Let $c\in W$ be a Coxeter element. If $b \in \bfG(\breve k)$ is any lift of $c$, then $b \sigma$ has a unique fixed point $\bfx = \bfx_b$ on $\caA_{\bfT,\breve k}$. Let $\caG_{\bfx}$ denote the corresponding parahoric $\caO_{\breve k}$-model of $\bfG$. 
We have the twisted Frobenius $\sigma_b = \Ad(b)\circ \sigma$ on $L\bfG,L^+\caG_\bfx$.
Let $({}^c\bfU \cap \bfU)_{\bfx}$ be the closure of ${}^c\bfU \cap \bfU$ in $\caG_{\bfx}$. Similarly, we have the closures $({}^c\bfU \cap \bfU^-)_{\bfx}$ and $({}^c\bfU)_{\bfx}$ of ${}^c\bfU \cap \bfU^-$ and ${}^c\bfU$ in $\caG_{\bfx}$. 

\begin{prop}\label{prop:loop_Steinbergs_twisted_crosssection}
Let $\bfG$ be of classical type and $c \in W$ a special Coxeter element. Then 
\[
\alpha_b \colon L({}^c\bfU \cap \bfU) \times L({}^c\bfU \cap \bfU^-) \longrar L({}^c\bfU), \quad (x,y) \longmapsto x^{-1} y \sigma_b(x) 
\]
is an isomorphism. 
The injectivity of $\alpha_b$ holds for arbitrary elements $c \in W$ which are of minimal length in an elliptic $\sigma$-conjugacy class $($\textit{cf.} \cite{HeL_12}\/$)$.

Furthermore, $\alpha_b$ restricts to an isomorphism
\[
\alpha_{b,\bfx} \colon L^+({}^c\bfU \cap \bfU)_{\bfx} \times L^+({}^c\bfU \cap \bfU^-)_{\bfx} \longrar L^+({}^c\bfU)_{\bfx}.
\]
\end{prop}

\begin{rem}
It looks very plausible that Proposition~\ref{prop:loop_Steinbergs_twisted_crosssection} holds for arbitrary Coxeter elements $c$. It should be possible to prove this by adapting \cite[Section~3.5, Theorem~3.6 and Section~3.14]{HeL_12} to the present setting. Let us make the relation of Proposition~\ref{prop:loop_Steinbergs_twisted_crosssection} with \cite{HeL_12} more clear. With obvious notation (in particular, $c$ is Coxeter and $b$ any lift of it), \cite[Theorem~ 3.6]{HeL_12} says that $\bfU \times (\bfU \cap {}^b \bfU^-) \rar \bfU b \bfU$, $(u,z) \mapsto uz\sigma(u)^{-1}$ is an isomorphism if $\bfG$ is reductive over a finite field with Frobenius morphism $\sigma$ (and, in  fact, in various other setups). As an easy corollary of this, \cite[Section~3.14]{HeL_12} states that $\alpha_b' \colon (\bfU \cap {}^{\sigma^{-1}(c)^{-1}}\bfU) \times (\bfU \cap {}^c\bfU^-) \rar \bfU$, $(x,y)\mapsto xy \sigma_{b}(x)^{-1}$ is an isomorphism in the same setup. Now the finite field version (without the loop functors) of the map $\alpha_b$ of Proposition~\ref{prop:loop_Steinbergs_twisted_crosssection} is obtained from $\alpha_b'$ by a series of changes of variables (\textit{cf.} the first paragraph of the proof of Proposition~\ref{prop:loop_Steinbergs_twisted_crosssection}), so that $\alpha_b$ is an isomorphism if and only if $\alpha_b'$ is.
\end{rem}

\begin{proof}[Proof of Proposition~\ref{prop:loop_Steinbergs_twisted_crosssection}]
We work with the pro-\'etale topology on $\Perf_{\obF}$ and regard $\alpha_b$ as a map of pro-\'etale sheaves. The injectivity of $\alpha_b$ may be derived from \cite{HeL_12}. First, change the variable by setting $y = \sigma_b(y_1)$. It remains to show that $L(\bfU \cap {}^c\bfU) \times L(\bfU \cap {}^{\sigma^{-1}(c^{-1})}\bfU^-) \rar L({}^c\bfU)$, $x,y_1 \mapsto x^{-1}\sigma_b(y_1 x)$ is injective. Compose this map with the isomorphism $(\cdot)^{-1} \circ \sigma^{-1} \circ \Ad(b) \colon L({}^c\bfU) \rar L\bfU$. It remains to show that this composition, given by $x,y_1 \mapsto x^{-1}y_1^{-1}\sigma^{-1}(b^{-1} x b)$, is injective. Finally, change the variables by putting $x = x_2^{-1}$ and $y_1 = y_2^{-1}$, and set $\tilde{b} := \sigma^{-1}(b)^{-1}$ with image $\tilde c$ in $W$. We are reduced to showing that $L(\bfU \cap {}^{\sigma(\tilde{c})^{-1}}\bfU) \times L(\bfU \cap {}^{\tilde c}\bfU^-) \rar L\bfU$, $x_2,y_2 \mapsto x_2y_2\tilde{b} \sigma^{-1}(x_2)^{-1}\tilde{b}^{-1}$ is injective. This follows from \cite[Section~3.14(b)]{HeL_12} applied to $A = \bW(R)[1/\varpi]$ (for any $R \in \Perf_{\obF}$) and $\chi = \phi_R^{-1}$.

We now show the surjectivity of $\alpha_b$. We may assume that $\bfG$ is of adjoint type. Then $\bfG = \prod_{i=1}^r \bfG_i$ with $\bfG_i$ almost simple $k$-groups, and we are immediately reduced to the case that $\bfG$ is almost simple over $k$. 

\begin{lm}\label{lm:crosssection_reduction_abs_almost_simple}
Let $k_d/k$ be an unramified extension of degree $d \geq 1$, and suppose that $\bfG = \Res_{k_d/k} \bfG'$ for a $k_d$-group $\bfG'$. If the surjectivity claim for $\alpha_b$ of Proposition~\ref{prop:loop_Steinbergs_twisted_crosssection} holds for $\bfG'$, then it holds for $\bfG$.
\end{lm}

\begin{proof}
Identify $\bfG_{\breve k} = \prod_{i=1}^d \bfG'_{\breve k}$ and similarly for the Weyl groups $W$, $W'$ and unipotent radicals $\bfU, \bfU'$ of Borels of $\bfG, \bfG'$, so that the Frobenius $\sigma$ permutes cyclically the components. As $c$ is special, $c = (c',1,\dots,1)$ for a special Coxeter element $c' \in W'$.  
We then have ${}^c\bfU\cap\bfU = ({}^{c'}\bfU' \cap \bfU') \times \bfU' \times \dots \times \bfU'$ and ${}^c\bfU \cap \bfU^- = ({}^{c'}\bfU' \cap \bfU^{\prime -}) \times 1 \times \dots \times 1$. The lift $b$ of $c$ is then given by $(b',\tau_2,\dots,\tau_d)$ for a lift $b'$ of $c'$ to $\bfG'(\breve k)$ and some $\tau_i \in \bfT'(\breve k)$. Then the map $\alpha_b$ from Proposition~\ref{prop:loop_Steinbergs_twisted_crosssection} is given by
\[(x_i)_{i=1}^d, y_1 \longmapsto (x_1^{-1}y_1 \Ad(b')(\sigma'(x_d)), x_2^{-1}\Ad(\tau_2)(x_1), \dots x_d^{-1}\Ad(\tau_2)(x_{d-1})). \]
From this, surjectivity of $\alpha_b$ for $\bfG$ immediately reduces to surjectivity of $\alpha_{b'}$ for $\bfG'$.
\end{proof}

\begin{lm}\label{lm:conjugation_crosssecion_proof}
If $\widetilde\alpha_b \colon L(\bfU \, \cap \, {}^{c^{-1}}\bfU) \times L(\bfU \,\cap\, {}^{c^{-1}}\bfU^-) \rar L\bfU$ denotes the map $\widetilde\alpha_b(x,y) = x^{-1}y\sigma(b x b^{-1})$, then  $\Ad(b^{-1})(\alpha_b(x,y)) = \widetilde\alpha_b(\Ad(b^{-1})(x), \Ad(b^{-1})(y))$.
\end{lm}

\begin{proof}
This follows immediately from a  computation. 
\end{proof}

As $\bfG$ splits over $\breve k$, Lemma~\ref{lm:crosssection_reduction_abs_almost_simple} reduces us to the case that $\bfG$ is absolutely almost simple over $k$.
It suffices to show that the map $\widetilde\alpha_b$ from Lemma~\ref{lm:conjugation_crosssecion_proof} is surjective. Write $A = L(\bfU \cap {}^{c^{-1}}\bfU)$, $B = L(\bfU \cap {}^{c^{-1}}\bfU^-)$, $C = L\bfU$. For $\star \in \{A,B,C\}$, write $\Phi_\star = \{\alpha \in \Phi \colon L\bfU_\alpha \subseteq \star \}$, where $\bfU_\alpha \subseteq \bfG$ denotes the root subgroup corresponding to~$\alpha$.
We call a subset $\Psi \subseteq \Phi$ \emph{closed under addition} if for any $\alpha, \beta \in \Psi$ with $\alpha + \beta \in \Phi$, we have $\alpha + \beta \in \Psi$.

\begin{lm}\label{lm:filtration_on_roots_Coxeter}
There exist a positive integer $r$ and subsets
\[
\Phi_C = \Psi_1 \supseteq \Psi_2 \supseteq \dots \supseteq \Psi_r = \Phi_B 
\]
satisfying the following conditions:
\begin{enumerate}
\item\label{forC-1} For all $1\leq i \leq r$, $\Psi_i$ and $\Psi_i \sm \Psi_r$ are closed under addition.
\item\label{forC-2} For all $1 \leq i < r$ and all $\alpha, \beta \in \Psi_i$, the implication $\alpha + \beta \in \Phi \Rar \alpha + \beta \in \Psi_{i+1}$ holds.
\item\label{forC-3} For all $1 \leq i \leq r$, $\sigma(c(\Psi_i \sm \Psi_r)) \subseteq \Psi_i$.
\end{enumerate}
\end{lm}

\begin{proof} 
We give the desired partitions in Section~\ref{sec:proof_of_filtration_lm_type_An1} (type $A_{n-1}$), Section~\ref{sec:proof_of_crosssec_lemma_type_CB} (types $B_m,C_m$), Section~\ref{sec:proof_of_crosssec_lemma_type_D} (type $D_m$), Section~\ref{sec:proof_of_crosssec_lemma_type_2An1} (type ${}^2A_{n-1}$) and Section~\ref{sec:proof_of_crosssec_lemma_type_2Dm} (type ${}^2D_m$) after the relevant notation is introduced. In each case, it is straightforward to check that the conditions of the lemma hold for the given partition;  we omit this checking.
\end{proof}

\begin{lm}\label{lm:loop_twisted_Lang}
Let $I$ be a finite set not containing $0$. Let $\lambda \colon I \rar I \cup \{0\}$ be a map such that $\lambda^{-1}(i)$ has at most one element if\, $i \in I$. For any $i \in I \cap \lambda(I)$, let $\varepsilon_i \in \breve k^\times$. Consider the map $\delta \colon \prod_{i \in I}L\bG_a \rar \prod_{i \in I}L\bG_a$ given by $\delta((a_i)_{i\in I}) = (b_i)_{i \in I}$, where $b_i = 0$ if $\lambda^{-1}(i) = \varnothing$ and $b_i = \varepsilon_i \phi(a_{\lambda^{-1}(i)})$ otherwise. Then the map $x \mapsto x - \delta(x)$ is surjective for the pro-\'etale topology.
\end{lm}

\begin{proof}
There are some integers $s\geq r\geq 0$ and a disjoint decomposition $I = \coprod_{i=1}^r I_i \dot\cup \coprod_{i=r+1}^sI_i$ such that for $1\leq i \leq r$, $\lambda$ restricts to a map $I_i \rar I_i$ which is a full cycle of length $\#I_i$ (first type), and for $r+1 \leq i \leq s$, $\lambda$ restricts to a map $I_i \rar I_i \dot\cup \{0\}$, which is, after choosing an isomorphism $I_i \cong \{1,\dots,m\}$, given by $\lambda(j) = j-1$ for all $j \in I_i$ (second type). It suffices to show that the restriction of $\delta$ to $\prod_{j \in I_i} L\bG_a$ is surjective. Thus we may assume that $I$ is one of the $I_i$. If $I = \{1,\dots,m\}$ is of the second type, one has the filtration $G^\bullet$ of $G = \prod_{i=1}^m L\bG_a$ by closed subschemes (ind-group schemes) $G_j = \prod_{i=1}^j L\bG_a$, stable by $\delta$, such that $x\mapsto x-\delta(x)$ induces the identity on the associated graded object $\gr\,G^\bullet$. From this, the result easily follows. If $I$ is of the first type, we may assume $I = \bZ/m\bZ$ and $\lambda(i) = i+1$. Replacing $\phi$ with $\phi^m$, one is then immediately reduced to the case that $m=1$ and $\delta \colon L\bG_a \rar L\bG_a$, $\delta(x) =           \varepsilon\phi(x)$ for some $\varepsilon \in \breve k^\times$. Then Lemma~\ref{lm:Lang_for_LGa} finishes the proof.
\end{proof}

\begin{lm}\label{lm:Lang_for_LGa}
Let $\varepsilon \in \breve k$. Then the map $\id - \varepsilon\phi \colon L\bG_a \rar L\bG_a$, $x \mapsto x-\varepsilon\phi(x)$ is surjective for the pro-\'etale topology.
\end{lm}

\begin{proof}
If $\varepsilon \not\in \caO_{\breve k}^\times$, then $\id - \varepsilon\phi$ is an isomorphism. Indeed, if $\ord_{\varpi}(\varepsilon) > 0$, $\sum_{i=0}^\infty (\varepsilon\phi)^i$ is a meaningful endomorphism of $L\bG_a$ (as $\bW(R)$ is $\varpi$-adically complete for each $R \in \Perf$) and the inverse of $\id - \varepsilon\phi$. If $\ord_\varpi(\varepsilon) < 0$, a similar argument goes through using $\id - \varepsilon\phi = -\varepsilon\phi (\id - (\varepsilon\phi)^{-1})$. Now assume $\varepsilon \in \caO_{\breve k}^\times$. We have $L\bG_a = \colim_{m \rar -\infty}\varpi^m L^+\bG_a$, and it suffices to show that $\id - \varepsilon \phi$ is surjective on $L^+\bG_a = \prolim_{r \rar \infty} L_r^+\bG_a$. This follows from a pro-version of Lang's theorem. \qedhere
\end{proof}

We now prove the surjectivity of $\alpha_b$. For a subset $\Psi \subseteq \Phi^+$ closed under addition, let $\bfU_\Psi$ denote the unique closed subgroup of $\bfU$ whose $\breve k$-points are generated by $\{\bfU_\alpha(\breve k) \}_{\alpha \in \Psi}$. As a $\breve k$-scheme, $\bfU_\Psi$ is isomorphic to $\prod_{\alpha \in \Psi} \bfU_\alpha$. Write $U_\Psi = L\bfU_\Psi$. Let $\{\Psi_i\}_{i=1}^r$ be a filtration as in Lemma~\ref{lm:filtration_on_roots_Coxeter}. For $1\leq i \leq r$, we have the closed subschemes (ind-group schemes) $C_i := U_{\Psi_i}$ of $L\bfU$, with $C_1 = C$, $C_r = B$. Moreover, condition~\eqref{forC-2} from Lemma~\ref{lm:filtration_on_roots_Coxeter} guarantees that $C_{i+1} \subseteq C_i$ is a closed normal subgroup, with abelian quotient. As $H^1(k,\bfH) = 0$ for a split unipotent $k$-group $\bfH$, we have $C_i/C_{i + 1} \cong \prod_{\alpha \in \Psi_i \sm \Psi_{i+1}} L\bfU_\alpha$. Moreover, by condition~\eqref{forC-1}, for $1\leq i\leq r-1$, $C_i$ has the closed subgroup $A_i = U_{\Psi_i \sm \Psi_r}$, and condition~\eqref{forC-2} implies that $A_{i+1}$ is normal in $A_i$ and that the inclusion $A_i \har C_i$ induces an isomorphism $\gamma_i \colon A_i/A_{i+1} \isor C_i/C_{i+1}$. Moreover, by condition~\eqref{forC-3}, the map $\sigma \circ \Ad(b) \colon A_i \rar C_i$ induces a map $\delta_i \colon A_i/A_{i+1} \rar C_i/C_{i+1}$. The composition $\gamma_i^{-1} \delta_i$ is an endomorphism of $A_i/A_{i+1} \cong \prod_{\alpha \in \Psi_i \sm \Psi_{i+1}}L\bfU_\alpha$, which is of the form given in Lemma~\ref{lm:loop_twisted_Lang}. From Lemma~\ref{lm:loop_twisted_Lang} it follows that $a \mapsto a\cdot (\gamma_i^{-1} \delta_i)(a)^{-1}$ is a surjective endomorphism of $A_i/A_{i+1}$.

Let $R \in \Perf_{\obF}$. Let $x_i \in C_i(R)$, and let $\bar x_i$ be its image in $(C_i/C_{i+1})(R)$. By the above, after replacing $R$ with a pro-\'etale cover, we may find some $a_i \in A_i(R)$ with image $\bar a_i \in (A_i/A_{i+1})(R)$ such that $\bar a_i \cdot (\gamma_i^{-1}\delta_i)(\bar a_i^{-1}) = \gamma_i^{-1}(\bar x_i)$. Applying $\gamma_i$, we deduce the equation $\gamma_i(\bar a_i)^{-1}\bar x_i\delta_i(\bar a_i) = 1$ in $(C_i/C_{i+1})(R)$. The left-hand side is precisely the image of $\widetilde\alpha_{b}(a_i,x_i) \in C_i(R)$ in $(C_i/C_{i+1})(R)$. In other words, for each $x_i \in C_i(R)$, we may find (after possibly enlarging $R$) some $a_i \in A_i$ such that $\widetilde\alpha_{b}(a_i,x_i) \in C_{i+1}$. Starting with an arbitrary $x = x_1 \in C(R)$ and applying this procedure consequently for $i=1,2,\dots,r-1$, we deduce that after possibly enlarging $R$, there are some $a_1,\dots, a_{r-1} \in A(R)$ such that for $a = a_{r-1} a_{r-2}\dots a_1$, we have $\widetilde \alpha_{b}(a,x) \in C_r(R) = B(R)$. (Here we use that $\widetilde\alpha_b(b,\widetilde\alpha_{b}(a,x)) = \widetilde\alpha_b(ab,x)$.) This finishes the proof that $\alpha_{b}$ is an isomorphism.


Next we turn to $\alpha_{b,\bfx}$. Note that it is well defined as for $x \in L^+({}^c\bfU \cap \bfU)_{\bfx}$, $\sigma_b(x) \in L({}^c\bfU) \cap L^+\caG_{\bfx} = L^+({}^c\bfU_{\bfx})$, the intersection taken in $L\bfG$. Being the restriction of $\alpha_{b}$, it is injective, as $\alpha_{b}$ is, and it remains to prove its surjectivity. Applying $\Ad(b^{-1})$, which conjugates $\caG_{\bfx}$ to $\caG_{\bfy}$, where $\bfy = b^{-1}(\bfx)$, and the analogue of Lemma~\ref{lm:conjugation_crosssecion_proof}, it suffices to show that (with the obvious notation) the map $\widetilde\alpha_{b,\bfx} \colon L^+(\bfU \cap {}^{c^{-1}}\bfU)_{\bfy} \times L^+(\bfU \cap ^{c^{-1}}\bfU^-)_{\bfy} \rar L^+\bfU_{\bfy}$, $\widetilde\alpha_{b}(x,y) = x^{-1}y\sigma(b x b^{-1})$ is surjective. By the same reasoning as in Lemma~\ref{lm:crosssection_reduction_abs_almost_simple} (and the paragraph preceding it), we are reduced to the case that $\bfG$ is absolutely almost simple over $k$. For $\Psi \subseteq \Phi^+$, let $\bfU_{\Psi,\bfy}$ denote the closure of $\bfU_{\Psi}$ in $\caG_{\bfy}$. We have $\bfU_{\Psi,\bfy} \cong \prod_{\alpha \in \Psi} \bfU_{\alpha, \bfy}$. Replacing $C_i,B,A_i$ above with $C_i^+ = L^+\bfU_{\Psi_i,\bfy}$, $B^+ = L^+(\bfU \cap {}^{c^{-1}}\bfU^-)$, $A_i^+ = L^+\bfU_{\Psi_i \sm \Psi_r,\bfy}$ and using the integral version of Lemma~\ref{lm:loop_twisted_Lang} (with essentially the same proof), we can run the above inductive argument for the quotients $A_i^+/A_{i+1}^+$, deducing the surjectivity of $\alpha_{b,\bfx}$.
\end{proof}

\subsection{A variant of a Deligne--Lusztig space}\label{sec:various_reformulations_DL_spaces}

Throughout this section we work with the arc-topology on $\Perf_{\obF}$. Steinberg's cross section allows us to give a different presentation of the Deligne--Lusztig spaces $X_c(b)$, $\dot X_{\bar c}(b)$. We have the following general definition. 

\begin{Def}\label{def:alternative_spaces_Xwb} Let $w \in W$, $\dot w \in F_w(\breve k)$ and $b \in \bfG(\breve k)$. Define $\dot X_{\dot w,b}$ and $\tilde X_{w,b}$ by the following cartesian diagrams:

\centerline{\begin{tabular}{cc}
\begin{minipage}{2in}
\begin{displaymath}
\leftline{
\xymatrix{
\dot X_{\dot w,b} \ar[r] \ar@{^(->}[d] & L({}^w\bfU \cap \bfU^-)\cdot \dot w b^{-1} \ar@{^(->}[d] \\
L\bfG \ar[r]^{g \mapsto g^{-1}\sigma_b(g)} & L\bfG
}
}
\end{displaymath}
\end{minipage}
& \qquad\qquad  and \qquad\qquad
\begin{minipage}{2in}
\begin{displaymath}
\leftline{
\xymatrix{
\tilde X_{w,b} \ar[r] \ar@{^(->}[d] & L({}^w\bfU \cap \bfU^-)\cdot L\bfT\ar@{^(->}[d] \\
L\bfG \ar[r]^{g \mapsto g^{-1}\sigma_b(g)} & L\bfG\rlap{,}
}
}
\end{displaymath}
\end{minipage}
\end{tabular}
}
\smallskip
\noindent and let $X_{w,b}:= \tilde X_{w,b}/L\bfT$.
\end{Def}

\begin{rem}\label{rem:definition_of_tilde_space_Xwb}
The map $g\mapsto g^{-1}\sigma_b(g) \colon L\bfG \rar L\bfG$ factors through $\kappa_\bfG^{-1}(\kappa_\bfG(b)) \subseteq L\bfG$. Thus the map $\tilde X_{c,b} \rar ({}^c\bfU\cap \bfU^-)\cdot L\bfT$ factors through $({}^c\bfU\cap \bfU^-)\cdot \kappa_{\bfG}|_{L\bfT}^{-1}(\kappa_{\bfG}(b))$. So, in the upper right entry of the right diagram of Definition~\ref{def:alternative_spaces_Xwb}, we could have written $({}^c\bfU\cap \bfU^-)\cdot \kappa_{\bfG}|_{L\bfT}^{-1}(\kappa_{\bfG}(b))$ without changing the definition.
\end{rem}

The group $\bfG_b(k) \times \bfT_w(k)$ acts on $\dot X_{\dot w,b}$ by $g,t\colon x \mapsto gxt$, and $\bfG_b(k)$ acts on $X_{w,b}$ by $g \colon x \mapsto gx$. We only use these spaces in the case that $w = c$ is a (special) Coxeter element.

\begin{prop}\label{prop:isom_dotXbw_Xdoubleprime}
Let $b \in \bfG(\breve k)$. Let $c \in W$ be a Coxeter element and $\dot c \in F_c(\breve k)$. Whenever the conclusion of Proposition~\ref{prop:loop_Steinbergs_twisted_crosssection} holds $($\textit{e.g.}, $\bfG$ classical, $c$ special Coxeter$)$, the following hold:  
\begin{enumerate}
\item\label{idX-1} There is a $\bfG_b(k) \times \bfT_c(k)$-equivariant isomorphism $\dot X_{\dot c}(b) \cong \dot X_{\dot c,b}$.

\item\label{idX-2} If\, $b \in F_c(\breve k)$,
then there is a $\bfG_b(k)$-equivariant isomorphism $X_c(b) \cong X_{c,b}$. 
\end{enumerate}
\end{prop}

\begin{proof} 
\eqref{idX-1} 
Consider the sheafification $\dot X'$ of the functor on $\Perf_{\obF}$,
\[
X_1 \colon R \longmapsto \left\{g \in L\bfG(R) \colon g^{-1}\sigma_b(g) \in L({}^c\bfU)(R) \dot c b^{-1} = \dot c L\bfU(R)b^{-1} \right\}/ L(\bfU \cap {}^{c}\bfU)(R). 
\]
By \cite[Proposition~12.1]{Ivanov_DL_indrep} $X_{\dot c}(b)$ and $\dot X'$ are equivariantly isomorphic.
It remains to show that $\dot X' \cong \dot X_{\dot c,b}$ equivariantly. Note that $\dot X_{\dot c,b}$ is the sheafification of 
\[
X_2 \colon R \longmapsto \left\{ g \in L\bfG(R) \colon g^{-1}\sigma_b(g) \in L({}^c\bfU \cap \bfU^-)(R) \cdot \dot cb^{-1} \right\}.
\]
There is an obvious equivariant map $X_2 \rar X_1$, of which we claim that it is an isomorphism. This follows once we check that
\[
L({}^c\bfU \cap \bfU) \times L({}^c\bfU \cap \bfU^-) \cdot \dot c b^{-1} \longrar \dot c L\bfU b^{-1} = L({}^c\bfU) \cdot \dot c b^{-1}, \quad x,y' \longmapsto x^{-1}y'\sigma_b(x)
\]
is an isomorphism. Substituting $y' = y \dot c b^{-1}$, with $y$ ranging in $L({}^c \bfU \cap \bfU^-)$, this map becomes $(x,y) \mapsto \left(x^{-1}y\dot c \sigma(x) \dot c^{-1}\right) \cdot \dot c b^{-1}$, which is an isomorphism by Proposition~\ref{prop:loop_Steinbergs_twisted_crosssection}. The proof of~\eqref{idX-2} is completely analogous.
\end{proof}

\begin{prop}\label{prop:Xwb_other_description}
Suppose $\bfG$ is of adjoint type. Let $c \in W$ be as in Proposition~\ref{prop:isom_dotXbw_Xdoubleprime} and $b \in F_c(\breve k)$. For any $\dot c \in F_c(\breve k)$ satisfying $\kappa_\bfG(b) = \kappa_\bfG(\dot c)$, there are $\bfG_b(k)$-equivariant isomorphisms $X_c(b) \cong X_{c,b} \cong \dot X_{\dot c,b}\,/\,\bfT_c(k)$.
\end{prop}

\begin{proof}
  By Proposition~\ref{prop:isom_dotXbw_Xdoubleprime}, it suffices to show that $\dot X_{\dot c,b} \,/ \bfT_c(k) \cong X_{c,b}$. Writing
  \[\tau := \dot c^{-1}b \in (L\bfT \cap \kappa_\bfG^{-1}(\kappa_\bfG(b)))(\obF),\]
  these sheaves are the sheafifications of the quotient presheaves
\begin{align*}
Y_1 \colon R &\longmapsto \{g \in L\bfG(R) \colon g^{-1}\sigma_b(g) \in L(\bfU^- \cap {}^c\bfU)(R) \cdot \tau  \}\,/\, \bfT_c(k)(R) \qquad \quad \text{and}\\
Y_2 \colon R &\longmapsto \{g \in L\bfG(R) \colon g^{-1}\sigma_b(g) \in L(\bfU^- \cap {}^c\bfU)(R)\cdot L\bfT(R) \}\,/ L\bfT(R),
\end{align*}
respectively. We have an obvious map $Y_1 \rar Y_2$, and we claim its sheafification is an isomorphism. First, the map between presheaves itself is injective (which already implies that its sheafification is injective, as quotient presheaves are separated). Indeed, let $g,h \in Y_1(R)$. Then $\sigma_b(g) = g u_g \tau$, $\sigma_b(h) = hu_h\tau$ for some $u_g,u_h \in L(\bfU^- \cap {}^c\bfU)(R)$. Suppose their images in $Y_2(R)$ agree, \textit{i.e.}, there is some $t \in L\bfT(R)$ such that $gt=h$. Then $\sigma_b(g)\sigma_b(t) = \sigma_b(h)$, and hence $gu_g\tau \sigma_b(t)= hu_h\tau$; \textit{i.e.}, 
\[
\sigma_c(t) = \sigma_b(t) = \tau^{-1} u_g^{-1} t u_h \tau
\]
(the first equality holds as $b \in F_c(\breve k)$). Thus, $\sigma_c(t) = t u$ for some $u \in L(\bfU^- \cap {}^c\bfU)(R)$, which implies $u = 1$, and hence $t \in (L\bfT)^{\sigma_c}(R) = \bfT_c(k)(R)$. 

Secondly, let $g \in L\bfG(R)$ satisfy $\sigma_b(g) = g u_g \tau'$ for some $u \in ({}^c\bfU\cap\bfU^-)(R)$ and $\tau' \in L\bfT(R)$. By Remark~\ref{rem:definition_of_tilde_space_Xwb}, $\kappa_\bfG(\tau') = \kappa_\bfG(b)$. Changing $g$ to $gt$ for some $t \in L\bfT(R)$ has no effect on the image of $g$ in $Y_2(R)$ but amounts to replacing $\tau'$ with $\tau' t^{-1}\sigma(t)$. Thus, to show that $Y_1 \rar Y_2$ is surjective when passing to sheafification, it remains to check that the $\sigma_c$-twisted action morphism 
\[ L\bfT \times (L\bfT \cap \kappa_\bfG^{-1}(\kappa_\bfG(b))) \longrar L\bfT \cap \kappa_\bfG^{-1}(\kappa_\bfG(b)), \quad t,\tau' \longmapsto t^{-1}\tau'\sigma_c(t)
\]
is surjective for the arc-topology. In fact, it is even surjective for the pro-\'etale topology. Indeed, the space $L\bfT \cap \kappa_\bfG^{-1}(\kappa_\bfG(b))$ is fibered over $(\bar\kappa_\bfG|_{X_\ast(\bfT)})^{-1}(\kappa_\bfG(b))$, all fibers being isomorphic to $L^+\caT$. Using pro-Lang for the $\sigma_c$-stable connected subscheme $L^+\caT \subseteq L\bfT$, we are reduced to showing that the map of discrete abelian groups $X_\ast(\bfT) \times (\bar\kappa_\bfG|_{X_\ast(\bfT)})^{-1}(\kappa_\bfG(b)) \rar (\bar\kappa_\bfG|_{X_\ast(\bfT)})^{-1}(\bar\kappa_\bfG(b))$ given $\mu,\chi \mapsto \chi + \sigma_c(\mu) - \mu$ is surjective, \textit{i.e.}, that the fibers of $\bar\kappa_\bfG|_{X_\ast(\bfT)}$ are torsors under $\im(\sigma_c - \id \colon X_\ast(\bfT) \rar X_\ast(\bfT)) = \ker\bar\kappa_c$. This holds by Remark~\ref{rem:beta_w_0_adjoint_elliptic} as $\bfG$ is adjoint and $c$ Coxeter.
\qedhere
\end{proof}

\section{Decomposition into integral \texorpdfstring{$\boldsymbol{p}$}{p}-adic spaces}\label{sec:pDL_on_Coxeter_case}

In this section we begin with the proof of Theorem~\ref{thm:main_decomposition}. Before that, we show that this theorem is indeed sufficiently general to describe \emph{all} spaces $X_c(b)$ and $\dot X_{\bar c}(b)$ for an unramified classical group $\bfG$, $c$ Coxeter and $b$ basic.

\begin{prop}\label{prop:simplifying_assumptions}
  Let $b \in \bfG(\breve k)$ be basic, $c \in W$ Coxeter and $\bar c \in \overline F_c\,/\ker\bar\kappa_c$. Suppose $\dot X_{\bar c}(b) \neq \varnothing$. 
Then there exist
\begin{itemize}
\item a special Coxeter element $c_1 \in W$ $($\textit{cf.} Definition~\ref{def:special_Coxeter}\,$)$, 
\item a lift $\bar c_1 \in \overline F_{c_1}/\ker\bar\kappa_{c_1}$ and 
\item an element $b_1 \in [b] \cap F_c(\breve k)$ which lies over $\bar c_1$
\end{itemize}
such that $X_c(b) \cong X_{c_1}(b_1)$ and $\dot X_{\bar c}(b) \cong \dot X_{\bar c_1}(b_1)$, $\bfG_b(k)$- $($resp.\ $\bfG_b(k) \times \bfT_c(k)$-$)$equivariantly.
\end{prop}

\begin{proof}
By \cite[Corollary~8.18]{Ivanov_DL_indrep} there exists a special Coxeter element $c_1$ such that $X_c(b) \cong X_{c_1}(b)$ equivariantly. Moreover, we may find some $\bar c_1$ over $c_1$ such that $\dot X_{\bar c}(b)$ is equivariantly isomorphic to $\dot X_{\bar c_1}(b)$; \textit{cf.} \cite[Lemma~8.16]{Ivanov_DL_indrep}. By Corollary~\ref{cor:choice_b_cox_over_barw} we may pick $b_1 \in [b] \cap F_{c_1}(\breve k)$ lying over $\bar c_1$.
\end{proof}

The rest of the article is devoted to the proof of Theorem~\ref{thm:main_decomposition}. Thus we start with
\begin{itemize}
\item a special Coxeter element $c \in W$ and $\bar c \in \overline F_c\,/\ker\bar\kappa_c$,
\item a basic element $b$ which lies in $F_c(\breve k)$ and lifts $\bar c$.
\end{itemize}

In Section~\ref{sec:reduction_of_main_1_to_adjoint} we reduce \eqref{eq:main_iso_1} to the case that $\bfG$ is absolutely almost simple and adjoint. In Section~\ref{sec:deco} we then prove \eqref{eq:main_iso_1} by a $v$-descent argument, assuming the technical Proposition~\ref{prop:Xbb_XbbO_equal} (which essentially is \eqref{eq:main_iso_2} for absolutely almost simple adjoint groups, \textit{cf.} Corollary~\ref{cor:disj_dec_on_U_level}). Then it remains to deduce \eqref{eq:main_iso_2} from \eqref{eq:main_iso_1}. This is done in Section~\ref{sec:proof_of_main_iso_2}. Finally, in Section~\ref{sec:type_by_type} we prove the remaining Proposition~\ref{prop:Xbb_XbbO_equal}, which is the most technical part of the proof. We also point out that specialness of $c$ is only used in the proof of Proposition~\ref{prop:Xbb_XbbO_equal} and nowhere else in the proof of Theorem~\ref{thm:main_decomposition}.

\subsection{Reduction of {\bf{\eqref{eq:main_iso_1}}} to the adjoint absolutely almost simple case}\label{sec:reduction_of_main_1_to_adjoint}

First, note that both sides of the equality claimed in the theorem only depend on $\bfG^{\ad}$ (for the left-hand side, use that $\bfG/\bfB = \bfG^{\ad}/\bfB^{\ad}$). Hence we may assume that $\bfG$ is of adjoint type. Then $\bfG$ decomposes as a direct product of almost simple $k$-groups, and both sides of the claimed isomorphism \eqref{eq:main_iso_1} do so accordingly. Thus we are reduced to the case that $\bfG$ is  of adjoint type and almost simple.

If $\bfG$ is adjoint and almost simple, then $\bfG = \Res_{k'/k}\bfG'$ for some finite unramified extension $k'/k$ and some absolutely almost simple $k'$-group $\bfG'$. As $B(\bfG) = B(\bfG')$, \textit{cf.} \cite[1.10]{Kottwitz_85}, (which allows us to change $b$ appropriately inside its $\sigma$-conjugacy class), Lemma~\ref{lm:DL_res_of_schalars} reduces us to the case that $\bfG$ is adjoint absolutely almost simple.

We state and prove Lemma~\ref{lm:DL_res_of_schalars} in bigger generality than just for adjoint absolutely almost simple groups and Coxeter elements. Let us fix notation first. Let $k'/k$ be a finite unramified extension of degree $d$, let $\bfH'$ be an unramified reductive group over $k'$, with maximal quasi-split torus $\bfT'$. Let $\bfH = \Res_{k'/k}\bfH'$. Let $W'$ and $W \cong \prod_{i=1}^d W'$ be the Weyl groups of $(\bfH', \bfT')$ and $(\bfH,\Res_{k'/k}{\bf\bfT})$. Identifying the buildings of $\bfH$ over $k$ and of $\bfH'$ over $k'$, any point $\bfx \in \caA_{\bfT,\breve k}$ defines a parahoric model $\caH'_{\bfx}$ of $\bfH'$, and we have $\caH_{\bf x} = \Res_{\caO_{k'}/\caO_k} \caH'_{\bfx}$; \textit{cf.} \cite[Proposition~4.7]{Haines_Richarz_20}. Note that for definition \eqref{eq:integral_Coxeter_Stcsss_space} to work, we only need that the affine transformation $b\sigma$ fixes $\bfx$ (assumptions on $c$ to be Coxeter and on $b$ to be basic are redundant), so part~\eqref{Dros-2} of the following lemma makes sense.

\begin{lm}\label{lm:DL_res_of_schalars}
  Let $w' \in W'$, $b' \in \bfH'(\breve k)$, and let $w = (w',1,\dots,1) \in W$, $b = (b',1,\dots,1) \in \bfH(\breve k)$.
\begin{enumerate}
\item\label{Dros-1} We have $X_w^{\bfH}(b) \cong X_{w'}^{\bfH'}(b')$. 
\item\label{Dros-2} If\, $b' \in F_{w'}(\breve k)$ is such that $b'\sigma'$ fixes $\bfx$, then $\dot X_{\bar w,b}^{\caG_x} \cong \dot X_{\bar w',b'}^{\caG'_x}$ and $X_{w,b}^{\caG_x} \cong X_{w',b'}^{\caG'_x}$, where $\bar w$ $($resp.\ $\bar w')$ is the image of $b$ $($resp.\ $b')$ in $\overline F_w\,/\ker\bar\kappa_w$ $($resp.\ $\overline F_{w'}/\ker\bar\kappa_{w'})$,
\end{enumerate}
These natural isomorphisms are compatible with the natural identifications $\bfH(k) = \bfH'(k')$ and $\caH(\caO_k) = \caH'(\caO_{k'})$.
\end{lm}

\begin{proof}
  \eqref{Dros-1} Let $\bfB'$ be a $k'$-rational Borel subgroup of $\bfH'$ which contains $\bfT'$, and let $\bfB = \Res_{k'/k}\bfB'$ be the corresponding Borel subgroup of $\bfH$. As $\breve k$-groups, we have $\bfH \cong \prod_{i=1}^d\bfH'$, with Frobenius permuting the components cyclically. As $L(\cdot)$ commutes with products, we have $L(\bfH/\bfB) = \prod_{i=1}^d L(\bfH'/\bfB')$. Let $g = (g_1,\dots,g_d)\in L(\bfG/\bfB)$. Then $g \in X_w(b)$ if and only if the relative position of $(g_1,\dots,g_d)$ and $(\sigma'(g_d),g_1,\dots,g_{d-1})$ is $(w',1,\dots,1)$, where $\sigma'$ is the geometric Frobenius on $L(\bfG'/\bfB')$. This hold if and only if $g_1 = g_2 = \dots = g_d$ and $g_1 \in X_{w'}^{\bfH'}(b)$. Thus $g \mapsto g_1$ defines the claimed isomorphism.

  \eqref{Dros-2} Just as above, we have $\caH_{\bfx} \cong \prod_{i=1}^d \caH_{\bfx}'$. The same holds after applying $L^+(\cdot)$. Moreover, ${}^w\bfU \cap \bfU^- = ({}^{w'}\bfU' \cap \bfU^{\prime -}, 1, \dots ,1)$. Thus, $g = (g_1,\dots,g_d) \in L^+\caH_{\bfx}$ lies in $\dot X_{w,b}^{\caH_{\bfx}}$ if and only if 
\[ 
g^{-1}\sigma_b(g) = (g_1^{-1}\Ad(b')(\sigma'(g_d)), g_2^{-1}g_1, \dots, g_d^{-1}g_{d-1})) \in (L^+({}^{w'}\bfU' \cap \bfU^{\prime -}), 1, \dots ,1),
\]
which is the case if and only if $g_1 \in \dot X_{w,b}^{\caH_{\bfx}'}$ and the $g_i$ ($i=2,\dots,d$) are determined by appropriate formulas in terms of $g_1$. Thus $g\mapsto g_1$ defines the required isomorphism. The last isomorphism follows from the first.
\end{proof}

Altogether, in the proof of \eqref{eq:main_iso_1} we may assume $\bfG$ to be absolutely almost simple and adjoint.

\subsection{Disjoint decomposition}\label{sec:deco} Here we reduce \eqref{eq:main_iso_1} to Proposition~\ref{prop:Xbb_XbbO_equal}.
Using Proposition~\ref{prop:isom_dotXbw_Xdoubleprime} we may replace the left-hand side in \eqref{eq:main_iso_1}, resp.\ \eqref{eq:main_iso_2}, by $X_{c,b}$, resp.\ $\dot X_{b,b}$ (as $b$ lies over $\bar c$). Consider the following (auxiliary) subsheaf of $\dot X_{b,b}$:
\begin{center}
\begin{displaymath}
\xymatrix{
\dot X_{b,b,\caO} \ar[r] \ar@{^(->}[d] & L^+({}^c\bfU \cap \bfU^-)_{\bfx} \ar@{^(->}[d] \\
L\bfG \ar[r]^{g \mapsto g^{-1}\sigma_b(g)} & L\bfG\rlap{.}
}
\end{displaymath}
\end{center}

We then have the following result, whose proof occupies Section~\ref{sec:type_by_type}.

\begin{prop}\label{prop:Xbb_XbbO_equal} Suppose $\bfG$ is absolutely almost simple and adjoint. We have $\dot X_{b,b} = \dot X_{b,b,\caO}$.
\end{prop}

\begin{cor}\label{cor:disj_dec_on_U_level} Suppose $\bfG$ is absolutely almost simple and adjoint. 
We have a $\bfG_b(k)\times \bfT_c(k)$-equivariant isomorphism 
\[ 
\dot X_{b,b} = \coprod_{\gamma \in \bfG_b(k)/\caG_{\bfx,b}(\caO_k)} \gamma \dot X_{\bar c,b}^{\caG_{\bfx}},
\] 
where $\dot X_{\bar c,b}^{\caG_{\bfx}}$ is as in \eqref{eq:integral_Coxeter_Stcsss_space}. 
\end{cor}

\begin{proof}
Consider the composed map
\begin{equation}\label{eq:Xbb_to_AFM}
\dot X_{b,b} \longhar L\bfG \longrar L\bfG/L^+\caG_{\bfx}, 
\end{equation}
where the second map is the natural projection. By Proposition~\ref{prop:Xbb_XbbO_equal}, for any $R \in \Perf_{\obF}$ and any $g \in \dot X_{b,b}(R)$, we have $\sigma_b(g) \in g L^+({}^c\bfU \cap \bfU)_{\bfx}(R) \subseteq gL^+\caG_{\bfx}(R)$. Now $L^+\caG_{\bfx} \subseteq L\bfG$ is a $\sigma_b$-stable subsheaf; thus $\sigma_b$ acts on $L\bfG/L^+\caG_{\bfx}$. It follows that \eqref{eq:Xbb_to_AFM} factors through $(L\bfG/L^+\caG_{\bfx})^{\sigma_b} = \bfG_b(k)/\caG_{\bfx,b}(\caO_k)$, which is a discrete set. The preimage of $1\cdot \caG_{\bfx}(\caO_k)$ is precisely $\dot X_{\bar c,b}^{\caG_{\bfx}}$.
\end{proof}

We now prove \eqref{eq:main_iso_1} for $\bfG$ as in Corollary~\ref{cor:disj_dec_on_U_level} (this suffices by Section~\ref{sec:reduction_of_main_1_to_adjoint}). We need to descend the disjoint union decomposition of Corollary~\ref{cor:disj_dec_on_U_level} to $X_{c,b}$. Recall that by \cite[Corollary~2.9]{Rydh_10}, a quasi-compact morphism of schemes is universally subtrusive if and only if it satisfies the condition in \cite[Definition~2.1]{BhattS_17}, \textit{i.e.}, is a $v$-cover. We need the following version of \cite[Theorem~4.1]{Rydh_10}.

\begin{lm}\label{lm:v_descend_of_clopens}
Let $S' \rar S$ be a quasi-compact universally subtrusive map between schemes. Let $E \subseteq S$ be a subset, $E' = f^{-1}(E)$. Then $E$ is closed $($resp.\ open$)$ if and only if $E'$ is closed $($resp.\ open$)$.
\end{lm}

\begin{proof}
It suffices to prove the claim for closed sets. The ``only if'' part is clear. For the other direction, it suffices by \cite[Corollary~1.5]{Rydh_10} to show that $E$ is closed in the constructible topology and stable under specializations if the same holds for $E'$. As $f$ is quasi-compact, $f$ is closed in the constructible topology; \textit{cf.} \cite[Proposition~1.2]{Rydh_10}. Hence $E=f(E')$ is closed in the constructible topology. Let $x_0 \rightsquigarrow x_1$ be a specialization relation in $S$ with $x_0 \in E$. As $f$ is universally subtrusive, $x_0 \rightsquigarrow x_1$ lifts to a specialization relation $y_0 \rightsquigarrow y_1$. Then $y_0 \in E'$ and as $E'$ is closed, also $y_1 \in E'$. Hence $x_1 = f(y_1) \in f(E')= E$. \qedhere
\end{proof}

The space $X_{c,b}$ is an ind-scheme by Proposition~\ref{prop:isom_dotXbw_Xdoubleprime} and \cite[Theorem~C]{Ivanov_DL_indrep}. Thus for any map $T \rar X_{c,b}$ with $T$ a quasi-compact scheme, the fiber product $T \times_{X_{c,b}} \dot X_{b,b}$ with the scheme $\dot X_{b,b}$ is again a scheme. Moreover, the map $T \times_{X_{c,b}} \dot X_{b,b} \rar T$ is quasi-compact and a $v$-cover by Corollary~\ref{cor:adjoint_Case_torsors}. Hence this map is universally subtrusive, \textit{cf.} \cite[Corollary~2.9]{Rydh_10}, and by Lemma~\ref{lm:v_descend_of_clopens} and Corollary~\ref{cor:disj_dec_on_U_level}, $T$ admits a disjoint union decomposition $T = \coprod_{\gamma \in G_b(k)/\caG_{\bfx,b}(\caO_k)} T_\gamma$ which is functorial in $T$. Doing this for all maps $T \rar X_{c,b}$ from a quasi-compact scheme $T$ into $X_{c,b}$, we deduce a disjoint union decomposition of $X_{c,b}$ indexed by the same set. The subscheme (ind-scheme) corresponding to $\gamma = 1$ is the quotient of $\dot X_{\bar c,b}^{\caG_{\bfx}}$ by $\bfT_c(k)$, which equals $X_{c,b}^{\caG_{\bfx}}$ by definition. This proves \eqref{eq:main_iso_1}.

\subsection{Proof of \eqref{eq:main_iso_2}}\label{sec:proof_of_main_iso_2}
Using Proposition~\ref{prop:isom_dotXbw_Xdoubleprime} we may replace the left-hand side in \eqref{eq:main_iso_2} by $\dot X_{b,b}$ (as $b$ lies over $\bar c$). If $\bfG$ is of adjoint type, \eqref{eq:main_iso_2} follows from Corollary~\ref{cor:disj_dec_on_U_level} and the discussion in Section~\ref{sec:reduction_of_main_1_to_adjoint}. We now proceed in two steps. 

\subsubsection{Special case: The center $\bfZ$ of $\bfG$ is an unramified induced torus}\label{sec:main_iso_2_part_case}

\begin{lm}\label{lm:proof_of_main_iso_2_fibers_of_index_set}
The natural map $p \colon \bfG_b(k)/\caG_{\bfx,b}(\caO_k) \rar \bfG_b^{\ad}(k)/\caG_{\bfx,b}^{\ad}(\caO_k)$ is surjective, and its fibers are $X_\ast(\bfZ)^{\langle \sigma \rangle}$-torsors.
\end{lm}

\begin{proof}
As $\bfZ$ is an induced torus, $H^1(\sigma,\bfZ) = 0$. Hence $\bfZ(k) \har \bfG_b(k) \tar \bfG^{\ad}_b(k)$ is exact. Let $\caZ$ denote the reductive $\caO_k$-model of $\bfZ$. Using $H^1(\sigma, \caZ(\caO_{\breve k})) = 0$ (by the pro-version of Lang's theorem), we deduce the exact sequence $\caZ(\caO_k) \har \caG_{\bfx,b}(\caO_k) \tar \caG_{\bfx,b}^{\ad}(\caO_k)$. Factoring this sequence out from the first one, and using $\bfZ(k)/\caZ(\caO_k) \cong X_{\ast}(\bfZ)^{\langle \sigma \rangle}$ (which again follows from $H^1(\sigma, \caZ(\caO_{\breve k})) = 0$), we obtain the lemma.
\end{proof}

Let $\pi \colon L\bfG \rar L\bfG^{\ad}$ be the natural map. It is surjective as $\pi_1(\bfG) \rar \pi_1(\bfG^{\ad})$ is. For $g \in \bfG$, write $\bar g$ for $\pi(g)$. For any $\tau \in X_\ast(\bfZ)$, choose a lift $\tilde \tau \in \bfZ(\breve k)$ subject to the condition that $\tilde\tau \in \bfZ(k)$ if $z \in X_\ast(\bfZ)^{\langle \sigma \rangle}$. We have $\pi^{-1}(L^+\caG_{\bfx}^{\ad}) = \coprod_{\tau \in X_\ast(\bfZ)} \tilde \tau L^+\caG_\bfx$ (indeed, $L^+\caG_{\bfx} \subseteq \pi^{-1}(L^+\caG_{\bfx}^{\ad})$, and the quotient is the discrete scheme $X_\ast(\bfZ)$). Consequently, if $\bar \gamma \in \bfG_b^{\ad}(k)/\caG_{\bfx,b}^{\ad}(\caO_k)$ with a fixed lift $\gamma \in \bfG_b(k)/\caG_{\bfx,b}(\caO_k)$, then
\begin{equation}\label{eq:disjoint_union_preimage_from_adjoint}
\pi^{-1}(\bar\gamma \dot X_{b,b}^{\caG_{\bfx}^{\ad}}) = \coprod_{\tau \in X_\ast(\bfZ)} \gamma \tilde \tau Y, \quad \text{ where } Y \eqdef \left\{g \in L^+\caG_{\bfx} \colon \bar g^{-1}\sigma_b(\bar g) \in L^+\left({}^c\bfU^{\ad} \cap \bfU^{\ad -}\right)_{\bfx} \right\}. 
\end{equation}

\begin{lm}\label{lm:emptyness_in_proof_of_main_2_for_conn_center}
If $\tau \in X_\ast(\bfZ) \sm X_\ast(\bfZ)^{\langle \sigma \rangle}$, then $\dot X_{b,b} \cap \gamma \tilde \tau Y = \varnothing$.
\end{lm}

\begin{proof} 
It suffices to prove that $(\dot X_{b,b} \cap \gamma \tilde \tau Y)(\ff) = \varnothing$ for any algebraically closed field $\ff \in \Perf_{\obF}$. Let $g \in L^+\caG_{\bfx}(\ff)$ be such that $\gamma \tilde \tau g \in (\dot X_{b,b} \cap \gamma \tilde \tau Y)(\ff)$. Then $g^{-1}\sigma_b(g) \in \pi^{-1}(L^+({}^c\bfU^{\ad} \cap \bfU^{\ad -})_{\bfx})(\ff) =$\linebreak  $(L^+\caZ)(\ff) \cdot L^+({}^c\bfU \cap \bfU^-)_{\bfx}(\ff)$, where $\caZ$ is the reductive $\caO_k$-model of $\bfZ$. Write $g^{-1}\sigma_b(g) = \zeta u$ with $\zeta \in (L^+\caZ)(\ff)$ and $u \in L^+({}^c\bfU \cap \bfU^-)_{\bfx}(\ff)$. Then
\[
(\gamma \tilde \tau g)^{-1} \sigma_b(\gamma\tilde\tau g) = \tilde \tau^{-1}\sigma(\tilde\tau) g^{-1}\sigma_b(g) = \tilde \tau^{-1}\sigma(\tilde\tau) \zeta u,
\]
where we have $\sigma_b(\tilde\tau) = \sigma(\tilde \tau)$ as $\tilde \tau$ is central. This cannot lie in $L^+({}^c\bfU \cap \bfU^-)(\ff)$ as the image of $\zeta\in L^+\caZ(\ff) \subseteq L\bfZ(\ff)$ in $X_\ast(\bfZ)$ is $0$ but the image of $\tilde \tau^{-1}\sigma(\tilde\tau)$ in $X_\ast(\bfZ)$ is non-zero by assumption.
\end{proof}

As $\dot X_{b,b} \subseteq \pi^{-1}(\dot X_{b,b}^{\ad})$, the disjoint decomposition for $\dot X_{b,b}^{\ad}$ (\textit{cf.} beginning of Section~\ref{sec:proof_of_main_iso_2}) implies, together with \eqref{eq:disjoint_union_preimage_from_adjoint} and Lemmas~\ref{lm:proof_of_main_iso_2_fibers_of_index_set} and~\ref{lm:emptyness_in_proof_of_main_2_for_conn_center}, that $\dot X_{b,b}$ is the disjoint union of all $\gamma \dot X_{b,b}^{\caG_{\bfx}}$, where $\gamma$ varies through $\bfG_b(k)/\caG_{\bfx,b}(\caO_k)$, which is precisely \eqref{eq:main_iso_2}.

\subsubsection{General case}\label{sec:proof_main_cover_general_case} An arbitrary $\bfG$ admits a derived embedding $\bfG \har \widetilde\bfG$ such that $\widetilde\bfG$ is unramified reductive and the center of $\widetilde\bfG$ is an unramified induced torus. (This was pointed out to us by M.~Borovoi; \textit{cf.} \cite{Borovoi_Mathoverflow_21}.) Now the result follows from Section~\ref{sec:main_iso_2_part_case} and the next proposition.

\begin{prop}\label{prop:reduction_to_conn_center}
Let\, $\bfG \har \widetilde \bfG$ be a derived embedding. If \eqref{eq:main_iso_2} holds for\, $\widetilde \bfG$, then it holds for\, $\bfG$. 
\end{prop}

\begin{proof}
We identify $\bfG$ with a subgroup of $\widetilde \bfG$. Knowing $\dot X_{b,b}^{\widetilde \bfG} = \coprod_{\gamma \in \widetilde\bfG_b(k)/\widetilde\caG_{\bfx ,b}(\caO_k)} \gamma  {\dot X}_{\bar c,b}^{\widetilde \caG_{\bfx}}$ as subsheaves of $L\widetilde \bfG$, we have to show that $\dot X_{b,b} = \coprod_{\gamma \in \bfG_b(k)/\caG_{\bfx ,b}(\caO_k)} \gamma \dot X_{\bar c,b}^{\caG_{\bfx}}$. 

\begin{lm}\label{lm:intersection_of_parahorics}
Let $\ff \in \Perf_{\obF}$ be an algebraically closed field. Let $L:= \bW(\ff)[1/\varpi]$. We have $\bfG(L) \cap \widetilde\caG_{\bfx}(\caO_L) = \caG_{\bfx}(\caO_L)$. 
\end{lm}

\begin{proof}
By \cite[Proposition~3]{HainesR_08} we have $\caG_{\bfx}(\caO_L) = \Fix_{\bfG(L)}(\bfx) \cap \ker\tilde\kappa_{\bfG}$ and $\widetilde \caG_{\bfx}(\caO_L) = \Fix_{\widetilde \bfG(L)}(\bfx) \cap \ker\tilde\kappa_{\widetilde\bfG}$, where $\Fix(\bfx)$ denotes the stabilizer of the vertex $\bfx$ in the respective group (we identify the adjoint groups of $\bfG$ and $\widetilde\bfG$). As $\pi_1(\bfG) \rar \pi_1(\widetilde\bfG)$ is injective, \textit{cf.} \cite[Lemma 1.5]{Borovoi_98}, we have $\ker\tilde\kappa_{\widetilde \bfG} \cap \bfG(L) = \ker\tilde\kappa_{\bfG}$. As the actions on the adjoint building are compatible with the inclusion $\bfG \har \widetilde\bfG$, we have $\Fix_{\widetilde\bfG(L)}(\bfx) \cap \bfG(L) = \Fix_{\bfG(L)}(\bfx)$. Putting all these together, we deduce $\widetilde\caG_{\bfx}(\caO_L) \cap \bfG(L) = \caG_{\bfx}(\caO_L)$.
\end{proof}

Taking $\sigma_b$-invariants in Lemma~\ref{lm:intersection_of_parahorics}, we deduce that $\widetilde \caG_{\bfx ,b}(\caO_k) \cap \bfG_b(k) = \caG_{\bfx ,b}(\caO_k)$. In particular, the natural map $\bfG_b(k)/\caG_{\bfx ,b}(\caO_k) \rar \widetilde\bfG_b(k)/\widetilde\caG_{\bfx ,b}(\caO_k)$ is injective. Pick some $\bar\gamma \in \widetilde\bfG_b(k)/\widetilde\caG_{\bfx ,b}(\caO_k)$, represented by $\gamma \in \widetilde\bfG_b(k)$. 

\begin{lm}\label{lm:gamma_rational_if_intersects}
If\, $\gamma \dot X_{\bar c,b}^{\widetilde \caG_{\bfx}} \cap L\bfG \neq \varnothing$, then $\bar\gamma \in \bfG_b(k)/\caG_{\bfx,b}(\caO_k)$. 
\end{lm}

\begin{proof}
By assumption, the sheaf $\gamma \dot X_{\bar c,b}^{\widetilde \caG_{\bfx}} \cap L\bfG$ admits a geometric point. Let $\ff \in \Perf_{\obF}$ be algebraically closed, such that there exists some $x \in \dot X_{\bar c,b}^{\widetilde \caG_{\bfx}}(\ff)$ with $\gamma x \in L\bfG(\ff)$. Write $L = \bW(\ff)[1/\varpi]$. By the definition of $\dot X_{\bar c,b}^{\widetilde \caG_{\bfx}}$, we have 
\[
x^{-1}\sigma_b(x) \in ({}^c\bfU \cap \bfU^-)_{\bfx}(\caO_L) \subseteq \caG_{\bfx}(\caO_L)
\]
(here, we identify $\bfU \subseteq \bfG$ with the unipotent radical of a Borel of $\widetilde\bfG$). As $H^1(\sigma_b, L^+\caG_{\bfx}) = 0$,
we find some $\zeta \in \caG_{\bfx}(\caO_L)$ such that $x^{-1}\sigma_b(x) = \zeta^{-1}\sigma_b(\zeta)$. In particular, $x\zeta^{-1} \in \widetilde\caG_{\bfx}(\caO_L)^{\sigma_b} = \widetilde\caG_{\bfx,b}(\caO_k)$. From this and $\gamma \in \widetilde\bfG_b(k)$, it follows that $\gamma x \zeta^{-1} \in \widetilde\bfG_b(k)$. Moreover, $\gamma x \zeta^{-1} \in \bfG(L)$ as $\gamma x \in \bfG(L)$ and $\zeta \in \caG_{\bfx}(\caO_L) \subseteq \bfG(L)$. Combining these observations, we deduce $\gamma x \zeta^{-1} \in \bfG(L)^{\sigma_b} = \bfG_b(\caO_k)$. Finally, $x\zeta^{-1} \in \widetilde\caG_{\bfx}(\caO_L)$; thus $\gamma$ and $\gamma x\zeta^{-1}$ represent the same class in $\bfG_b(k)/\caG_{\bfx,b}(\caO_k)$ as $\bfG_b(k)/\caG_{\bfx ,b}(\caO_k) \rar \widetilde\bfG_b(k)/\widetilde\caG_{\bfx ,b}(\caO_k)$ is injective. The lemma is proved.
\end{proof}

Next we claim that 
\begin{equation}\label{eq:intersection_of_fund_part_Gtilde_G} 
\dot X_{\bar c,b}^{\widetilde \caG_{\bfx}} \cap L\bfG = \dot X_{\bar c,b}^{\caG_{\bfx}}.
\end{equation}

\begin{lm}\label{lm:sub_ind_schemes}
Let $X$ be an ind-(reduced scheme) over $\obF$ and let $Y,Z$ be two closed subfunctors, which are themselves ind-(reduced schemes). If $Y(\ff) = Z(\ff)$ for any algebraically closed field $\ff/\obF$, then $Y = Z$.
\end{lm}

\begin{proof}
Write $X = \dirlim_i X_i$ for qcqs reduced schemes $X_i$ and an index set $I$. Then $Y \cap X_i, Z \cap X_i$ are closed reduced subschemes of $X_i$, which agree on geometric points. Thus $Y \cap X_i = Z \cap X_i$ for each $i$, and hence also $Y = Z$ as the $X_n$ exhaust $X$.
\end{proof}

By combining Lemma \ref{lm:sub_ind_schemes} with Lemma \ref{lm:intersection_of_parahorics} we see that $L\bfG\cap L^+\widetilde\caG_{\bfx} = L^+\caG_{\bfx}$ (note that any ind-(perfect scheme) is automatically an ind-(reduced scheme)). As $\dot X_{\bar c,b}^{\widetilde \caG_{\bfx}} \subseteq L^+\widetilde\caG_{\bfx}$ by definition, we deduce $\dot X_{\bar c,b}^{\widetilde \caG_{\bfx}} \cap L\bfG = \dot X_{\bar c,b}^{\widetilde \caG_{\bfx}} \cap L^+\caG_{\bfx}$. Now, $\dot X_{\bar c,b}^{\widetilde \caG_{\bfx}} \cap L^+\caG_{\bfx} = \dot X_{\bar c,b}^{\caG_{\bfx}}$ is immediate, proving \eqref{eq:intersection_of_fund_part_Gtilde_G}. 

Finally, observe that $\dot X_{b,b} = \dot X_{b,b}^{\widetilde\bfG} \cap L\bfG$. Now we compute
\begin{align*}
\dot X_{b,b} &= \dot X_{b,b}^{\widetilde\bfG} \cap L\bfG 
= \coprod_{\gamma \in \widetilde\bfG_b(k)/\widetilde\caG_{\bfx,b}(\caO_k)} \left(\gamma \dot X_{\bar c,b}^{\widetilde\caG_{\bfx}} \cap L\bfG \right) = \coprod_{\gamma \in \bfG_b(k)/\caG_{\bfx,b}(\caO_k)} \left(\gamma \dot X_{\bar c,b}^{\widetilde\caG_{\bfx}} \cap L\bfG \right) \\
&= \coprod_{\gamma \in \bfG_b(k)/\caG_{\bfx,b}(\caO_k)} \gamma \left(\dot X_{\bar c,b}^{\widetilde\caG_{\bfx}} \cap L\bfG \right) 
= \coprod_{\gamma \in \bfG_b(k)/\caG_{\bfx,b}(\caO_k)} \gamma \dot X_{\bar c,b}^{\caG_{\bfx}},
\end{align*}
where the second equality is by the assumption of the proposition, the third is by Lemma~\ref{lm:gamma_rational_if_intersects} and the fifth holds by \eqref{eq:intersection_of_fund_part_Gtilde_G}. This proves the proposition.
\end{proof}

\subsection{Quasi-split case}\label{sec:quasisplit_proof} 
We work in the setup of Section~\ref{sec:quasisplit_statement}. In particular, we have the Coxeter element $c$, and we have $\bar c \in [1] \cap \overline F_c\,/\ker\bar\kappa_c$ (if $\bar c \in \overline F_c\,/\ker\bar\kappa_c \sm [1]$, then $\dot X_{\bar c}(1) = \varnothing$ by Corollary~\ref{cor:choice_b_cox_over_barw}). First, we justify that a choice of $\dot c$ as claimed in Section~\ref{sec:quasisplit_statement} is indeed possible. It is well known (in the quasi-split case) that $\fT(\bfG,c)\,/\Ad\bfG(k)$ is in natural bijection with $\bfG(k)$-orbits in the set of hyperspecial vertices of $\caB_{\bfG,k} = \caB_{\bfG,\breve k}^{\langle \sigma \rangle}$ (\textit{cf.} \cite{deBacker_06} or \cite[Lemma~2.6.2]{DudasI_20}). Let $[\bfT_{\bar c}] \in \fT(\bfG,c)\,/\Ad\bfG(k)$ be the class corresponding to $\bar c$ under the surjection $[1] \cap \overline F_c\,/\ker\bar\kappa_c \tar \fT(\bfG,c)\,/\Ad\bfG(k)$ from Proposition~\ref{prop:conj_classes_of_tori_basic_case}. Then  let$\bfx$ be any hyperspecial vertex in $\caB_{\bfG,k}$ corresponding to $[\bfT_{\bar c}]$ under the above bijection. As $\bfx$ is hyperspecial, $c$  admits a lift $\dot c \in F_c(\breve k)\cap \caG_\bfx(\caO_{\breve k})$. As $c$ is Coxeter, the affine transformation $\dot c\sigma$ of $\caA_{\bfT,\breve k}$ admits a unique fixed point, which necessarily equals $\bfx$. Comparing the two parametrizations of $\fT(\bfG,c)\,/\Ad\bfG(k)$ above, we see that the image of $\dot c$ equals the image of $\bar c$ in $([1] \cap \overline F_c\,/\ker\bar\kappa_c)\,/C_W(c\sigma)$. Replacing $\dot c$ with $\gamma \dot c$ for an appropriate $\gamma \in C_W(c\sigma)$, we have that $\dot c$ lifts~$\bar c$.

Now we prove Corollary~\ref{cor:main_decomposition_b1_case}. We only prove the first isomorphism (the second is completely analogous). We may assume that $\bfG$ is adjoint. By the pro-version  of Lang's theorem for $\caG_{\bfx}$, we may find an $h \in \caG_{\bfx}(\caO_{\breve k})$ such that $h^{-1}\dot c\sigma(h) = 1$. By \cite[Remark~8.7]{Ivanov_DL_indrep}, $g \mapsto hg \colon L\bfG \rar L\bfG$ induces an isomorphism $X_c(\dot c) \isor X_c(1)$ which is equivariant with respect to $\Ad(h) \colon \bfG(k) \isor \bfG_{\dot c}(k)$. Moreover, $\Ad(h)$ restricts to an isomorphism $\caG_{\bfx,1}(\caO_{\breve k}) \isor \caG_{\bfx,\dot c}(\caO_{\breve k})$, thus inducing $\bfG(k)/\caG_{\bfx,1}(\caO_{\breve k}) \isor \bfG_{\dot c}(k)/\caG_{\bfx,\dot c}(\caO_{\breve k})$. Exploiting Theorem~\ref{thm:main_decomposition} for $X_c(\dot c)$, we are reduced to showing that $g \mapsto hg \colon L^+\caG_{\bfx} \rar L^+\caG_{\bfx}$ induces an isomorphism $X_{c,\dot c}^{\caG_{\bfx}^{\ad}} \isor X_c^{\caG_{\bfx}}(1)$. We compute (quotients are taken in the category of arc-sheaves)
\begin{align*}
X_{c,\dot c}^{\caG_{\bfx}^{\ad}} &= \{g \in L^+\caG_{\bfx} \colon g^{-1}\sigma_{\dot c}(g) \in L^+({}^c\bfU \cap \bfU^-)_{\bfx} \}/\bfT_c(k) \\
&\cong \{g \in L^+\caG_{\bfx} \colon g^{-1}\sigma_{\dot c}(g) \in L^+({}^c\bfU \cap \bfU^-)_{\bfx} \cdot L^+\caT \}/L^+\caT \\
&\cong \{g \in L^+\caG_{\bfx} \colon g^{-1}\sigma(g) \in L^+({}^c\bfU \cap \bfU^-)_{\bfx} \cdot L^+\caT \cdot \dot c \}/L^+\caT \\
&\cong \{g \in L^+\caG_{\bfx} \colon g^{-1}\sigma(g) \in \dot c \cdot L^+ \bfB_{\bfx} \} / L^+({}^c\bfU \cap \bfU)_{\bfx} L^+\caT \\
&\cong \{ g \in L^+(\caG_{\bfx}/\bfB_{\bfx}) \colon g^{-1}\sigma(g) \in L^+\bfB_{\bfx} \cdot \dot c \cdot L^+\bfB_{\bfx} \}/L^+\bfB_{\bfx} \cong X_c^{\caG_{\bfx}}(1), 
\end{align*}
where the third isomorphism is $g \mapsto hg$ and in the fourth isomorphism, we exploit the last statement of Proposition~\ref{prop:loop_Steinbergs_twisted_crosssection}.

\section{Estimates via Newton polygons}\label{sec:type_by_type}

In this section we prove Proposition~\ref{prop:Xbb_XbbO_equal} for an absolutely almost simple classical group $\bfG^{\ad}$ of adjoint type (so, compared to Section~\ref{sec:deco}, we change our notation, and the group which was denoted by $\bfG$ there is $\bfG^{\ad}$ here). It will be more convenient to work with a certain central extension $p \colon \bfG \rar \bfG^{\ad}$, whose kernel $\bfZ = \ker(p)$ is an unramified torus. (Which central extension we take is specified in each particular case of our case-by-case study below; see Sections~\ref{sec:basic_setup_An1}, \ref{sec:basic_setup_C}, \ref{sec:basic_setup_odd_orthogonal}, \ref{sec:type_Dn_setup}, \ref{sec:type_2An1_setup}, \ref{sec:type_2Dm_setting}.) Then we have the exact sequence
\begin{equation}\label{eq:induced_zext_in_casebycase}
 X_\ast(\bfZ)_{\langle \sigma \rangle} \longrar \pi_1(\bfG)_{\langle \sigma \rangle} \longrar \pi_1(\bfG^{\ad})_{\langle \sigma \rangle} \longrar 0. 
\end{equation}

\begin{rem}\label{rem:trivial_kernel_exception}
In our case-by-case study below, the left map will be injective, except in type ${}^2A_{n-1}$ with $n$ even. (This is clear in the split cases $A_{n-1}, B_m,C_m,D_m$; see Sections~\ref{sec:type_2An1_roots_etc} and \ref{sec:type_2Dm_further_setting} for the cases ${}^2A_{n-1}$ ($n$ odd) and ${}^2D_m$.) In type ${}^2A_{n-1}$ with $n$ even, the left map will be the zero map $\bZ/2\bZ \rar \bZ/2\bZ$ (see Section~\ref{sec:type_2An1_roots_etc}). 
\end{rem}

We let $\bar b \in F_c^{\ad}(\breve k) \subseteq \bfG^{\ad}(\breve k)$ be a basic element lying over a Coxeter element $c \in W$, so that we have to show that $X_{\bar b,\bar b}^{\ad} = X_{\bar b,\bar b, \caO}^{\ad}$ within $L\bfG^{\ad}$.
It suffices to show that the Lang map $g \mapsto g^{-1}\sigma_{\bar b}(g) \colon X_{\bar b,\bar b}^{\ad} \rar L({}^c\bfU \cap \bfU)$ factors through the closed subscheme $L^+({}^c\bfU \cap \bfU)_{\bfx}$. As all rings in $\Perf_{\obF}$ are reduced, 
it suffices to do so on geometric points. Let $\ff$ denote a fixed algebraically closed field in $\Perf_{\obF}$ and put $L := \bW(\ff)[1/\varpi]$. 

As $\pi_1(\bfG)_{\langle \sigma \rangle} \rar \pi_1(\bfG^{\ad})_{\langle \sigma \rangle}$ is surjective, $\bar b$ lifts to a basic element $b \in F_c(\breve k) \subseteq \bfG(\breve k)$. For any element $\bar\zeta \in \ker\left(X_{\ast}(\bfZ)_{\langle \sigma \rangle} \rar \pi_1(\bfG)_{\langle \sigma \rangle}\right)$, fix a representative $\zeta \in X_{\ast}(\bfZ)$ of $\bar \zeta$. For each such $\zeta$, we have the commutative diagram
\begin{equation}\label{eq:lifting_Xbb_Xbb_ad}
\xymatrix{ 
\dot X_{b\varpi^\zeta,b} \ar[r]^{q_{b,\zeta}} & \dot X_{\bar b,\bar b}^{\ad} \\
\dot X_{b\varpi^\zeta,b,\caO} \ar[r] \ar@{^(->}[u] & \dot X_{\bar b,\bar b,\caO}^{\ad}\rlap{,} \ar@{^(->}[u]
}
\end{equation}
where the objects on the left (resp.\ right) are contained in $L\bfG$ (resp.\ $L\bfG^{\ad}$), and where 
\[ 
\dot X_{b\varpi^\zeta,b,\caO} \eqdef \text{preimage of $L^+({}^c\bfU \cap \bfU^-)_{\bfx}\varpi^\zeta$ under $g\mapsto g^{-1}\sigma_b(g) \colon L\bfG \rar L\bfG$}  
\]
is a subsheaf of $\dot X_{b\varpi^\zeta,b}$. 

\begin{lm}\label{lm:lift_to_zext} We have
\[
\bigcup_{\bar \zeta \in \ker\left(X_{\ast}(\bfZ)_{\langle \sigma \rangle} \rar \pi_1(\bfG)_{\langle \sigma \rangle}\right)} q_{b,\zeta}(\dot X_{b\varpi^\zeta,b}(\ff)) = \dot X_{\bar b, \bar b}(\ff). 
\]
\end{lm}

\begin{proof}
Let $\bar g \in \dot X_{\bar b,\bar b}^{\ad}(\ff)$. The map $L\bfG(\ff) \rar L\bfG^{\ad}(\ff)$ is surjective (as $H^1(L,\bfZ) = 0$). Thus we may lift $\bar g$ to some $g \in L\bfG(\ff)$. Then $g^{-1}\sigma_b(g) = uz$ with $u \in ({}^c\bfU \cap \bfU^-)(L)$ and $z \in \bfZ(L)$. On the other hand, it is clear that under $\kappa_{\bfG} \colon \bfG(L) \rar \pi_1(\bfG)_{\langle \sigma \rangle}$, the element $g^{-1}\sigma_b(g)$ maps to $0$. As also $\kappa_{\bfG}(({}^c\bfU \cap \bfU^-)(L)) = 0$, we deduce that $\kappa_{\bfG}(z) = 0$. Thus the image $\bar \zeta$ of $z$ in $X_\ast(\bfZ)_{\langle \sigma \rangle}$ lies in $\ker\left(X_{\ast}(\bfZ)_{\langle \sigma \rangle} \rar \pi_1(\bfG)_{\langle \sigma \rangle}\right)$. The exactness of \eqref{eq:torus_sequence_kottwitz_map} implies that there exists some $\tau \in \bfZ(L)$ such that $z = \tau^{-1}\sigma(\tau)\varpi^\zeta$. Replacing $g$ with $g\tau^{-1}$ and recalling that $\bfZ$ is central in $\bfG$, we get $g^{-1}\sigma_b(g) \in ({}^c\bfU \cap \bfU^-)(L)\varpi^\zeta$, \textit{i.e.}, $g \in \dot X_{b\varpi^\zeta,b}(\ff)$. 
\end{proof}

To show that the right arrow in \eqref{eq:lifting_Xbb_Xbb_ad} is surjective, it suffices (using Lemma~\ref{lm:lift_to_zext}) to show that the left vertical arrow is surjective on $\ff$-points for all $\bar{\zeta}$. Thus, finally, we are reduced to showing that the inclusion $\dot X_{b\varpi^\zeta,b,\caO}(\ff) \har \dot X_{b\varpi^\zeta,b}(\ff)$ is surjective, \textit{i.e.}, that 
\begin{equation}\label{eq:to_show_in_case_by_case}
\forall g \in \dot X_{b\varpi^\zeta,b}(\ff), \quad \text{ we have } \quad g^{-1}\sigma_b(g) \in L^+({}^c\bfU \cap \bfU^-)_{\bfx}(\ff)\varpi^\zeta; 
\end{equation}
we will do this in the following subsections case by case. As mentioned in Remark~\ref{rem:trivial_kernel_exception}, we have $\#\ker\left(X_{\ast}(\bfZ)_{\langle \sigma \rangle} \rar \pi_1(\bfG)_{\langle \sigma \rangle}\right) = 1$, and it will suffice to prove only the surjectivity of $\dot X_{b,b,\caO}(\ff) \har \dot X_{b,b}(\ff)$, with the only exception (treated in Section~\ref{sec:type_2An1_even_extra_comp}) in the case ${}^2A_{n-1}$ with $n$ even, when $\#\ker\left(X_{\ast}(\bfZ)_{\langle \sigma \rangle} \rar \pi_1(\bfG)_{\langle \sigma \rangle}\right) = 2$. We will use the following well-known property of isocrystals.

\begin{lm}\label{lm:isocrystal_lemma}
Let $(V,\varphi)$ be an isocrystal over $\ff$ which is isoclinic of slope $\lambda$ and dimension $n \geq 1$. Let $v$ be a \emph{cyclic vector} of $(V,\varphi)$, that is, an element $v \in V$ such that the set $\{\varphi^i(v)\}_{i=0}^{n-1}$ is a basis of $V$.
Write $\varphi^n(v) = \sum_{i=0}^{n-1}A_i\varphi^i(v)$ with $A_i \in L$. Then $\ord_\varpi(A_i) \geq (n-i)\lambda$.
\end{lm}

\begin{proof}
The Newton polygon of $(V,\varphi)$ is the straight line segment connecting the points $(0,0)$ and $(n,\lambda n)$ in the plane. It coincides with the Newton polygon of the polynomial $\sum_{i=0}^{n} A_i X^i \in L[X]$ with $A_n = 1$, which is by definition the convex hull of the points $(i,\ord_\varpi(A_{n-i}))$; \textit{cf.} \cite[Sections~3.1 and~3.2]{Beazley_12} or \cite[Proposition~5.7]{Kedlaya_04} . 
\end{proof}

Finally, note that if $g \in \dot X_{b,b}(\ff)$, then $g^{-1}\sigma_b(g) = y \in ({}^c\bfU \cap \bfU^-)(L)$, \textit{i.e.}, we have 
\begin{equation}\label{eq:coxeter_almost_cyclic_equation_GSp}
b\sigma(g) = gyb.
\end{equation}

We will now proceed case by case. Below we will denote the fixed point $\bfx$ of $b\sigma$ by $\bfx_b$ (to make the notation non-ambiguous). Also recall the notation from Section~\ref{sec:notation_and_setup}, especially Section~\ref{sec:notation_groups}.

\subsection{Type $\boldsymbol{A_{n-1}}$}\label{sec:type_A_case_by_case}
This case is treated in \cite{CI_ADLV,CI_loopGLn}. We include it for the sake of completeness.

\subsubsection{}\label{sec:basic_setup_An1} Let $n\geq 2$, and let $V_0$ be an $n$-dimensional $k$-vector space with an (ordered) basis $\{e_1, \dots, e_n \}$. Whenever convenient, identify $V_0$ with $k^n$ via this basis. Let $\bfG = {\bf GL}_n(V_0)$. 

\subsubsection{}\label{sec:type_A_further_setup} 
Let $\bfT \subseteq \bfB \subseteq \bfG$ be the diagonal torus and the upper triangular Borel subgroup. Then $X_\ast(\bfT) \cong \bigoplus_{i=1}^n\bZ\beta_i$, with $\beta_i \colon t \mapsto (1^{\oplus i-1},t,1^{\oplus n-i})$. Then $X_\ast(\bfT^{\sssc}) = \bZ\Phi^{\vee} = \bigoplus_{i=1}^{n-1}\bZ\cdot \left(\beta_{i+1}-\beta_i\right)$, $X_\ast(\bfZ) = \bZ \cdot \sum_{i=1}^n \beta_i$.
Then $\pi_1(\bfG^{\ad}) = X_\ast(\bfT)/\bZ\Phi^{\vee} + X_\ast(\bfZ) \cong \bZ/n\bZ$, the isomorphism given by sending $1$ to the class $\bar\beta_i$ of $\beta_i$ for any $i$.

We have $\Phi  = \{ \alpha_{i-j} \colon 1\leq i < j \leq n\}$, where $\alpha_{i-j} \colon (t_\lambda)_{\lambda=1}^n \mapsto t_it_j^{-1}$ and $\Delta = \{ \alpha_{i-(i+1)} \colon 1\leq i \leq n-1\}$. We have the Coxeter element
\begin{equation}\label{eq:special_Coxeter_element_type_A}
c = s_{\alpha_{1-2}}s_{\alpha_{2-3}} \dots s_{\alpha_{(n-1)-n}} = (1,2,\dots,n).
\end{equation}
For $0\leq \kappa < n$, consider the lift $\dot c_\kappa$ of $c$ to $\bfG(\breve k)$ given by $e_i \mapsto e_{i+1}$ for $1 \leq i < n$ and $e_n \mapsto \varpi^\kappa e_1$. Then $\dot c_\kappa$ maps to $\kappa\bar\beta_1 \in \pi_1(\bfG^{\ad})$,  and $\bfG_{\dot c_\kappa}$ run through all inner forms of $\bfG$.

We identify $\caA_{\bfT^{\ad}, \breve k}$ with $X_\ast(\bfT^{\ad})_{\bR}$ by sending $0 \in X_\ast(\bfT)_{\bR}$ to the $\sigma$-stable hyperspecial point corresponding to $\bigoplus_{i=1}^n \caO_{\breve k}e_i \subseteq V_0 \otimes_k \breve k$.

The group ${}^c\bfU \cap \bfU^-$ is generated by the root subgroups attached to $\alpha_{i-1}$ ($1<i\leq n$). An element $y \in ({}^c\bfU \cap \bfU^-)(R)$ is given by $e_1 \mapsto e_1 + \sum_{i=2}^n a_{i-1} e_i$; $e_i \mapsto e_i$ for $2\leq i\leq n$.

\subsubsection{}\label{sec:computation_case_An} Let $b = \dot c_\kappa$ with $0 \leq \kappa < n$. We compute $\bfx_b = \left(-\frac{(i-1)\kappa}{n}\right)_{i=1}^n$. Let $g \in \dot X_{b,b}(\ff)$. Then \eqref{eq:coxeter_almost_cyclic_equation_GSp} holds with $y$ as described in Section~\ref{sec:type_A_further_setup}. We have to show that $y \in \caG_{\bfx_b}(\caO_L)$, or equivalently (see \textit{e.g.}~\cite[Section~2.5]{MoyP_94})
\begin{equation}\label{eq:MP_counts_type_A}
\ord\nolimits_\varpi(a_i) \geq \langle \alpha_{(i+1) - 1}, {\bf x}_b \rangle = -\frac{\kappa i}{n} \qquad \forall 1\leq i\leq n-1.
\end{equation}

Consider the $\phi$-linear isocrystal $(V,\varphi) = (V_0 \otimes_k L, b\sigma)$ over $\ff$. It is isoclinic of slope $\frac{\kappa}{n}$. For $1\leq i\leq n$, let $g_i = g(e_i)$. Then $\sigma(g)(e_i) = \sigma(g_i)$ for all $i$. Thus, \eqref{eq:coxeter_almost_cyclic_equation_GSp} implies $\varphi(g_i) =gyb(e_i)$. Using this equation for all $i \neq n$ and writing $v := g_1$, we deduce $g_i = \varphi^{i-1}(v)$ for $1\leq i \leq n-1$. In particular, $\{\varphi^i(v)\}_{i=0}^{n-1}$ is an $L$-basis of~$V$. Using $\varphi(g_i) =gyb(e_i)$ for $i=n$ and inserting the values for $g_i$, we arrive at the equation
\[
\varphi^n(v) = \varpi^\kappa v + \sum_{i=1}^{n-1} \varpi^\kappa a_i \varphi^i(v).
\]
Now Lemma~\ref{lm:isocrystal_lemma} shows that $\ord_\varpi(a_i) \geq \frac{\kappa}{n}(n-i) - \kappa = -\frac{\kappa i}{n}$. This gives \eqref{eq:MP_counts_type_A}, and we are done.

\subsubsection{Proof of Lemma~\ref{lm:filtration_on_roots_Coxeter} for type $\boldsymbol{A_{n-1}}$}\label{sec:proof_of_filtration_lm_type_An1}

The action of $c$ on $\Phi^+$ is as follows: $c(\alpha_{i-j}) = \alpha_{(i+1)-(j+1)}$ for $1\leq i < j < n$; $c(\alpha_{i-n}) \in \Phi^-$ for $1\leq i < n$. For $1\leq i\leq r = n-1$, put $\Psi_i = \{\alpha_{i'-j'} \in \Phi^+ \colon j' \geq i+1 \}$.

\subsection{Type $\boldsymbol{C_m}$}\label{sec:type_C_case_by_case}

\subsubsection{}\label{sec:basic_setup_C} Let $m\geq 2$. Let $V_0$ be a $2m$-dimensional $k$-vector space with an (ordered) basis $\{e_1, \dots, e_m,$ $e_{-m}, \dots e_{-1} \}$. Whenever convenient, we will identify $V_0$ with $k^{2m}$ via this basis. Let $\langle \cdot, \cdot \rangle$ be the alternating bilinear form on $V_0$ determined by $\langle e_{\pm i}, e_{\pm j} \rangle = 0$ if $i\neq j$ and $\langle e_i, e_{-i} \rangle = 1$ for $1\leq i\leq m$. Let $\bfG$ be the closed $k$-subgroup of ${\bf GL}(V_0)$ of elements preserving $\langle \cdot , \cdot \rangle$ up to a scalar; \textit{i.e.}, 
\[
\bfG(R) = \{g \in {\bf GL}(V_0)(R) \colon \langle gu,gv \rangle = \lambda(g)\langle u,v\rangle \text{ for all $u,v \in V_0 \otimes_k R$} \}, 
\]
where $\lambda(g) \in R^\times$ is a unit only depending on $g$. The group $\bfG$ is the \emph{group of symplectic similitudes}, and the character $\lambda\colon \bfG \rar \bG_m$ is called the \emph{similitude character}.

\subsubsection{}\label{sec:type_C_further_settings} Let $\bfT \subseteq \bfB \subseteq \bfG$ be the diagonal torus and the Borel subgroup consisting of upper triangular matrices in $\bfG$. Then $\bfT = \{(t_0t_1,\dots,t_0t_m,t_m^{-1},\dots,t_1^{-1}) \colon t_i \in \bG_m \}$ and $X_\ast(\bfT) = \bigoplus_{i=0}^m \bZ \varepsilon_i$, where 
\[
\varepsilon_0 \colon t \longmapsto (t^{\oplus m}, 1^{\oplus m})\quad \text{and} \quad\varepsilon_i \colon t \longmapsto (1^{\oplus i-1}, t, 1^{\oplus 2(m-i-1)}, t^{-1},1^{\oplus i-1}) \text{ for $1\leq i\leq m$} 
\]
are cocharacters of $\bfT$. We have $X_\ast(\bfT^{\ad}) = X_\ast(\bfT)/\bZ \cdot (2\varepsilon_0 - \sum_{i=1}^m \varepsilon_i)$ and $X_\ast(\bfT^{\sssc}) = \bigoplus_{i=1}^m \bZ\varepsilon_i$. In particular, $\pi_1(\bfG^{\ad}) = X_\ast(\bfT^{\ad})/X_\ast(\bfT^{\sssc}) \cong \bZ/2\bZ$, generated by the class of $\varepsilon_0$.

We have $\Phi = \{ \alpha_{\pm i \pm j} \colon (t_i)_{i=1}^m \mapsto t_i^{\pm 1} t_j^{\pm 1} \colon 1\leq i<j\leq m\} \cup \{\alpha_{\pm 2\cdot i} \colon (t_i)_{i=1}^m \mapsto t_i^{\pm 2} \colon 1\leq i \leq m\}$ and $\Delta = \{\alpha_{i-(i+1)} \colon 1\leq i \leq m-1\} \cup \{ \alpha_{2\cdot m}\}$. We have the Coxeter element
\begin{equation}\label{eq:special_Coxeter_type_C}
c = s_{\alpha_{1-2}} \cdot \dots \cdot s_{\alpha_{(m-1)-m}} s_{\alpha_{2m}} = (1,2,\dots,m-1,m,2m,2m-1,\dots,m+1). 
\end{equation}
We consider two lifts $\dot c_0$, $\dot c_1$ of $c$ to $\bfG(\breve k)$. The first is given by $\dot c_0(e_{\pm i})=e_{\pm (i+1)}$ for $1\leq i \leq m-1$, $\dot c_0(e_{\pm m}) = e_{\mp 1}$. The second is $\dot c_1 = \dot c_0 \varepsilon_0(\varpi)$. Then $\dot c_0$, $\dot c_1$ are basic, and the image of $\dot c_0$, resp.\ $\dot c_1$, in $\bfG^{\ad}(\breve k)$ represents $0 \in \pi_1(\bfG^{\ad})$, resp.\ $\bar \varepsilon_0 \in \pi_1(\bfG^{\ad})$. Thus the pairs $(\bfG,\sigma_{\dot c_i})$ ($i=1,2$) represent the two inner forms of~$\bfG$.

We identify $\caA_{\bfT^{\ad},\breve k}$ with $X_\ast(\bfT^{\ad})_\bR$ by sending $0 \in X_\ast(\bfT)_\bR$ to the $\sigma$-stable hyperspecial point (``the origin'') corresponding to the self-dual lattice $\bigoplus_{i=1}^m \left(\caO_{\breve k} e_i \oplus \caO_{\breve k} e_{-i}\right) \subseteq V_0 \otimes_k \breve k$.

The group ${}^c\bfU \cap \bfU^-$ is generated by root subgroups attached to $\alpha_{i-1}$ ($2\leq i\leq m$) and $\alpha_{-2\cdot 1}$. Explicitly, any element $y \in ({}^c\bfU \cap \bfU^-)(R)$ is an $R$-linear map in ${\bf GL}(V_0)(R)$ given by $e_1 \mapsto e_1 + \sum_{i=2}^m a_{i-1} e_i + a_m e_{-m}$; $e_i \mapsto e_i$ for $2\leq i \leq m$ and $i = -1$; $e_{-i} \mapsto e_{-i} - a_{i-1} e_{-1}$ for all $2\leq i \leq m$.

\subsubsection{}\label{sec:type_C_some_general_stuff} Let $b = \dot c_\kappa$ with $\kappa \in \{0,1\}$. We compute that $\bfx_{b} = 0$ if $\kappa = 0$ and that $\bfx_b = \sum_{i=1}^m \lambda_i \varepsilon_i$ with $\lambda_i = -\frac{m}{4} + \frac{i-1}{2}$ for $1\leq i \leq m$ if $\kappa= 1$. Let $g \in \dot X_{b,b}(\ff)$. Then \eqref{eq:coxeter_almost_cyclic_equation_GSp} holds with $y$ as described in Section~\ref{sec:type_C_further_settings}, depending on some $a_1,\dots,a_m \in L$. 
We have to show that $y \in \caG_{\bfx_b}(\caO_L)$, which is (as in Section~\ref{sec:computation_case_An}) equivalent to 
\begin{equation}\label{eq:MP_counts_type_C_evaluated}
\ord\nolimits_\varpi(a_i) \geq \frac{\kappa i}{2} \text{\quad $\forall\, 1\leq i\leq m$}. 
\end{equation}

Consider the $\phi$-linear isocrystal $(V,\varphi) = (V_0 \otimes_k L, b\sigma)$ over $\ff$. It is isoclinic of slope $\frac{\kappa}{2}$. For $i \in \{\pm 1, \dots,\pm m\}$, let $g_i := g(e_i)$. Then $\sigma(g)(e_i) = \sigma(g_i)$ for all $i$. Hence \eqref{eq:coxeter_almost_cyclic_equation_GSp} implies $\varphi(g_i) = gyb(e_i)$. Using this equation for all $i \neq m+1$ and writing $v := g_1$, we deduce that 
$g_i = \varpi^{\kappa(1-i)} \varphi^{i-1}(v)$ for $2\leq i \leq m$, $g_{-1} = \varpi^{-\kappa m}\varphi^m(v)$ and 
\[ 
g_{-i} = \varpi^{-\kappa m}\varphi^{m+i-1}(v) + \sum_{j=1}^{i-1} \varpi^{-\kappa m}\phi^{i-1-j}(a_j)\varphi^{m+i+1 - j}(v) \quad \text{for $2\leq i\leq m$.}
\]
In particular, $\{\varphi^i(v)\}_{i=0}^{2m-1}$ is a basis of $V$; \textit{i.e.}, $v$ is a cyclic vector for $V$. Finally, applying $\varphi(g_i) = gyb(e_i)$ for $i=m+1$ and inserting the above formulas for the $g_i$, we deduce
\[
\varphi^{2m}(v) = \varpi^{\kappa m} v + \sum_{i=1}^m \varpi^{\kappa (m-i)}a_i\varphi^i(v) - \sum_{m+1}^{2m}\phi^{i-m}(a_{2m-i})\varphi^i(v).
\]
Now Lemma~\ref{lm:isocrystal_lemma} shows that $\ord_\varpi(a_i) \geq 0$ for all $1\leq i\leq m$ if $j=0$ and that $\ord_\varpi(\varpi^{m-i}a_i) \geq \frac{1}{2}(2m-i)$, \textit{i.e.}, that $\ord_\varpi(a_i) \geq \frac{i}{2}$ for all $1\leq i \leq m$, if $j=1$. This shows \eqref{eq:MP_counts_type_C_evaluated}, and we are done.

\subsubsection{Proof of Lemma~\ref{lm:conjugation_crosssecion_proof} for types $\boldsymbol{C_m}$ and $\boldsymbol{B_m}$}\label{sec:proof_of_crosssec_lemma_type_CB} The action of $c$ on $\Phi^+$ is as follows: for $1\leq i < j < m$, $c(\alpha_{i\pm j}) = \alpha_{(i+1)\pm (j+1)}$; for $1\leq i < m$, $c(\alpha_{i-m}) = \alpha_{1 + (i+1)}$; for $1\leq i \leq m-1$, $c(\alpha_{i+m}) \in \Phi^-$; for $1\leq i \leq m-1$, $c(\alpha_{2\cdot i}) = \alpha_{2\cdot (i+1)}$; $c(\alpha_{2\cdot m}) \in \Phi^-$.

We set $r = m+1$; $\Psi_{m+1} = \{\alpha_{i+m}\}_{i=1}^{m-1} \cup \{\alpha_{2\cdot m}\}$; $\Psi_m = \{\alpha_{i+j} \colon 1\leq i<j\leq m\} \cup \{\alpha_{2\cdot i}\}_{i=1}^m$; and for $1\leq i_0 \leq m-1$, $\Psi_{i_0} = \Psi_m \cup \{\alpha_{i-j} \colon 1 \leq i < j \leq m \text{ and $i_0 < j$} \}$. Replacing all roots with the corresponding coroots, we obtain by duality the desired partition for type $B_m$ (under the induced isomorphism of Weyl groups, the special Coxeter element for type $C_m$ maps to the one in Section~\ref{sec:type_B_further_settings}).

\subsection{Type $\boldsymbol{B_m}$}\label{sec:casebycase_odd_orthogonal}

\subsubsection{}\label{sec:basic_setup_odd_orthogonal} Let $m\geq 2$. Let $V_0$ be a $(2m+1)$-dimensional $k$-vector space with an (ordered) basis $\{e_1, \dots, e_m, e_{m+1}, e_{-m}, \dots e_{-1} \}$. Whenever convenient, identify $V_0$ with $k^{2m+1}$ via this basis. Let $Q \colon V_0 \rar k$ be the split quadratic form given by $Q\left(\sum_{i=1}^m (a_i e_i + a_{-i} e_i) + a_{m+1}e_{m+1}\right) = \sum_{i=1}^m a_i a_{-i} + a_{m+1}^2$. Let $\bfG = {\bf SO}(V_0,Q)$ be the split odd orthogonal group attached to $(V_0,Q)$; \textit{i.e.},
\[
\bfG(R) = \{g \in {\bf GL}(V_0)(R) \colon \langle gu,gv \rangle = \langle u,v\rangle \text{ for all $u,v \in V_0 \otimes_k R$, $\det(g) = 1$} \}. 
\]
It is of adjoint type. (The computations below work independently of the characteristic and the residue characteristic of $k$.)

\subsubsection{}\label{sec:type_B_further_settings} Let $\bfT \subseteq \bfB \subseteq \bfG$ be the diagonal torus and the Borel subgroup consisting of upper triangular matrices in $\bfG$. Then $\bfT = \{(t_1,\dots,t_m,1,t_m^{-1},\dots,t_1^{-1}) \colon t_i \in \bG_m \}$ and $X_\ast(\bfT) = \bigoplus_{i=1}^m \bZ \varepsilon_i$, where 
\[
\varepsilon_i \colon t \longmapsto (1^{\oplus i-1}, t, 1^{\oplus 2(m-i-1)}, t^{-1},1^{\oplus i-1}) \text{ for $1\leq i\leq m$} 
\]
are cocharacters of $\bfT$. We have $X_\ast(\bfT^{\sssc}) = \{\sum_{i=1}^m a_i \varepsilon_i \in X_\ast(\bfT) \colon \sum_i a_i \,\, \text{even}\}$. Thus $\pi_1(\bfG) = X_\ast(\bfT)/X_\ast(\bfT^{\sssc}) \cong \bZ/2\bZ$, generated by the class of $\varepsilon_i$ for any $i$.

We have $\Phi = \{ \alpha_{\pm i \pm j} \colon (t_\lambda)_{\lambda = 1}^m \mapsto t_i^{\pm 1} t_j^{\pm 1} \colon 1\leq i < j \leq m\}\cup \{\alpha_i \colon (t_\lambda)_{\lambda = 1}^m \mapsto t_i \colon 1\leq i\leq m\}$ and  $\Delta = \{\alpha_{i-(i+1)} \colon 1\leq i \leq m-1\} \cup \{\alpha_m\}$. We have the Coxeter element
\begin{equation}\label{eq:special_Coxeter_type_B}
c = s_{\alpha_{1-2}} \cdot \dots \cdot s_{\alpha_{(m-1)-m}} s_{\alpha_{m}} = (1,2,\dots,m,2m+1,2m,\dots,m+2).
\end{equation}
We consider two lifts $\dot c_0$, $\dot c_1$ of $c$ to $\bfG(\breve k)$. The first is given by $\dot c_0(e_{\pm i})=e_{\pm (i+1)}$ for $1\leq i \leq m-1$, $\dot c_0(e_{\pm m}) = e_{\mp 1}$ and $\dot c_0(e_{m+1}) = -e_{m+1}$. The second is given by $\dot c_1(e_{\pm i})=e_{\pm (i+1)}$ for $1\leq i \leq m-1$, $\dot c_1(e_{\pm m}) = \varpi^{\pm 1}e_{\mp 1}$ and $\dot c_1(e_{m+1}) = -e_{m+1}$. Then $\dot c_0$, $\dot c_1$ are basic, and the image of $\dot c_0$, resp.\ $\dot c_1$, in $\bfG^{\ad}(\breve k)$ represents $0 \in \pi_1(\bfG)$, resp.\ $\bar \varepsilon_1 \in \pi_1(\bfG)$. Thus the $\bfG_{\dot c_i}$ ($i=1,2$) are the two inner forms of $\bfG$.

We identify $\caA_{\bfT,\breve k}$ with $X_\ast(\bfT)_\bR$ by sending $0 \in X_\ast(\bfT)_\bR$ to the $\sigma$-stable hyperspecial point (``the origin'') corresponding to the self-dual lattice $\bigoplus_{i=1}^m \left(\caO_{\breve k} e_i \oplus \caO_{\breve k} e_{-i}\right) \oplus \caO_{\breve k} e_{m+1} \subseteq V_0 \otimes_k \breve k$.

The group ${}^c\bfU \cap \bfU^-$ is generated by root subgroups attached to $\alpha_{i-1}$ ($2\leq i\leq m$) and $\alpha_{-1}$. Explicitly, any element $y \in ({}^c\bfU \cap \bfU^-)(R)$ is an $R$-linear map in ${\bf GL}(V_0)(R)$ given by $e_1 \mapsto e_1 + \sum_{i=2}^m a_{i-1} e_i + a_m e_{m+1} - a_m^2 e_{-1}$; $e_i \mapsto e_i$ for $2\leq i \leq m$ and $i = -1$; $e_{m+1} \mapsto e_{m+1} - 2a_m e_{-1}$; $e_{-i} \mapsto e_{-i} - a_{i-1} e_{-1}$ for $2 \leq i \leq m$.

\subsubsection{} Let $b = \dot c_\kappa$ with $\kappa \in \{0,1\}$. We have $\bfx_b = 0$ if $\kappa = 0$ and $\bfx_b = -\frac{1}{2}\sum_{i=1}^m \varepsilon_i$ if $\kappa = 1$. Let $g \in \dot X_{b,b}(\ff)$. Then \eqref{eq:coxeter_almost_cyclic_equation_GSp} holds with $y$  as described in Section~\ref{sec:type_B_further_settings}, depending on some $a_1,\dots,a_m \in L$. We have to show that $y \in \caG_{\bfx_b}(\caO_L)$, which is (as in Section~\ref{sec:computation_case_An}) equivalent to 
\begin{equation}\label{eq:MP_counts_type_B_evaluated}
\ord\nolimits_\varpi(a_i) \geq 0 \text{\quad $\forall\, 1\leq i\leq m-1$} \quad \text{ and } \ord\nolimits_\varpi(a_m) \geq \frac{\kappa}{2}.
\end{equation}

Consider the $\phi$-linear isocrystal $(V,\varphi) = (V_0 \otimes_k L, b\sigma)$ over $\ff$. It is isoclinic of slope $0$. For $i \in \{\pm 1, \dots,\pm m, m+1\}$, let $g_i := g(e_i)$. Then \eqref{eq:coxeter_almost_cyclic_equation_GSp} implies that $\varphi(g_i) = gyb(e_i)$ for all $i$. 
Using this equation for all $i \neq m+1, m+2$ and writing $v := g_1$ and $u:= g_{m+1}$, we see 
that $g_i = \varphi^{i-1}(v)$ for $2\leq i \leq m$, $g_{-1} = \varpi^{-\kappa}\varphi^m(v)$ and 
\[ 
g_{-i} = \varpi^{-\kappa}\left(\varphi^{m+i-1}(v) + \sum_{j=1}^{i-1} \phi^{i-1-j}(a_j)\varphi^{m+i-1 - j}(v)\right) \quad \text{for $2\leq i\leq m$.}
\]
In particular, it follows that $\{\varphi^i(v)\}_{i=0}^{2m-1} \cup \{u\}$ are linearly independent over $L$. Now, evaluating the equation $\varphi(g_i) = gyb(e_i)$ for $i = m+1$ and $i = m+2$ gives
\begin{align}
\label{eq:type_B_aux_eq} 2a_m \varphi^m(v) &= u + \varphi(u), \\
\label{eq:type_B_min_pol_equation} \varphi^{2m}(v) &= v + \sum_{i=1}^{m-1} a_i\varphi^i(v) - a_m^2 \varpi^{-\kappa}\varphi^m(v) - \sum_{i=m+1}^{2m}\phi^{i-m}(a_{2m-i})\varphi^i(v) + a_m u.
\end{align}
First suppose $a_m = 0$. Then \eqref{eq:type_B_min_pol_equation} together with the linear independence of $\{\varphi^i(v)\}_{i=0}^{2m-1}$ shows that $v$ spans a $2m$-dimensional sub-isocrystal of $V$ and is a cyclic vector in it. Then Lemma~\ref{lm:isocrystal_lemma} and \eqref{eq:type_B_min_pol_equation} show that $\ord_\varpi(a_i) \geq 0$ for $1\leq i \leq m-1$.

Now suppose  $a_m \neq 0$. Then $a_m \in L^\times$, and $\lambda := \frac{a_m}{\phi(a_m)} \in L^\times$ satisfies $\ord_\varpi(\lambda) = 0$. Apply $\lambda\varphi(\cdot)$ to \eqref{eq:type_B_min_pol_equation}, and add the result to \eqref{eq:type_B_min_pol_equation}; then use \eqref{eq:type_B_aux_eq} to eliminate $u$:
\begin{align*}
\lambda \varphi^{2m+1}(v) &= v + (a_1 + \lambda) \varphi(v) + \sum_{i=2}^{m-1}(a_i + \lambda \phi(a_{i-1})) \varphi^i(v) + \ast \cdot \varphi^m(v) \\ &\quad + \left(\phi(a_{m-1}) - \lambda\phi(a_m)^2\varpi^{-\kappa}\right) \varphi^{m+1}(v) + \sum_{i=m+2}^{2m} \ast \cdot \varphi^i(v),
\end{align*}
where the $\ast$ denote some unspecified elements of $L$. Moreover, $v$ is a cyclic vector for $V$. Indeed, $\{\varphi^i(v)\}_{i=0}^{2m-1} \cup\{u\}$ is a basis of $V$, and hence it follows from \eqref{eq:type_B_min_pol_equation} and $a_m \neq 0$, that $\{\varphi^i(v)\}_{i=0}^{2m}$ also is. Using the cyclicity of $v$, the last equation, the fact that $\ord_\varpi(\lambda) = 0$ and Lemma~\ref{lm:isocrystal_lemma}, we deduce by induction on $i$ that $\ord_\varpi(a_i) \geq 0$ for $i = 1,\dots,m$ and that, moreover, $\ord_\varpi(a_m) \geq \frac{1}{2}$ if $\kappa=1$. This agrees with what we had to show in \eqref{eq:MP_counts_type_B_evaluated}.

\subsection{Type $\boldsymbol{D_m}$}\label{sec:casebycase_even_orthogonal}

\subsubsection{}\label{sec:type_Dn_setup} Let $m\geq 4$. Let $V_0$ be a $2m$-dimensional $k$-vector space with an (ordered) basis $\{e_1, \dots, e_m,$ $e_{-m}, \dots e_{-1} \}$. Whenever convenient, we use this basis to identify $V_0$ with $k^{2m}$. 
Consider the quadratic form $Q\left(\sum_{i=1}^m (a_i e_i + a_{-i}e_{-i})\right) = \sum_{i=1}^m a_ia_{-i}$ on $V_0$. We have the orthogonal group $\bfO(V_0,Q) \subseteq {\bf GL}(V_0)$. Via the Clifford algebra attached to $Q$, we have the surjective Dickson morphism $D_Q \colon \bfO(V_0,Q) \rar \bZ/2\bZ$; \textit{cf.} \cite[Proposition~C.2.2]{Conrad_14}.
Its kernel is the special orthogonal group ${\bf SO}(V_0, Q)$ of $Q$. Finally, let $\bfG = {\bf GSO}(V_0,Q)$ be the closed subgroup of ${\bf GL}(V_0)$ of elements which up to a scalar lie in ${\bf SO}(V_0,Q)$.\footnote{If $\charac k\neq 2$, then the Dickson morphism is not necessary to describe $\bfG$. For any $k$-algebra $R$, we have
\[
\bfG(R) = \left\{g \in {\bf GL}(V_0)(R) \colon Q(gv) = \lambda(g)Q(v) \text{ for all $u,v \in V_0 \otimes_k R$\, and \,$\frac{\det(g)}{\lambda(g)^m} = 1$} \right\}, 
\]
where $\lambda(g) \in R^\times$ is a unit only depending on $g$.}

\subsubsection{}\label{sec:type_D_further_setting} Let $\bfT \subseteq \bfB \subseteq \bfG$ be the diagonal torus and the Borel subgroup consisting of upper triangular matrices in $\bfG$. Then $\bfT = \{(t_0t_1,\dots,t_0t_m,t_m^{-1},\dots,t_1^{-1}) \colon t_i \in \bG_m \}$ and $X_\ast(\bfT) = \bigoplus_{i=0}^m \bZ \varepsilon_i$, where 
\[
\varepsilon_0 \colon t \longmapsto (t^{\oplus m}, 1^{\oplus m}) \quad\text{and}\quad \varepsilon_i \colon t \longmapsto (1^{\oplus i-1}, t, 1^{\oplus 2(m-i)}, t^{-1},1^{\oplus i-1}) \text{ for $1\leq i\leq m$} 
\]
are cocharacters of $\bfT$. We have $X_\ast(\bfT^{\sssc}) = \bZ\Phi^\vee = \{\sum_{i=1}^m \lambda_i \varepsilon_i \colon \sum_i\lambda_i \text{\, even}\}$, $X_\ast(\bfT^{\der}) = \bigoplus_{i=1}^m \bZ\varepsilon_i$ and $X_\ast(\bfZ) = \bZ(2\varepsilon_0 - \sum_{i=1}^m\varepsilon_i)$. We have $\pi_1(\bfG^{\ad}) = X_\ast(\bfT)/\bZ\Phi^\vee + X_\ast(\bfZ)$. If $m$ is even, $\pi_1(\bfG^{\ad}) \cong \bZ/2\bZ \times \bZ/2\bZ$, generated by the classes of $\varepsilon_0$ and $\varepsilon_i$ for an arbitrary $1\leq i\leq m$. If $m$ is odd, $\pi_1(\bfG^{\ad}) \cong \bZ/4\bZ$, generated by the class of $\varepsilon_0$. In any case, $\pi_1(\bfG^{\ad})$ has the distinguished subgroup $\pi_1(\bfG^{\der}) \cong \bZ/2\bZ$, generated by the class of some (\textit{i.e.}, any) $\varepsilon_i$ ($1\leq i\leq m$).

We have $\Phi = \{ \alpha_{\pm i \pm j} \colon (t_\lambda)_{\lambda=1}^m \mapsto t_i^{\pm 1} t_j^{\pm 1} \colon 1\leq i<j \leq m\}$ and $\Delta = \{\alpha_{i-(i+1)} \colon 1\leq i \leq m-1\} \cup \{\alpha_{(m-1) + m}\}$. We have the Coxeter element
\begin{equation}\label{eq:special_Coxeter_type_D}
c = s_{\alpha_{1-2}} \cdot \dots \cdot s_{\alpha_{(m-1)-m}} s_{\alpha_{(m-1) + m}} = (1,2,\dots,m-2,m-1,2m,2m-1,\dots,m+1)(m,m+1).  
\end{equation}
We consider three lifts $\dot c_0$, $\dot c_1$, $\dot c_2$ of $c$ to $\bfG(\breve k)$. For $\kappa \in \{0,1\}$, $\dot c_\kappa$ is given by $\dot c_\kappa(e_{\pm i}) = e_{\pm(i+1)}$ for $1\leq i\leq m-2$, $\dot c_\kappa(e_{\pm (m-1)}) = e_{\pm 1}$, $\dot c_\kappa(e_{\pm m}) = \varpi^{\mp \kappa}e_{\mp m}$. The third is given by $\dot c_2 = \dot c_0 \varpi^{\varepsilon_0}$. Then $\tilde\kappa_\bfG(\dot c_0) = 0$, $\tilde\kappa_\bfG(\dot c_1)$ is the non-trivial class in $\pi_1(\bfG^{\der})$ and $\tilde\kappa_\bfG(\dot c_2) = \bar\varepsilon_0 \in \pi_1(\bfG) \sm \pi_1(\bfG^{\der})$. Thus the $\bfG_{\dot c_\kappa}$ ($\kappa\in \{0,1,2\}$) are three of the four inner forms of $\bfG$. The fourth inner form is conjugate to $(\bfG,\dot c_2)$ by the outer automorphism of $\bfG$ induced by an inner automorphism of the orthogonal group ${\bf GO}(V_0)$; hence we need not to consider it separately.

We identify $\caA_{\bfT^{\ad},\breve k}$ with $X_\ast(\bfT^{\ad})_\bR$ by sending $0 \in X_\ast(\bfT)_\bR$ to the $\sigma$-stable hyperspecial point (``the origin'') corresponding to the self-dual lattice $\bigoplus_{i=1}^m \left(\caO_{\breve k} e_i \oplus \caO_{\breve k} e_{-i}\right) \subseteq V_0 \otimes_k \breve k$.

The group ${}^c\bfU \cap \bfU^-$ is generated by root subgroups attached to $\alpha_{i-1}$ ($2\leq i\leq m$) and $\alpha_{- 1 - m}$. Explicitly, any element $n \in ({}^c\bfU \cap \bfU^-)(R)$ is an $R$-linear map in ${\bf GL}(V_0)(R)$ given by $e_1 \mapsto e_1 + \sum_{i=2}^m a_{i-1} e_i + a_m e_{-m} - a_{m-1}a_m e_{-1}$; $e_i \mapsto e_i$ for $2\leq i \leq m-1$ and $i = -1$; $e_m \mapsto e_m - a_me_{-1}$; $e_{-i} \mapsto e_{-i} - a_{i-1} e_{-1}$ for all $2\leq i \leq m$.

\subsubsection{} Let $b = \dot c_\kappa$ with $\kappa \in \{0,1,2\}$. We compute $\bfx_b = 0$ if $\kappa = 0$; $\bfx_b = -\frac{1}{2}\varepsilon_m$ if $\kappa = 1$; $\bfx_b = -\frac{1}{4}\varepsilon_m + \sum_{i=1}^{m-1} (-\frac{1}{4} + \frac{1}{2}(i-1))\varepsilon_i$ if $\kappa = 2$. Let $g \in \dot X_{b,b}(\ff)$. Then \eqref{eq:coxeter_almost_cyclic_equation_GSp} holds, with $n$ as described in Section~\ref{sec:type_D_further_setting}, depending on some $a_1,\dots,a_m \in L$. We have to show that $y \in \caG_{\bfx_b}(\caO_L)$, which (as in Section~\ref{sec:computation_case_An}) is equivalent to 
\begin{align}
\nonumber &\ord\nolimits_\varpi(a_i) \geq 0 \,\text{ for } 1\leq i\leq m & \text{if $\kappa=0$,}
\\
\label{eq:show_valuations_bigger_sth_type_D}&\ord\nolimits_\varpi(a_i) \geq 0 \,\text{ for } 1\leq i \leq m-1 \;\text{ and } \;\ord\nolimits_\varpi(a_m) \geq 1 & \text{if $\kappa=1$, }
\\
\nonumber &\ord\nolimits_\varpi(a_i) \geq \frac{i}{2} \,\text{ for } 1\leq i \leq m-2,\; \ord\nolimits_\varpi(a_{m-1}) \geq \frac{m-2}{4} 
\;\text{ and } \;\ord\nolimits_\varpi(a_m) \geq \frac{m}{4}& \text{if $\kappa=2$.}
\end{align}

Consider the $\phi$-linear isocrystal $(V,\varphi) = (V_0 \otimes_k L, b\sigma)$ over $\ff$. It is isoclinic of slope $0$ if $\kappa \in \{0,1\}$ and slope $\frac{1}{2}$ if $\kappa = 2$. For $i \in \{\pm 1, \dots,\pm m\}$, put $g_i := \varphi(e_i)$. Then \eqref{eq:coxeter_almost_cyclic_equation_GSp} implies that $\varphi(g_i) = gyb(e_i)$ for all $i$. Using this equation for all $i \neq m+1,m+2$ and  writing $v := g_1$ and $u:= g_m$, we compute that $g_i = \varpi^{A_i} \varphi^{i-1}(v)$ for $2\leq i \leq m$, $g_{-1} = \varpi^{A_{-1}}\varphi^m(v)$, $g_{-m} = \varpi^{A_{-m}}\varphi(u) + \varpi^{A'_{-m}} a_{m-1}\varphi^{m-1}(v)$ and 
\[ 
g_{-i} = \varpi^{A_{-i}} \left(\varphi^{m+i-2}(v) + \sum_{d=m-1}^{m-3+i} \phi^{d-m+1}(a_{m-d+i-2})\varphi^d(v)\right) \quad \text{for $2\leq i\leq m-1$,}
\]
where $A_i = A'_{-m} = 0$ for all $i \in \{\pm 1,\dots,\pm m\}$ if $\kappa=0$; $A_i = A'_{-m} = 0$ for all $i \neq -m$ and $A_{-m} = -1$ if $\kappa=1$; $A_i = 1-i$ for $2\leq i\leq m-1$, $A_{-1} = A'_{-m} = A_{-i} = 1-m$ for $2\leq i \leq m-1$ and $A_{-m} = -1$ if $\kappa=2$. 

\begin{lm}\label{lm:lin_indep_of_vi_u}
The vectors $\{\varphi^i(v)\}_{i=0}^{2m-3}$, $u, \varphi(u)$ form a basis of\, $V$.
\end{lm}

\begin{proof}
From the formulas for the $g_{\pm i}$ above, it follows that the vectors are linearly independent, hence form a basis by dimension reasons.
\end{proof}

Now, evaluating the equation $\varphi(g_i) = gyb(e_i)$ for $i = m+1$ and $i = m+2$ gives the formulas
\begin{align}
\label{eq:isocrystal_equation_v_1_type_D}
\varphi^{2m-2}(v) &= \varpi^{B_0} v + \sum_{i=1}^{m-2} \varpi^{B_i} a_i \varphi^i(v) - \sum_{i=m}^{2m-3} \phi^{i - m + 1}(a_{2m-2 - i}) \varphi^i(v) + \varpi^{B_0'} a_{m-1} u + \varpi^{B_1'} a_m \varphi(u), 
\\
\varphi^2(u) &= \varpi^{C_0} u - \varpi^{C_1} \phi(a_{m-1})\varphi^m(v) - \varpi^{C_2} a_m \varphi^{m-1}(v), 
\label{eq:isocrystal_equation_u_1_type_D}
\end{align}
where all $B_i,B_i',C_i$ are $0$ if $\kappa=0$; $B_1' = -1$, $C_1 = 1$ and all other $B_i,B_i',C_i$ are $0$ if $\kappa=1$; $B_i = m-1 - i$ for $0\leq i \leq m-2$, $B_0' = B_1' = m-2$, $C_0=1$ and $C_1 = C_2 = 2-m$ if $\kappa=2$. Put 
\[ 
\mu := \varpi^D a_m \phi(a_m) - a_{m-1}\phi(a_{m-1}) \in L,
\] 
where $D = 0$ if $\kappa=0$; $D = -2$ if $\kappa=1$; $D = 1$ if $\kappa = 2$. We have five cases, which we consider separately in the following subsections.

\subsubsection{Case $\boldsymbol{a_m \neq 0}$, $\boldsymbol{a_{m-1} \neq 0}$, $\boldsymbol{\mu \neq 0}$}\label{sec:generic_case_type_D} Suppose  $a_ma_{m-1}\mu \neq 0$. Apply $\frac{a_m}{\phi(a_{m-1})}\varphi(\cdot)$ to \eqref{eq:isocrystal_equation_v_1_type_D}, and subtract \eqref{eq:isocrystal_equation_v_1_type_D} from the resulting equation. This gives an expression of $\frac{a_m}{\phi(a_{m-1})} \varphi^{2m-1}(v)$ as a linear combination of $\{\varphi^i(v)\}_{i=0}^{2m-2}$, $\varphi^2(u)$ and $u$.  Substitute \eqref{eq:isocrystal_equation_u_1_type_D} into this expression to eliminate $\varphi^2(u)$.  The coefficient of $u$ in the resulting equation is $\mu \neq 0$, and it gives an expression of $u$ as a linear combination of $\{\varphi^i(v)\}_{i=0}^{2m-1}$:
\begin{align}
\nonumber\varpi^E u \ =\ &\varpi^{E_0} \frac{\phi(a_{m-1})}{\mu} v - \frac{\varpi^{E_1}a_m - \varpi^{E_1'}\phi(a_{m-1}) a_1}{\mu} \varphi(v) 
- \sum_{i=2}^{m-2} \frac{\varpi^{E_i} a_m \phi(a_{i-1}) - \varpi^{E_i'} \phi(a_{m-1}) a_i}{\mu} \varphi^i(v) \\
\label{eq:expr_of_u_through_v_type_D}&- \frac{\varpi^{E_{m-1}} a_m \phi(a_{m-2}) - \varpi^{E_{m-1}'} a_m^2 \phi(a_m)}{\mu} \varphi^{m-1}(v) \\ 
\nonumber &- \frac{\phi(a_{m-1})\phi(a_{m-2}) - \varpi^{E_m'} \phi(a_{m-1})a_m \phi(a_m)))}{\mu} \varphi^m(v) 
+ \sum_{i=m+1}^{2m-1} \ast \cdot \varphi^i(v) + \varpi^{E_{2m-1}} \frac{a_m}{\mu}\varphi^{2m-1}(v),
\end{align}
where the $\ast \in L$ denote unspecified coefficients; all $E,E_i,E'_i = 0$ if $\kappa=0$; $E = E_0 = E_i' = 0$ ($1\leq i \leq m-2$), $E_m' = E_{2m-1} = E_i = -1$ ($1\leq i \leq m-1$), $E_{m-1}' = -2$ if $\kappa=1$; $E = m-2$, $E_0 = m-1$, $E_i = m-i$ and $E_i' = m-1-i$ for $1\leq i \leq m-1$, $E_m' = E_{2m-1} = 0$ if $\kappa=2$. 

As $a_m \neq 0$, the coefficient of $\varphi^{2m-1}(v)$ in \eqref{eq:expr_of_u_through_v_type_D} is non-zero. Using this observation, as well as the version of it after applying $\varphi(\cdot)$, shows, together with Lemma~\ref{lm:lin_indep_of_vi_u}, that $v$ is a cyclic vector for $V$. Now insert \eqref{eq:expr_of_u_through_v_type_D} into \eqref{eq:isocrystal_equation_v_1_type_D} to eliminate $u$ there. This gives (after multiplication with an element in $L^\times$) an expression 
\[
\frac{\mu}{\phi(\mu)}\varphi^{2m}(v) = \sum_{i=0}^{2m-1}\beta_i \varphi^i(v)
\] 
with $\beta_i \in L$ satisfying $\ord_\varpi(\beta_i) \geq (2m-i) \cdot \lambda$, where $\lambda$ is the slope of $(V,\varphi)$, by Lemma~\ref{lm:isocrystal_lemma}. Let $\gamma := - \frac{a_{m-1}}{\phi(a_m)} + \frac{\phi^2(a_{m-1})\mu}{\phi(a_m)\phi(\mu)}$.

\begin{lm}\label{lm:special_valuation_comp}
We have $\ord_\varpi(\gamma) \geq 0$ if\, $\kappa = 0$, $\ord_\varpi(\gamma) \geq -1$ if\, $\kappa = 1$ and $\ord_\varpi(\gamma) \geq 1$ if\, $\kappa = 2$. 
\end{lm}

\begin{proof}
One checks that $\gamma = \varpi^D \cdot \frac{\phi^2(a_{m-1})a_m - \phi^2(a_m)a_{m-1}}{\phi(\mu)}$. First assume $\kappa=0$. If $\ord_\varpi(a_{m-1}) > \ord_\varpi(a_m)$, then $\ord_\varpi(\phi^2(a_{m-1})a_m - \phi^2(a_m)a_{m-1}) \geq \ord_\varpi(a_{m-1}) + \ord_\varpi(a_m) \geq 2\ord_\varpi(a_m) = \ord_\varpi(\phi(\mu))$, and we are done as $D = 0$ in this case. One concludes similarly in the case $\ord_\varpi(a_m) < \ord_\varpi(a_{m-1})$. Finally, suppose $\ord_\varpi(a_{m-1}) = \ord_\varpi(a_m) =: \alpha$. Write $x = \varpi^\alpha a_m$, $y = \varpi^\alpha a_{m-1}$, so that $x,y \in \caO_L^\times$. Write $A := x\phi(x) - y\phi(y), B := x\phi^2(y) - y\phi^2(x) \in \caO_L$. It suffices to show that $\ord_\varpi(B) \geq \ord_\varpi(A)$. We now compute in $\caO_L/\left(\varpi^{\ord_\varpi(A)}\right)$, where we surely have $A = 0$, \textit{i.e.}, $x\phi(x) = y\phi(y)$, or equivalently $\phi\left(\frac{x}{y}\right) = \left(\frac{x}{y}\right)^{-1}$. Exploiting this equation twice, we see $\phi^2\left(\frac{x}{y}\right) = \phi\bigl( \left(\frac{x}{y}\right)^{-1}\bigr) = \frac{x}{y}$,
which shows that $\phi^2(x)y = x\phi^2(y)$ in $\caO_L\,/\left(\varpi^{\ord_\varpi(A)}\right)$, \textit{i.e.} $B = 0$ in $\caO_L\,/\left(\varpi^{\ord_\varpi(A)}\right)$, \textit{i.e.}, $\ord_\varpi(B) \geq \ord_\varpi(A)$.

Suppose $\kappa=1$. Write $a_m = \varpi a_m'$. Then $\gamma = \varpi^{-1} \cdot \frac{\phi^2(a_{m-1})a_m' - \phi^2(a_m')a_{m-1}}{\phi(a_m'\phi(a_m') - a_{m-1}\phi(a_{m-1}))}$, where the fraction is just $\gamma$ from case $\kappa=0$ for $a_{m-1}$ and $a_m'$, and we are done by applying the above.

Finally, suppose $\kappa=2$. Then $\mu = \varpi a_m \phi(a_m) - a_{m-1} \phi(a_{m-1})$, and  the claim is immediately checked in both cases, $\ord_\varpi(a_m) < \ord_\varpi(a_{m-1})$ and the opposite.
\end{proof}

Now we proceed case by case.

\subsubsection*{\it Case~$b = \dot c_0$} The slope is $\lambda = 0$; \textit{i.e.}, $\ord_\varpi(\beta_i) \geq 0$ for all $i$. One computes $\beta_1 = a_1 + \gamma$ and $\beta_2 = a_2 + \phi(a_1)\gamma - \frac{\mu}{\phi(\mu)}$. 
As $\ord_\varpi\left(\frac{\mu}{\phi(\mu)}\right) = 0$, Lemma~\ref{lm:special_valuation_comp} implies $\ord_\varpi(a_i) \geq 0$ for $i=1,2$. Using this as the first step, and computing $\beta_i = a_i - \frac{\mu\phi^2(a_{i-2})}{\phi(\mu)} + \phi(a_{i-1})\gamma$ for $3\leq i \leq m-2$, we conclude by induction (using Lemma~\ref{lm:special_valuation_comp}) that $\ord_\varpi(a_i) \geq 0$ for all $1\leq i \leq m-2$. Next, we have $\beta_{m-1} = a_{m-1}a_m - \frac{\mu\phi^2(a_{m-3})}{\phi(\mu)} + \phi(a_{m-2})\gamma$, which (using the above) implies that $\ord_\varpi(a_{m-1}a_m) \geq 0$. Thus, if $\ord_\varpi(a_{m-1}) = \ord_\varpi(a_m)$, we deduce that this number is at least $0$, and we are done according to \eqref{eq:show_valuations_bigger_sth_type_D}. Finally, suppose  $\ord_\varpi(a_{m-1}) \neq \ord_\varpi(a_m)$. We have $\beta_m = -\left(1+ \frac{\mu}{\phi(\mu)}\right)\phi(a_{m-2}) + a_{m-1}\phi(a_{m-1}) + \frac{\mu}{\phi(\mu)}\phi(a_m)\phi^2(a_m)$, which shows (as $\ord_\varpi(\beta_m) \geq 0$) that $\ord_\varpi\left(a_{m-1}\phi(a_{m-1}) + \frac{\mu}{\phi(\mu)}\phi(a_m)\phi^2(a_m)\right) \geq 0$. As the two summands have different valuations by assumption, both valuations must be non-negative, which implies $\ord_\varpi(a_{m-1}),\ord_\varpi(a_m) \geq 0$. We are done by \eqref{eq:show_valuations_bigger_sth_type_D}.

\subsubsection*{\it Case~$b = \dot c_1$} The slope is again $\lambda = 0$, so that $\beta_i \geq 0$ for all $i$. We have $\beta_1 = a_1 + \varpi \gamma$, $\beta_2 = a_2 - \frac{\mu}{\phi(\mu)} + \varpi\phi(a_1) \gamma$ and $\beta_i = a_i - \frac{\mu\phi^2(a_{i-2})}{\phi(\mu)} + \varpi\phi(a_{i-1})\gamma$ for $3\leq i \leq m-2$. As in the case $b = \dot c_0$, we deduce $\ord_\varpi(a_i) \geq 0$ for $1\leq i \leq m-2$. We have $\beta_{m-1} = a_{m-1} a_m + \varpi\phi(a_{m-2}) \gamma$. Using Lemma~\ref{lm:special_valuation_comp}, we deduce that if $\ord_\varpi(a_{m-1}) = \ord_\varpi(a_m) - 1$, then this number must be non-negative, and we are done according to \eqref{eq:show_valuations_bigger_sth_type_D}.  Assume $\ord_\varpi(a_{m-1}) \neq \ord_\varpi(a_m) - 1$. We compute $\beta_m = -\left(1 + \frac{\mu}{\phi(\mu)}\right)\phi(a_{m-2}) + \varpi a_{m-1}\phi(a_{m-1}) + \varpi^{-1} \frac{\mu}{\phi(\mu)} \phi(a_m)\phi^2(a_m)$, and it follows that $\ord_\varpi(a_{m-1}), \ord_\varpi(a_m) - 1 \geq 0$.

\subsubsection*{\it Case~$b = \dot c_2$} The slope is $\lambda = \frac{1}{2}$, so that for any $i$, $\ord_\varpi(\beta_i) \geq m- \frac{i}{2}$, or equivalently $\ord_\varpi(\beta_i)\geq m-\lfloor \frac{i}{2} \rfloor$. We have $\beta_1 = \varpi^{m-1} (a_1 + \gamma)$, $\beta_2 = \varpi^{m-2}(a_2 + \gamma \phi(a_1) - \varpi \frac{\mu}{\phi(\mu)})$, $\beta_i = \varpi^{m-i}(a_i + \gamma\phi(a_{i-1}) - \varpi \frac{\mu}{\phi(\mu)} \phi^2(a_{i-2}))$ for $3\leq i\leq m-2$, $\beta_{m-1} = a_{m-1}a_m + \varpi \gamma \phi(a_{m-2}) - \varpi^2 \frac{\mu}{\phi(\mu)} \phi^2(a_{m-3})$ and $\beta_m = -\varpi(1+\frac{\mu}{\phi(\mu)})\phi(a_{m-2}) + a_{m-1}\phi(a_{m-1}) + \frac{\mu}{\phi(\mu)}\phi(a_m) \phi^2(a_m)$. As in the previous two cases, we deduce that \eqref{eq:show_valuations_bigger_sth_type_D} holds.


In the following cases, we do not give all details as the computations are similar to (and easier than) those in Section~\ref{sec:generic_case_type_D}.

\subsubsection{Case $\boldsymbol{a_{m-1} \neq 0}$, $\boldsymbol{a_m = 0}$}\label{sec:case_am1neq0_am0} In this case \eqref{eq:isocrystal_equation_u_1_type_D} gives $\varphi^m(v) = \frac{1}{\phi(a_{m-1})} (u-\varphi^2(u))$. By Lemma~\ref{lm:lin_indep_of_vi_u} this implies that $u$ is a cyclic vector for $V$. Apply $\varphi^m(\cdot)$ to \eqref{eq:isocrystal_equation_v_1_type_D}, and insert the above expression of $\varphi^m(v)$ in terms of $u$ to eliminate $v$. This gives an expression $\varphi^{2m}(u) = \sum_{i=0}^{2m-1} \beta_i\varphi^i(u)$. Making all $\beta_i$ explicit for $1\leq i \leq m$ and using Lemma~\ref{lm:isocrystal_lemma}, one immediately deduces \eqref{eq:show_valuations_bigger_sth_type_D}.

\subsubsection{Case $\boldsymbol{a_{m-1} = 0}$, $\boldsymbol{a_m \neq 0}$} This case is completely analogous to the one in Section~\ref{sec:case_am1neq0_am0} (one has $\varphi^{m-1}(v) = \frac{1}{a_{m-1}} (u-\varphi^2(u))$, from which the cyclicity of $u$ and an explicit formula $\varphi^{2m}(u) = \sum_{i=0}^{2m-1}\beta_i\varphi^i(u)$ follow).

\subsubsection{Case $\boldsymbol{a_{m-1} = 0}$, $\boldsymbol{a_m = 0}$} By Lemma~\ref{lm:lin_indep_of_vi_u} $v$ is a cyclic vector for the $(2m-2)$-dimensional subisocrystal with $L$-basis $\{ \varphi^i(v) \}_{i=0}^{2m-3}$. The result follows easily from Lemma~\ref{lm:isocrystal_lemma} by looking at the coefficients of \eqref{eq:isocrystal_equation_v_1_type_D}. 

\subsubsection{Case $\boldsymbol{a_{m-1} \neq 0}$, $\boldsymbol{a_m \neq 0}$, $\boldsymbol{\mu = 0}$} (Note that $\kappa \neq 2$ in this case.) As $a_{m-1} \neq 0$, \eqref{eq:isocrystal_equation_v_1_type_D} along with Lemma~\ref{lm:lin_indep_of_vi_u} show that the vectors $\{\varphi^i(v)\}_{i=0}^{2m-2}$ are linearly independent. Further, following the same steps as in the first lines of Section~\ref{sec:generic_case_type_D} and exploiting $\mu = 0$, we arrive at an expression $\frac{a_m}{\phi(a_{m-1})} \varphi^{2m-1}(v) = \sum_{i=0}^{2m-2}\beta_i \varphi^i(v)$, thus showing that $v$ is a cyclic vector for the $(2m-1)$-dimensional subisocrystal generated by it. Proceeding as in the previous cases, one verifies \eqref{eq:show_valuations_bigger_sth_type_D}. This finishes the proof in type $D_m$.

\subsubsection{Proof of Lemma~\ref{lm:conjugation_crosssecion_proof} for type $\boldsymbol{D_m}$}\label{sec:proof_of_crosssec_lemma_type_D}
The action of $c$ on $\Phi^+$ is as follows: for $1\leq i< j<m-1$, $c(\alpha_{i\pm j}) = \alpha_{(i+1)\pm (j+1)}$; for $1\leq i\leq m-2$, $c(\alpha_{i-(m-1)}) = \alpha_{1+(i+1)}$; for $1 \leq i \leq m-2$, $c(\alpha_{i \pm m}) = \alpha_{(i+1) \mp m}$; for $1\leq i \leq m-2$, $c(\alpha_{i+(m-1)}) \in \Phi^-$, $c(\alpha_{(m-1)\pm m}) \in \Phi^-$.

Set $r = m+1$; $\Psi_{m+1} = \{\alpha_{i+(m-1)}\}_{i=1}^{m-2} \cup \{\alpha_{(m-1) - m}, \alpha_{(m-1)+m}\}$; $\Psi_m = \{\alpha_{i+j} \colon 1\leq i < j \leq m \} \cup \{\alpha_{(m-1) - m}, \alpha_{(m-1)+m}\}$; $\Psi_{m-1} = \Psi_m \cup \{\alpha_{i-m}\}_{i=1}^{m-2} \cup \{\alpha_{i+m}\}_{i=1}^{m-2}$; for $1 \leq i_0 \leq m-2$, $\Psi_{i_0} = \Psi_{m-1} \cup \{\alpha_{i-j} \colon 1\leq i<j\leq m-1 \text{ and } j> i_0\}$. 

\subsection{Type $\boldsymbol{{}^2A_{n-1}}$}\label{sec:type_twisted_An}

\subsubsection{}\label{sec:type_2An1_setup} Let $n \geq 2$. Let $m = \lfloor \frac{n}{2}\rfloor$, so that $n = 2m + 1$ if $n$ odd and $n=2m$ if $n$ is even. Let $V_0, e_i$ be as in Section~\ref{sec:basic_setup_An1}.
Let $\bfG$ be the unitary group attached to the standard hermitian form on $V_0 \otimes_k k_2$, relative to the unramified extension $k_2/k$ of degree $2$.
We may identify $\bfG \otimes_k k_2 = {\bf GL}_n(V_0 \otimes_k k_2)$, so that for any $\bF_q$-algebra $R_0$ with $\widetilde R = \bW(R_0 \otimes_{\bF_q} \obF)[1/\varpi]$, the Frobenius $\sigma$ on $\bfG(\widetilde R) = {\bf GL}_n(V_0)(\widetilde R)$ is given by
\[
g \longmapsto \sigma(g) = J \sigma_0(g)^{-1,T} J^{-1}, 
\]
where $\sigma_0$ is the Frobenius on ${\bf GL}(V_0)(\widetilde R)$ corresponding to the natural $k$-structure on ${\bf GL}(V_0)$, $(\cdot)^T$ is transposition and $J \in {\bf GL}_n(V_0)(k)$ is given by $e_i \mapsto e_{n+1-i}$ ($1\leq i\leq n$).

\subsubsection{}\label{sec:type_2An1_roots_etc} 
Let $\bfT,\bfB,\beta_i,\alpha_{i-j}$ be as in Section~\ref{sec:type_A_further_setup}. Note that $\bfT,\bfB$ are both $k$-rational (with respect to $\sigma$). Moreover, the action of $\sigma$ on $\pi_1(\bfG^{\ad})$ is by multiplication with $-1$. Thus, $\pi_1(\bfG^{\ad})_{\langle \sigma \rangle} = 1$ if $n$ is odd, and $\pi_1(\bfG^{\ad})_{\langle \sigma \rangle} \cong \bZ/2\bZ$, with the non-trivial element represented by $\bar\beta_i$ for any $i$, when $n$ is even. We have an isomorphism $\bZ\isor\pi_1(\bfG)$, $1 \mapsto \text{class of $\beta_1$}$, under which $X_\ast(\bfZ) \subseteq \pi_1(\bfG)$ corresponds to $n\bZ$. The action of $\sigma$ corresponds to multiplication by $-1$. Justifying Remark~\ref{rem:trivial_kernel_exception} in this case, we deduce from this that the induced map $X_\ast(\bfZ)_{\langle \sigma\rangle} \cong \bZ/2\bZ \rar \bZ/2\bZ \cong \pi_1(\bfG)_{\langle \sigma \rangle}$ is an isomorphism when $n$ odd and is $0$ when $n$ even.

The sets $\Phi$ and $\Delta$ are the same as in Section~\ref{sec:type_A_further_setup}. The action of $\sigma$ on $\Delta$ is given by $\sigma \colon \alpha_{i-(i+1)} \mapsto \alpha_{(n-i) - (n-i+1)}$. 
We have the Coxeter element
\begin{equation}\label{eq:special_Coxeter_element_type_2A}
c = s_{\alpha_{1-2}} \cdot \dots \cdot s_{\alpha_{m-(m+1)}} = (1,2,\dots,m,m+1).
\end{equation}
We consider the lift $\dot c_0$ of $c$ to $\bfG(\breve k)$, given by $\dot c_0(e_i) = e_{i+1}$ for $1\leq i \leq m$, $\dot c_0(e_{m+1}) = e_1$ and $\dot c_0(e_i) = e_i$ for the remaining $i$. If $n = 2m$ is even, we also consider the lift $\dot c_1$ given by $\dot c_1(e_i) = \dot c_0(e_i)$ for $i \neq m+1$, $\dot c_1(e_{m+1}) = \varpi e_1$. Then $\dot c_0$, $\dot c_1$ are basic, and if $n$ is even, $\bar\kappa_\bfG(\dot c_0) = 0$, $\bar\kappa_\bfG(\dot c_1) \neq 0$ in $\pi_1(\bfG)_{\langle \sigma \rangle}$. Thus, if $n = 2m+1$ is odd, the quasi-split group $\bfG$ (represented by the pair $(\bfG,\dot c_0)$) has no non-trivial inner forms; if $n=2m$ is even, there are two inner forms, represented by the pairs $(\bfG,\dot c_\kappa)$ with $\kappa \in \{0,1\}$.

We identify $\caA_{\bfT^{\ad},\breve k}$ with $X_\ast(\bfT^{\ad})_\bR$ by sending $0$ to the $\sigma$-stable hyperspecial point (``the origin'') corresponding to the standard lattice $\bigoplus_{i=1}^n \caO_{\breve k}e_i$.

The group ${}^c\bfU \cap \bfU^-$ is generated by the root subgroups attached to $\alpha_{i-1}$ ($2\leq i \leq m+1$). For a $\breve k$-algebra $R$, any element $y \in ({}^c\bfU \cap \bfU^-)(R)$ is an element of ${\bf GL}(V_0)(R)$, determined by $e_1 \mapsto e_1 + \sum_{\lambda=1}^m a_\lambda e_{\lambda+1}$, $e_i \mapsto e_i$ ($2\leq i \leq n$) for some $a_1, \dots, a_m \in R$.

\subsubsection{}\label{sec:2An1_computation} Let $b = \dot c_\kappa$ with $\kappa=0$ if $n$ odd and $\kappa \in \{0,1\}$ if $n$ even. We compute that ${\bf x}_b = 0$ if $\kappa = 0$ and that ${\bf x}_b$ is the image of $\sum_{\lambda = 1}^m \frac{1}{2} \beta_\lambda - \sum_{\lambda = m+2}^n \frac{1}{2} \beta_\lambda$ under $X_\ast(\bfT)_\bR \tar X_\ast(\bfT^{\ad})_\bR \cong \caA_{\bfT^{\ad}, \breve k}$ if $\kappa = 1$ (and $n=2m$ even). Let $g \in \dot X_{b,b}(\ff)$. Then \eqref{eq:coxeter_almost_cyclic_equation_GSp} holds with $y$ as described in Section~\ref{sec:type_2An1_roots_etc}, depending on some $a_1,\dots, a_m \in L$. We have to show that $y \in \caG_{{\bf x}_b}(\caO_L)$, which is (as in Section~\ref{sec:computation_case_An}) equivalent to 
\begin{equation}\label{eq:show_valuations_bigger_sth_type_2An1}
\ord\nolimits_\varpi(a_i) \geq 0 \quad \forall\, 1 \leq i\leq m
\end{equation}
(independently of what $\kappa$ is).

Let $c^{\spl} = c\sigma(c)$, and let $\dot c^{\spl}_\kappa = \dot c_\kappa \sigma(\dot c_\kappa)$. We write $b^{\spl}$ for $\dot c^{\spl}$ if $b = \dot c_\kappa$. We have the automorphism $b\sigma \colon g \mapsto \dot b\sigma(g)$ of $L\bfG$ (regarded as a sheaf on $\Perf_{\obF}$). As $\sigma^2 = \sigma_0^2$ (recall the identification in Section~\ref{sec:type_2An1_setup}), we have 
\begin{equation}\label{eq:bsigma_twice_type_2An1} 
(b\sigma)^2 = b^{\spl}\sigma^2 = b^{\spl}\sigma_0^2. 
\end{equation}

Let $V = V_0 \otimes_k L$, and denote by $\sigma_0 \colon V \rar V$ the Frobenius automorphism $v \otimes \lambda \mapsto v \otimes \phi(\lambda)$. Then $\varphi := b^{\spl}\sigma_0^2 \colon V \rar V$ is a $\phi^2$-linear automorphism of $V$. Moreover, $(V,\varphi)$ is an isocrystal over $\ff$ (relative to $\bF_{q^2}$) of slope $0$.

For $1 \leq i \leq n$, let $g_i := g(e_i)$. Then $\sigma_0^2(g)(e_i) = \sigma_0^2(g_i)$. Using \eqref{eq:bsigma_twice_type_2An1}, \eqref{eq:coxeter_almost_cyclic_equation_GSp} applied twice gives the equation 
\begin{equation}\label{eq:type_2An1_cyclic_equation_pre}
b^{\spl} \sigma_0^2(g) = g \cdot \left(y\cdot {}^b\sigma(y)\right) \cdot b^{\spl}.
\end{equation} 
Applying both sides of this to $e_i \in V$ and using the above, for $1\leq i \leq n$, we get 
\begin{equation}\label{eq:type_2An1_cyclic_equation}
\varphi(g_i) = g \cdot \left(y\cdot {}^b\sigma(y)\right) \cdot b^{\spl}(e_i).
\end{equation}

\subsubsection{} Assume that $n=2m+1$ is odd (and $b = \dot c_0$). Then \[ 
c^{\spl} = (1,2, \dots, m, m+1,2m+1,2m, \dots, m+3,m+2),
\]
$b^{\spl}$ is its lift with all entries $1$ or $0$, and $y\cdot {}^b\sigma(y)$ is given by $e_1 \mapsto e_1 + \sum_{\lambda = 1}^m a_i e_{i+1} - \phi(a_m)e_{2m+1}$; $e_i \mapsto e_i$ for $2 \leq i \leq m+1$ and $i=2m+1$; $e_i \mapsto e_i - \phi(a_{2m+1-i})e_{2m+1}$ for $m+2 \leq i \leq 2m$. Put $v := g_1$. Applying \eqref{eq:type_2An1_cyclic_equation} for all $i \neq m+2$, we deduce that $g_i = \varphi^{i-1}(v)$ for $1\leq i \leq m+1$, $g_i = \varphi^{3m+2-i}(v) + \sum_{\lambda = 1}^{2m+1-i} \phi^{1 + 2(2m-i-\lambda+1)}(a_\lambda) \phi^{3m+2-i-\lambda}(v)$ for $m+2 \leq i \leq 2m+1$. As $g \in {\bf GL}(V_0)(\breve k)$, the set $\{g_i\}_{i=1}^{2m+1}$ is a basis of $V$. This easily implies that the set $\{\varphi^i(v)\}_{i=0}^{2m}$ is a basis of $V$; \textit{i.e.}, $v$ is a cyclic vector for $(V,\varphi)$. Now, applying \eqref{eq:type_2An1_cyclic_equation} with $i = m+2$ and inserting the values of $g_i$, we deduce
\[
\varphi^{2m+1}(v) = v + \sum_{i=1}^m a_i \varphi^i(v) - \phi(a_m)\varphi^{m+1}(v) - \sum_{i = m+2}^{2m} \phi^{1+2(i-m-1)}(a_{2m-i+1})\varphi^i(v). 
\]
Applying Lemma~\ref{lm:isocrystal_lemma}, we deduce \eqref{eq:show_valuations_bigger_sth_type_2An1}.

\subsubsection{}\label{sec:type_2An1_even_comps} Assume that $n = 2m$ is even. Then 
\[
c^{\spl} = (1,2, \dots, m-1,m, 2m, 2m-1, \dots, m+3,m+2),
\]
$c^{\spl}_0$ is its lift with all entries $1$ or $0$, and $c^{\spl}_1$ is given by $c^{\spl}_1(e_m) = \varpi^{-1}e_{2m}$, $c^{\spl}_1(e_{m+2}) = \varpi e_1$, $c^{\spl}_1(e_i) = c^{\spl}_0(e_i)$ for all $i \neq m,m+2$. Next, $y\cdot {}^b\sigma(y)$ is given by $e_1 \mapsto e_1 + \sum_{\lambda = 1}^m a_i e_{i+1} - \varpi^{-\kappa}\phi(a_{m-1})e_{2m}$; $e_i \mapsto e_i$ for $2 \leq i \leq m$ and $i=2m$; $e_{m+1} \mapsto e_{m+1} - \phi(a_m)e_{2m}$;  $e_i \mapsto e_i - \phi(a_{2m-i})e_{2m}$ for $m+2\leq i \leq 2m-1$. Put $v := g_1$ and $u:= g_{m+1}$. Applying \eqref{eq:type_2An1_cyclic_equation} for all $i \neq m+1,m+2$, we deduce $g_i = \varphi^{i-1}(v)$ for $1\leq i \leq m$; $g_i = \varpi^\kappa\varpi^{3m-i}(v) + \varpi^\kappa \sum_{j = 1}^{2m-i} \phi^{1+2(2m-i-j)}(a_j)\varphi^{3m-i-j}(v)$ for $m+2\leq i \leq 2m$. As in the previous cases, from this description we deduce the following. 

\begin{lm}\label{lm:lin_indep_type_2An1}
The set $\{\varphi^i(v)\}_{i=0}^{2m-2} \cup \{u\}$ is a basis of\, $V$.
\end{lm}

Inserting the above values for $g_i$ into \eqref{eq:type_2An1_cyclic_equation} for $i=m+1$ and $i=m+2$, we deduce
\begin{align}
\label{eq:type_2An1_minor_cyclic_equation}0 &= u - \varphi(u) - \varpi^\kappa \phi(a_m)\varphi^m(v), \\
\label{eq:type_2An1_cyclic_equation_2} \varphi^{2m-1}(v) &= v + \sum_{i=1}^{m-1}a_i \varphi^i(v) - \phi(a_{m-1}) \varphi^m(v) - \sum_{i=m+1}^{2m-2} \phi^{1+2(i-m)}(a_{2m-1-i}) \varphi^i(v) + a_m u.
\end{align}
If $a_m = 0$, Lemma~\ref{lm:lin_indep_type_2An1} shows that $v$ is a cyclic vector for the $(2m-1)$-dimensional isocrystal generated by it. Then \eqref{eq:type_2An1_cyclic_equation_2} and Lemma~\ref{lm:isocrystal_lemma} finish the proof by showing formulas \eqref{eq:show_valuations_bigger_sth_type_2An1}. 

We may now assume  $a_m \neq 0$. Dividing \eqref{eq:type_2An1_cyclic_equation_2} by $a_m$, we obtain an expression of $u$ as a linear combination of $\{\varphi^i(v)\}_{i=0}^{2m-1}$, with $\varphi^{2m-1}(v)$ appearing with non-zero coefficient. Then Lemma~\ref{lm:lin_indep_type_2An1} shows that $v$ is a cyclic vector for $V$.  Moreover, inserting this expression of $u$ into \eqref{eq:type_2An1_minor_cyclic_equation} and multiplying with $a_m$, we eliminate $u$ and obtain an expression of the form $\varphi^{2m}(v) = \sum_{i=0}^{2m-1}\gamma_i \varphi^i(v)$ with fixed $\gamma_i \in L$. By Lemma~\ref{lm:isocrystal_lemma} we have $\ord_\varpi(\gamma_i) \geq 0$ for all $i$. We then compute that $\gamma_{2m-1} = 1 - \frac{a_m}{\phi^2(a_m)} \phi^{1+2(m-1)}(a_1)$, $\gamma_i = \phi^{1+2(i-m)}(a_{2m-1-i}) - \frac{a_m}{\phi^2(a_m)} \phi^{1+ 2(i-m)}a_{2m-i}$ for $m+1 \leq i \leq 2m-2$ and $\gamma_m = -\varpi^\kappa a_m \phi(a_m) + \phi(a_{m-1}) - \frac{a_m}{\phi^2(a_m)} \phi^2(a_{m-1})$. Using induction, it follows from these equations that $\ord_\varpi(a_i) \geq 0$ for $1\leq i \leq m$. This shows \eqref{eq:show_valuations_bigger_sth_type_2An1}.

\subsubsection{}\label{sec:type_2An1_even_extra_comp}

Suppose that $n=2m$ is even and $b=\dot c_\kappa$ ($\kappa \in\{0,1\}$) as in Section~\ref{sec:2An1_computation}. We have the central cocharacter $\zeta = \sum_{i=1}^n \beta_i \in X_{\ast}(\bfZ)$, whose image in $X_{\ast}(\bfZ)_{\langle \sigma \rangle} \cong \bZ/2\bZ$ is non-trivial but which maps to $0 \in \pi_1(\bfG)_{\langle \sigma \rangle}$. To finish the proof in the ${}^2A_{n-1}$-case (\textit{cf.}~the discussion after Lemma~\ref{lm:lift_to_zext}), we have to show that for all $g \in X_{b\varpi^\zeta,b}(\ff)$, we have $g^{-1}\sigma_b(g) \in ({}^c\bfU \cap \bfU^-)_{\bfx_b}(\caO_L)\varpi^\zeta$. As in \eqref{eq:coxeter_almost_cyclic_equation_GSp}, we can write 
\begin{equation}\label{eq:cyclic_eq_type_2An1_even_extra} 
b\sigma(g) = gy\varpi^\zeta b
\end{equation}
with $y \in ({}^c\bfU \cap \bfU^-)(L)$ depending on $a_1,\dots,a_m \in L$ as described in Section~\ref{sec:type_2An1_roots_etc}. We thus have to show that \eqref{eq:show_valuations_bigger_sth_type_2An1} holds for these $a_i$. However, doing the same computation as in Section~\ref{sec:2An1_computation} (with $b^{\spl}$ and other notation as there), we see that 
\[
b^{\spl}\sigma_0^2(g) = (b\sigma)^2(g) = gy\varpi^{\zeta+\sigma(\zeta)} {}^b\sigma(y) b^{\spl}
\]
(where we used that $\varpi^\zeta$ and $\varpi^{\sigma(\zeta)}$ are central). As $\zeta + \sigma(\zeta) = 0$, we arrive simply at \eqref{eq:type_2An1_cyclic_equation_pre} and may proceed as there (and in Section~\ref{sec:type_2An1_even_comps}) to deduce \eqref{eq:show_valuations_bigger_sth_type_2An1}.

\subsubsection{Proof of Lemma~\ref{lm:conjugation_crosssecion_proof} for type $\boldsymbol{{}^2A_{n-1}}$ }\label{sec:proof_of_crosssec_lemma_type_2An1}
The action of $c$ on $\Phi^+$ is as follows: for $1\leq i < j \leq m$, $c(\alpha_{i-j}) = \alpha_{(i+1)-(j+1)}$; for $1 \leq i \leq m$ and $m+2\leq j \leq n$, $c(\alpha_{i-j}) = \alpha_{(i+1) - j}$ and $c(\alpha_{(m+1) - j}) = \alpha_{1-j}$; for $m+2 \leq i < j \leq n$, $c(\alpha_{i-j}) =\alpha_{i-j}$; for $1 \leq j \leq m$, $c(\alpha_{(m+1)-j}) \in \Phi^-$. The action of $\sigma$ on positive roots is given by $\sigma(\alpha_{i-j}) = \alpha_{(n+1-j) - (n+1-i)}$.

We set $r = m+2$; $\Psi_{m+2} = \{ \alpha_{(m+1)-j} \}_{j=1}^m$; $\Psi_{m+1} = \Psi_{m+2} \cup \{\alpha_{i-j} \colon 1 \leq i \leq m \text{ and } m+2\leq j \leq n \}$; for $m \geq i_0 \geq 2$, $\Psi_{i_0} = \Psi_{i_0+1} \cup \{\alpha_{i-i_0}\}_{i=1}^{i_0-1} \cup \{\alpha_{(n-i_0)-j}\}_{j=n-i_0+1}^n$; $\Psi_1 = \Phi^+$.

\subsection{Type $\boldsymbol{{}^2D_m}$}

\subsubsection{}\label{sec:type_2Dm_setting} 
Let $m \geq 4$, $V_0$, $Q$ be as in Section~\ref{sec:type_Dn_setup}. Consider the closed subgroup ${\bf GSO}(V_0,Q) \subseteq {\bf GL}(V_0)$, which was denoted by $\bfG$ in Section~\ref{sec:type_Dn_setup}. It is $k$-split, and its derived group is of type $D_m$. First, suppose $\charac k \neq 2$. Fix some $\Delta \in \caO_k^\times$ which is not a square in $k$, and fix a square root $\sqrt{\Delta} \in \caO_{\breve k}^\times$. Let $h_0 \in {\bf GL}_{2m}(V_0)(k)$ be given by $h_0(e_i) = e_i$ for all $i \neq \pm m$ and $h_0(e_{\pm m}) = (-2\Delta)^{\pm 1} e_{\mp m}$. Now suppose $\charac k = 2$. Then let $h_0 \in {\bf GL}_{2m}(V_0)(k)$ be given by $h_0(e_i) = e_i$ for all $i \neq \pm m$ and $h_0(e_{\pm m}) = e_{\mp m}$. In this section we let $\bfG$ be the $k$-group ${\bf GSO}(V_0,Q)_{h_0}$ in the sense of Section~\ref{sec:twisting_Frobenius}. In particular, for any $\bF_q$-algebra $R_0$ with $\widetilde R = \bW(R_0 \otimes_{\bF_q} \obF)[1/\varpi]$, the geometric Frobenius on $\bfG(\widetilde R) = {\bf GSO}(V_0,Q)(\widetilde R) \subseteq {\bf GL}(V_0)(\widetilde R)$ is given by 
\[
g \longmapsto \sigma(g) = h_0 \sigma_0(g)h_0^{-1},
\]
where $\sigma_0$ is the Frobenius on ${\bf GL}(V_0)(\widetilde R)$ corresponding to the natural $k$-structure on ${\bf GL}(V_0)$. 

\begin{lm}\label{lm:type_2Dm_realization}
The derived group of the $k$-group $\bfG$ is of type ${}^2D_m$.
\end{lm}

\begin{proof}
First suppose  $\charac k \neq 2$. Let $Q' \colon V_0 \rar k$ be the quadratic form $Q'\left(\sum_{i=1}^m (a_i e_i + a_{-i}e_{-i})\right) = \sum_{i=1}^{m-1}a_ia_{-i} + a_m^2 - a_{-m}^2 \Delta$. It is non-split of Witt index $m-1$; hence ${\bf SO}(V_0,Q')$ is of type ${}^2D_m$. It suffices to construct a $k$-isomorphism ${\bf GSO}(V_0,Q') \cong \bfG$. Let $h_1 \in {\bf GL}(V_0)(\breve k)$ be the element given by $h_1(e_{\pm i}) = e_{\pm i}$ for $1\leq i\leq m-1$ and $h_1(e_m) = \frac{1}{2\sqrt{\Delta}} e_m + \frac{1}{2}e_{-m}$, $h_1(e_{-m}) = \sqrt{\Delta}e_m - \Delta e_{-m}$. Then $h_1^T\sigma_0(h_1)^{-1,T} = h_0$. If $Q$ (resp.\ $Q'$) denotes the matrix representing $Q$ (resp.\ $Q'$) with respect to the fixed basis of $V_0$ (so that $Q(v) = v^TQv$ for all $v\in V$), one immediately checks that $h_1Qh_1^T = Q'$. This means that $v \mapsto h_1^T v \colon (V,Q') \rar (V,Q)$ is a $\breve k$-isometry (\textit{i.e.}, $Q'(v) = Q(h_1^Tv)$) and that $\Int(h_1^{T,-1}) \colon {\bf GL}(V_0)_{\breve k} \rar {\bf GL}(V_0)_{\breve k}$ restricts to a $\breve k$-rational isomorphism $\Int(h_1^{T,-1}) \colon {\bf GSO}(V_0, Q)_{\breve k} \rar  {\bf GSO}(V_0,Q')_{\breve k}$. (Indeed, if $g \in {\bf GSO}(V_0, Q)$,  then $g^tQg = Q$ and one computes using $h_1 Q h_1^T = Q'$ that $\Int(h_1^{T,-1})(g)^T \, Q' \,\Int(h_1^{T,-1})(g) = Q'$.) For $\widetilde R$ as above, this isomorphism transfers the geometric Frobenius $\sigma_0$ of ${\bf GSO}(V_0,Q')(\widetilde R)$ to the Frobenius $\sigma$ given by
\[
\sigma(g) = \Int(h_1^{T,-1})\left(\sigma_0(\Int(h_1^{T,-1})(g))\right) = \Int(h_0)\left(\sigma_0(g)\right). 
\] 
Now suppose  $\charac k = 2$. Let $Q' \colon V_0 \rar k$ be the quadratic form $Q'\left(\sum_{i=1}^m (a_i e_i + a_{-i}e_{-i})\right) = \sum_{i=1}^{m-1}a_ia_{-i} + a_m^2 + a_m a_{-m} + \Delta a_{-m}^2$, where $\Delta \in \caO_k^\times$ is such that $X^2 + X + \Delta \in k[X]$ is irreducible. It is non-split of Witt index $m-1$; hence ${\bf SO}(V_0,Q')$ is of type ${}^2D_m$. Let $\lambda \in \caO_{\breve k}^\times$ be a root of $X^2 + X + \Delta$. Let $h_1 \in {\bf GL}(V_0)(\breve k)$ be given by $h_1(e_i) = e_{i}$ for $i\neq \pm m$, $h_1(e_m) = e_m + \lambda e_{-m}$, $h_1(e_{-m}) = e_m + \lambda^{-1}\Delta e_{-m}$. Then $h_1$ induces a $\breve k$-isometry between $(V,Q')$ and $(V,Q)$. We may conclude as above, computing that $h_1^T\sigma_0(h_1)^{-1,T} = h_0$. \qedhere
\end{proof}

\subsubsection{}\label{sec:type_2Dm_further_setting} Via the identification $\bfG \otimes_k \breve k \cong {\bf GSO}(V_0,Q)_{\breve k}$, we can use $\bfT,\bfB$, $\varepsilon_i \in X_\ast(\bfT)$ ($0\leq i \leq m$) from Section~\ref{sec:type_D_further_setting}. The induced action of $\sigma$ on $X_\ast(\bfT)$ is given by $\sigma(\varepsilon_i) = \varepsilon_i$ for $1\leq i \leq m-1$, $\sigma(\varepsilon_m) = -\varepsilon_m$, $\sigma(\varepsilon_0) = \varepsilon_0 - \varepsilon_m$. It follows that $\pi_1(\bfG^{\ad})_{\langle \sigma \rangle} \cong \bZ/2\bZ$, generated by the class of $\varepsilon_0$.

We have isomorphisms $\bZ \oplus \bZ/2\bZ \isor \pi_1(\bfG)$, $(1,0) \mapsto \varepsilon_0$, $(0,1) \mapsto \bar\varepsilon_1$ and $\bZ \isor X_\ast(\bfZ)$, $1 \mapsto 2 \varepsilon_0-\sum_{i=1}^m \varepsilon_i$. The inclusion $X_\ast(\bfZ) \rar \pi_1(\bfG)$ corresponds to the map $1\mapsto (2,1)$ if $m$ is odd and  to $1 \mapsto (2,0)$ if $m$ is even. Moreover, $\sigma$ acts on $X_\ast(\bfZ)$ as the identity, and the action on $\pi_1(\bfG)$ is given by $(a,b) \mapsto (0,(a \, \modd \, 2) + b)$. Justifying Remark~\ref{rem:trivial_kernel_exception} in this case, one easily deduces that the induced map $X_{\ast}(\bfZ)_{\langle \sigma\rangle} \cong \bZ \rar \bZ \cong \pi_1(\bfG)_{\langle \sigma\rangle}$ is multiplication by $2$, hence injective.

We have the roots $\alpha_{\pm i \pm j}$ of $\bfT$ in $\bfG$ as in Section~\ref{sec:type_D_further_setting}. The $\sigma$-action on the simple roots fixes all $\alpha_{i-(i+1)}$ ($1\leq i \leq m-2$) and interchanges $\alpha_{(m-1)-m}$ and $\alpha_{(m-1)+m}$. Thus we obtain a Coxeter element
\begin{equation}\label{eq:special_Coxeter_element_type_2D}
c = s_{\alpha_{1-2}} \cdot \dots \cdot s_{\alpha_{(m-1)-m}} = (1,2,\dots,m-1,m)(2m,2m-1,\dots,m+2,m+1). 
\end{equation}
We consider two lifts $\dot c_0$, $\dot c_1$ of $c$ to $\bfG(\breve k)$. For $\kappa \in \{0,1\}$, $\dot c_\kappa$ is given by $\dot c_\kappa(e_i) = \varpi^\kappa e_{i+1}$ for $1\leq i\leq m-1$,; $\dot c_\kappa(e_m) = \varpi^\kappa e_1$; $\dot c_\kappa(e_{-i}) = e_{-(i+1)}$ for $1\leq i \leq m-1$; $\dot c_\kappa(e_{-m}) = e_{-1}$. Then $\kappa_\bfG(\dot c_0) = 0$, $\kappa_\bfG(\dot c_1) = \bar\varepsilon_0$ in $\pi_1(\bfG)_{\langle \sigma \rangle}$. Thus the $\bfG_{\dot c_j}$ ($\kappa\in \{0,1\}$) are the two inner forms of $\bfG$. 

We use the $\caA_{\bfT^{\ad},\breve k}$ as in Section~\ref{sec:type_D_further_setting}. The group ${}^c\bfU \cap \bfU^-$ is generated by root subgroups attached to $\alpha_{i-1}$ ($2\leq i\leq m$). Explicitly, any element $y \in ({}^c\bfU \cap \bfU^-)(R)$ is an $R$-linear map in ${\bf GL}(V_0)(R)$ given by $e_1 \mapsto e_1 + \sum_{i=2}^m a_{i-1} e_i$; $e_i \mapsto e_i$ for $2\leq i \leq m$ and $i = -1$; $e_{-i} \mapsto e_{-i} - a_{i-1} e_{-1}$ for all $2\leq i \leq m$.

\subsubsection{}
Let $b = \dot c_\kappa$ with $\kappa \in \{0,1\}$. We compute $\bfx_b = 0$ if $\kappa = 0$ and ${\bf x}_b = \sum_{i=1}^m (-\frac{m}{4} + \frac{i}{2}) \varepsilon_i$ if $\kappa = 1$. Let $g \in \dot X_{b,b}(\ff)$. Then \eqref{eq:coxeter_almost_cyclic_equation_GSp} holds, with $n$ as described in Section~\ref{sec:type_2Dm_further_setting}, depending on some $a_1,\dots,a_{m-1} \in L$. We have to show that $n \in \caG_{\bfx_b}(\caO_L)$, which is (as in Section~\ref{sec:computation_case_An}) equivalent to
\begin{equation}\label{eq:to_show_type_2Dm}
\ord\nolimits_\varpi(a_i) \geq \frac{\kappa i}{2} \quad \text{for $1\leq i\leq m-1$.}
\end{equation}

Let $V = V_0 \otimes_k L$. Let $\sigma_0$ be the $\phi$-linear automorphism $x \otimes \lambda \mapsto x \otimes \phi(\lambda)$ of $V$.  Consider the $\phi$-linear automorphism $\varphi = bh_0\sigma_0$ of $V$. Then $(V,\varphi)$ is an isocrystal over $\ff$ of slope $\frac{\kappa}{2}$.
Equation \eqref{eq:coxeter_almost_cyclic_equation_GSp} gives $bh_0 \sigma_0(g) = gybh_0$. Applying this to $e_i$ for $1 \leq \pm i \leq m$, we get $\varphi(g_i) = gybh_0(e_i)$, where we write $g_i = g(e_i)$ (and where we use $\sigma_0(g)(e_i) = \sigma_0(g_i)$). Writing $v := g_1$, we deduce $g_i = \varpi^{1-i}\varphi^{i-1}(v)$ and $g_{-i} = -(2\Delta)^{-1} \varpi^{\kappa(1-m)}\left( \varphi^{m-1+i}(v) + \sum_{j=1}^{i-1} \phi^{i-j-1}(a_j)\varphi^{m+i-j-1}(v)\right)$ for $1\leq i \leq m$. Inserting this into the equation attached to $\varphi(g_{m+1}) = gybh_0(e_{m+1})$ gives 
\[
\varphi^{2m}(v) = \varpi^{\kappa m} v + \sum_{i=1}^{m-1} \varpi^{\kappa(m-i)} a_i \varphi^i(v) - \sum_{i=m+1}^{2m-1} \phi^{i-m}(a_{2m-i}) \varphi^i(v).
\]
It is easy to see that $v$ is cyclic for $V$. Using Lemma~\ref{lm:isocrystal_lemma}, \eqref{eq:to_show_type_2Dm} follows.

\subsubsection{Proof of Lemma~\ref{lm:conjugation_crosssecion_proof} for type $\boldsymbol{{}^2D_m}$ }\label{sec:proof_of_crosssec_lemma_type_2Dm}

The action of $c$ on $\Phi^+$ is as follows: for $1\leq i< j<m$, $c(\alpha_{i \pm j}) = \alpha_{(i+1)\pm (j+1)}$; for $1 \leq i \leq m-1$, $c(\alpha_{i+m}) = \alpha_{1 + (i+1)}$; for $1\leq i\leq m-1$, $c(\alpha_{i-m}) \in \Phi^-$. Moreover, $\sigma(\alpha_{i\pm j}) = \alpha_{i \pm j}$ if $j \neq m$ and $\sigma(\alpha_{i\pm m}) = \alpha_{i\mp m}$.

Set $r = m+1$; $\Psi_{m+1} = \{\alpha_{i-m}\}_{i=1}^{m-1}$; $\Psi_m = \Psi_{m+1} \cup \{\alpha_{i+j} \colon 1 \leq i < j < m \}$; $\Psi_{m-1} = \Psi_m \cup \{\alpha_{i+m}\}_{i=1}^{m-1}$; for $1 \leq i_0 \leq m-2$, $\Psi_{i_0} = \Psi_{i_0+1} \cup \{\alpha_{i-(i_0+1)} \colon 1\leq i < i_0 + 1 \}$.


\vspace{-\baselineskip}

\newcommand{\etalchar}[1]{$^{#1}$}

\end{document}